\documentclass{ios-book-article}

\usepackage{epsfig}
\usepackage{rotating}
\usepackage[english]{babel}
\usepackage{morefloats}
\usepackage{amsmath}
\usepackage{amsfonts}
\usepackage{amssymb}
\usepackage{amsthm}
\usepackage{graphicx}
\usepackage{newlfont}
\usepackage{textcomp}
\usepackage{latexsym}
\usepackage[all,cmtip]{xy}
\usepackage{layout}
\usepackage[usenames,dvipsnames]{color}
\usepackage{cite}
\usepackage{epstopdf}
\usepackage[overload]{textcase}      
\usepackage{epsfig} 
\usepackage{epic}
\usepackage{rotating}

\newtheorem{thm}{Theorem}

\newtheorem {lem}{Lemma}

\newtheorem {exa}{Example}

\newtheorem{alg}{Algorithm}
\newtheorem {rem}{Remark}
\newtheorem {cor}{Corollary}

\usepackage{xy} \input xy \xyoption{all}

\newcommand{\ch}[2]
{\begin{bmatrix}
 #1 \\
 #2\\
\end{bmatrix}}

\newcommand{\chr}[4]
{\begin{bmatrix}
 #1 & #2\\
 #3 & #4
\end{bmatrix}}

\newcommand{\chs}[6]
{\begin{bmatrix}
 #1 & #2 & #3\\
 #4 & #5 & #6
\end{bmatrix}}

\def\N{\mathbb N}
\def\Z{\mathbb Z}
\def\Q{\mathbb Q}
\def\R{\mathbb R}
\def\C{\mathbb C}
\def\P{\mathbb P}

\def\H{\mathcal H}
\def\M{{\mathcal M}}
\def\A{\mathcal A}
\def\X{\mathcal X}

\def\L{\mathcal L}

\def\m{\mathfrak m}
\def\an{\mathfrak a}
\def\bn{\mathfrak b}
\def\en{\mathfrak e}
\def\hn{\mathfrak h}
\def\n{\mathfrak n}
\def\HS{\mathfrak H}

\def\det{\mbox{det }}

\def\J{\mbox{Jac }}
\def\bAut{\mbox{$\overline{Aut}$ }   }

\def\<{\langle}
\def\>{\rangle}
\def\normal{\triangleleft }

\def\a{\alpha}
\def\b{\beta}
\def\t{\tau}
\def\s{\sigma}

\def\T{\theta}
\def\d{{\delta }}

\def\O{\Omega}
\def\om{\omega}

\def\e{\eta}
\def\G{\Gamma}
\def\l{\lambda}
\def\ep{\epsilon}
\def\xs{{\Bbb o}}

\def\iso{{\, \cong\, }}

\def\sem{\rtimes}

\def\embd{\hookrightarrow}

\def\Aut{\mbox{Aut}}

\begin{document}
\begin{frontmatter}          
%
\title{Theta functions and algebraic curves with automorphisms}
\runningtitle{}

\author{ \fnms{T. Shaska \and G.S. Wijesiri }   \snm{} }

\maketitle

\begin{abstract}
Let $\X$ be an irreducible, smooth, projective curve of genus $g \geq 2$ defined over the complex field $\C.$ Then
there is a covering $\pi: \X \longrightarrow \P^1,$ where $\P^1$ denotes the projective line. The problem of expressing
branch points of the covering $\pi$ in terms of the transcendentals (period matrix, thetanulls, e.g.) is classical. It
goes back to Riemann, Jacobi, Picard and Rosenhein. Many mathematicians, including Picard and Thomae,
have offered partial treatments for this problem. In this work, we address the problem for cyclic curves of genus 2, 3,
and 4 and find relations among theta functions for curves with automorphisms. We consider curves of genus $g
> 1$ admitting an automorphism $\sigma$ such that $\X^\sigma$ has genus zero and $\sigma$ generates a normal subgroup of
the automorphism group $Aut(\X)$ of $\X$.

To characterize the locus of cyclic curves by analytic conditions on its Abelian coordinates, in other words, theta
functions, we use  some classical formulas, recent results of Hurwitz spaces, and symbolic computations, especially for
genera 2 and 3. For hyperelliptic curves, we use Thomae's formula to invert the period map and discover relations among
the classical thetanulls of cyclic curves. For non hyperelliptic curves, we write the equations in terms of thetanulls.

Fast genus 2 curve arithmetic in the Jacobian of the curve is used in cryptography and is based on inverting the moduli
map for genus 2 curves and on some other relations on theta functions.  We determine similar formulas and relations for
genus 3 hyperelliptic curves and offer an algorithm for how this can be done for higher genus curves. It is still to be
determined whether our formulas for $g=3$ can be used in cryptographic applications as in $g=2.$
\end{abstract}
\end{frontmatter}


\section{Introduction to Theta Functions of Curves}

Let $\X$ be an irreducible, smooth, projective curve of genus $g \geq 2$ defined over the complex field $\C.$ We denote
the moduli space of genus $g$ by $\M_g$ and  the hyperelliptic locus in $\M_g$ by $\H_g.$ It is well known that dim
$\M_g = 3g-3$ and $\H_g$ is a $(2g-1)$ dimensional subvariety of $\M_g.$  Choose a symplectic homology basis for $\X$,
say
\[ \{ A_1, \dots, A_g, B_1, \dots , B_g\}\]
such that the intersection products $A_i \cdot A_j = B_i \cdot B_j =0$ and $A_i \cdot B_j= \d_{i j}$.
We choose a basis $\{ w_i\}$ for the space of holomorphic 1-forms such that $\int_{A_i} w_j = \d_{i j},$ where $\d_{i
j}$ is the Kronecker delta. The matrix $\O= \left[ \int_{B_i} w_j \right] $ is  the \textbf{period matrix} of $\X$ .
The columns of the matrix $\left[ I \ | \O \right]$ form a lattice $L$ in  $\C^g$ and the Jacobian  of $\X$ is $\J(\X)
= \C^g/ L$. Let $$\HS_g =\{\t : \t \,\, \textit{is symmetric}\,\, g \times g \, \textit{matrix with positive definite
imaginary part} \}$$ be the \textbf{Siegel upper-half space}. Then $\O \in \HS_g$. The group of all $2g \times 2g$
matrices $M \in GL_{2g}(\Z)$ satisfying
\[M^t J M = J  \,\,\,\,\,\,\,\, \textit{with} \, \,\,\,\,\,\, J = \begin{pmatrix}  0 & I_g \\ -I_g & 0 \end{pmatrix} \]
is called the \textbf{symplectic group} and denoted  by $Sp_{2g}(\Z)$.
Let $M = \begin{pmatrix} R & S \\ T & U
\end{pmatrix} \in Sp_{2g}(\Z) $ and $\t \in \HS_g$ where $R,$ $S,$ $T$ and $U$ are $g \times g$ matrices. $Sp_{2g}(\Z) $ acts transitively on $\HS_g$ as

\[ M(\t) = (R \t + S)(T \t + U)^{-1}. \]

\noindent Here, the multiplications are matrix multiplications. There is an injection
\[ \M_g \embd \HS_g/ Sp_{2g}(\Z) =: \A_g \] where each curve $C$ (up to isomorphism) goes to its Jacobian in $\A_g.$
If $\ell$ is a positive integer, the principal congruence group of degree $g$ and of level $\ell$ is defined as a
subgroup of $Sp_{2g}(\Z)$ by the condition $M \equiv I_{2g} \mod \ell.$ We shall denote this group by
$Sp_{2g}(\Z)(\ell).$

For any $z \in \C^g$ and $\t \in \HS_g$ the \textbf{Riemann's theta function} is defined as
\[ \T (z , \t) = \sum_{u\in \Z^g} e^{\pi i ( u^t \t u + 2 u^t z )  }\]
where $u$ and $z$ are $g$-dimensional column vectors and the products involved in the formula are matrix products. The
fact that the imaginary part of $\t$ is positive makes the series absolutely convergent over every compact subset of
$\C^g \times \HS_g$.
The theta function is holomorphic on $\C^g\times \HS_g$ and has quasi periodic properties,
$$\T(z+u,\tau)=\T(z,\tau)\quad \textit{and}\quad
\T(z+u\tau,\tau)=e^{-\pi i( u^t \tau u+2z^t u )}\cdot  \T(z,\tau)$$
where $u\in \Z^g$; see \cite{Mu1} for details.
The locus $ \Theta: = \{ z \in \C^g/L : \T(z, \O)=0 \}$ is called the \textbf{theta divisor} of $\X$.
Any point $e \in \J (\X)$ can be uniquely written  as $e = (b,a) \begin{pmatrix} 1_g \\ \O \end{pmatrix}$ where $a,b
\in \R^g$ are the characteristics of $e.$
We shall use  the notation $[e]$ for the characteristic of $e$ where $[e] = \ch{a}{b}.$ For any $a, b \in \Q^g$, the
theta function with rational characteristics is defined as a translate of Riemann's theta function multiplied by an
exponential factor

\begin{equation} \label{ThetaFunctionWithCharac} \T  \ch{a}{b} (z , \t) = e^{\pi i( a^t \t a + 2 a^t(z+b))} \T(z+\t a+b ,\t).\end{equation}

\noindent By writing out Eq.~\eqref{ThetaFunctionWithCharac}, we have
\[ \T  \ch{a}{b} (z , \t) = \sum_{u\in \Z^g} e^{\pi i ( (u+a)^t \t (u+a) + 2 (u+a)^t (z+b) )  }. \]
The Riemann's theta function is $\T \ch{0}{0}.$ The theta function with rational characteristics has  the following
properties:
\begin{equation}\label{periodicproperty}
\begin{split}
& \T \ch{a+n} {b+m} (z,\t) = e^{2\pi i a^t m}\T \ch {a} {b} (z,\t),\\
&\T \ch{a} {b} (z+m,\t) = e^{2\pi i a^t m}\T \ch {a} {b} (z,\t),\\
&\T \ch{a} {b} (z+\t m,\t) = e^{\pi i (-2b^t m -m^t \t m - 2m^t z)}\T \ch {a} {b} (z,\t)\\
\end{split}
\end{equation}
where $n,m \in \Z^n.$ All of these properties are immediately verified by writing them out.
 A scalar obtained by evaluating a theta function with characteristic at $z=0$ is called a \emph{theta constant} or
\emph{thetanulls}. When the entries of column vectors $a$ and $b$ are from the set $\{ 0,\frac{1}{2}\}$, then the
characteristics $ \ch {a}{b} $ are called the \emph{half-integer characteristics}. The corresponding theta functions
with rational characteristics are called \emph{theta characteristics}.
Points of order $n$ on $\J(\X)$ are called the $\frac 1 n$-\textbf{periods}. Any point $p$ of $\J(\X)$ can be written
as $p = \t \,a + b. $ If $\ch{a}{b}$ is a $\frac 1 n$-period, then $a,b \in (\frac{1}{n}\Z /\Z)^{g}.$ The $\frac 1
n$-period $p$ can be associated with an element of $H_1(\X,\Z / n\Z)$ as follows: Let $a = (a_1,\cdots,a_g)^t,$ and $b
= (b_1,\cdots,b_g)^t.$ Then
\[
\begin{split}
p & = \t a + b \\
           & = \big(\sum a_i \int_{B_i} \om_1,\cdots, \sum a_i \int_{B_i} \om_g \big)^t + \big(b_1 \int_{A_1} \om_1,\cdots,b_g
           \int_{A_g} \om_g \big) \\
\end{split}
\]
\[
\begin{split}
           & = \big(\sum (a_i \int_{B_i} \om_1 + b_i\int_{A_i} \om_1) ,\cdots, \sum (a_i \int_{B_i} \om_g + b_i\int_{A_i} \om_g) \big)^t\\
           & = \big( \int_C \om_1, \cdots, \int_C \om_g \big)^t
\end{split}
\]
where $C = \sum a_i B_i + b_i A_i. $ We identify the point $p$ with the cycle $\bar{C} \in H_1(\X,\Z / n\Z)$ where
$\bar{C} =\sum \bar{a_i} B_i + \bar{b_i} A_i,$  $\bar{a_i} = n a_i$ and $\bar{b_i} = n b_i$ for all $i.$
%

\subsection{Half-Integer Characteristics and the G\"opel Group} In this section we study groups of
half-integer characteristics. Any half-integer characteristic $\m \in\frac{1}{2}\Z^{2g}/\Z^{2g}$ is given by
\[
\m = \frac{1}{2}m = \frac{1}{2}
\begin{pmatrix} m_1 & m_2 &  \cdots &  m_g \\ m_1^{\prime} & m_2^{\prime} & \cdots & m_g^{\prime}  \end{pmatrix},
\]
where $m_i, m_i^{\prime} \in \Z.$ For $\m = \ch{m ^\prime}{m^{\prime \prime}} \in \frac{1}{2}\Z^{2g}/\Z^{2g},$ we
define $e_*(\m) = (-1)^{4 (m^\prime)^t m^{\prime \prime}}.$ We say that $\m$ is an \emph{even} (resp. \emph{odd})
characteristic if $e_*(\m) = 1$ (resp. $e_*(\m) = -1$). For any curve of genus $g$, there are $2^{g-1}(2^g+1)$ (resp.,
$2^{g-1}(2^g-1)$ ) even theta functions (resp., odd theta functions). Let $\an$ be another half-integer characteristic.
We define

\[
 \m \, \an = \frac{1}{2} \begin{pmatrix} t_1 & t_2 &  \cdots &  t_g \\ t_1^{\prime} & t_2^{\prime} & \cdots &
t_g^{\prime}
\end{pmatrix}
\]
where $t_i \equiv (m_i\, + a_i)  \mod 2$ and $t_i^{\prime} \equiv (m_i^{\prime}\, + a_i^{\prime} ) \mod 2.$

For the rest of the thesis we only consider  characteristics $\frac{1}{2}q$ in which each of the elements
$q_i,q_i^{\prime}$ is either 0 or 1. We use the following abbreviations:
\[
\begin{split}
&|\m| = \sum_{i=1}^g m_i m_i^{\prime},  \quad \quad \quad \quad \quad \quad \quad \quad \quad
|\m, \an| = \sum_{i=1}^g (m_i^{\prime} a_i - m_i a_i^{\prime}), \\
& |\m, \an, \bn| = |\an, \bn| + |\bn, \m| + |\m, \an|, \quad \quad {\m\choose \an} = e^{\pi i \sum_{j=1}^g m_j
a_j^{\prime}}.
\end{split}
\]
\indent The set of all half-integer characteristics forms a group $\G$ which has $2^{2g}$ elements. We say that two
half integer characteristics $\m$ and $\an$ are \emph{syzygetic} (resp., \emph{azygetic}) if $|\m, \an| \equiv 0 \mod
2$ (resp., $|\m, \an| \equiv 1 \mod 2$) and three half-integer characteristics $\m, \an$, and $\bn$ are syzygetic if
$|\m, \an, \bn| \equiv 0 \mod 2$.
A \emph{G\"opel group} $G$ is a group of $2^r$ half-integer characteristics where $r \leq g$ such that every two
characteristics are syzygetic. The elements of the group $G$ are formed by the sums of $r$ fundamental characteristics;
see \cite[pg. 489]{Baker} for details. Obviously, a G\"opel group of order $2^r$ is isomorphic to $C^r_2$. The proof of
the following lemma can be found on   \cite[pg.  490]{Baker}.
\begin{lem}The number of different G\"opel groups which have $2^r$ characteristics is
\[
\frac{(2^{2g}-1)(2^{2g-2}-1)\cdots(2^{2g-2r+2}-1)}{(2^r-1)(2^{r-1}-1)\cdots(2-1)}.
\]
\end{lem}
If $G$ is a G\"opel group with $2^r$ elements, it has $2^{2g-r}$ cosets. The cosets are called \emph{G\"opel systems}
and are denoted by $\an G$, $\an \in \G$. Any three characteristics of a G\"opel system are syzygetic. We can find a
set of characteristics called a basis of the G\"opel system which derives all its $2^r$ characteristics by taking only
combinations of any odd number of characteristics of the basis.
\begin{lem}
Let $g \geq 1$ be a fixed integer, $r$ be as defined above and $\sigma = g-r.$ Then there are
$2^{\sigma-1}(2^\sigma+1)$ G\"opel systems which only consist of even characteristics and there are
$2^{\sigma-1}(2^\sigma-1)$ G\"opel systems which consist of odd characteristics. The other $2^{2\sigma}(2^r-1)$ G\"opel
systems consist of as many odd characteristics as even characteristics.
\end{lem}
\proof The proof can be found on \cite[pg. 492]{Baker}. \qed
\begin{cor}\label{numb_systems}
When $r=g,$ we have only one (resp., 0) G\"opel system which consists of even (resp., odd) characteristics.
\end{cor}
Let us consider $s=2^{2\sigma}$ G\"opel systems which have  distinct characters. Let us denote them by
\[\an_1 G,\an_2 G,\cdots,\an_s G.\] We have the following lemma.
\begin{lem}
It is possible to choose $2\sigma+1$ characteristics from $\an_1, \an_2,\cdots, \an_s,$  say $\bar{\an}_1,$
$\bar{\an}_2,$ $\cdots,$ $\bar{\an}_{2\sigma+1}$, such that every three of them are azygetic and all have the same
character. The above $2\sigma+1$ fundamental characteristics are even (resp., odd) if $\sigma \equiv 1,0 \mod 4$
(resp.,$\equiv 2,3 \mod 4$).
\end{lem}
\noindent The proof of the following lemma can be found on \cite[pg. 511]{Baker}.
\begin{lem}
For any half-integer characteristics $\an$ and $\hn,$ we have the following:
\begin{equation}\label {Bakereq1}
\T^2[\an](z_1,\t) \T^2[\an \hn](z_2,\t) = \frac{1}{2^{g}} \sum_\en  e^{\pi i |\an \en|} { \hn \choose \an \en}
\T^2[\en](z_1,\t)\T^2[\en \hn](z_2,\t).
\end{equation}
\end{lem}

We can use this relation to get identities among half-integer theta constants. Here $\en$ can be any half-integer
characteristic. We know that we have $2^{g-1}(2^g+1)$ even characteristics. As the genus increases, we have multiple
choices for $\en.$ In the following, we explain how we reduce the number of possibilities for $\en$ and how to get
identities among theta constants.

First we replace $\en$ by $\en \hn$ and $z_1=z_2= 0$ in Eq.~\eqref{Bakereq1}. Eq.~\eqref{Bakereq1} can then be written
as follows:
\begin{equation}\label {Bakereq2}
\T^2[\an] \T^2[\an \hn] = 2^{-g} \sum_\en  e^{\pi i |\an \en \hn|} { \hn \choose \an \en \hn} \T^2[\en] \T^2[\en \hn].
\end{equation}
We have $e^{\pi i |\an \en \hn|}{ \hn \choose \an \en \hn} = e^{\pi i |\an \en|}{ \hn \choose \an \en} e^{\pi i |\an
\en, \hn|}.$ Next we put $z_1=z_2= 0$ in Eq.~\eqref{Bakereq1} and add it to Eq.~\eqref{Bakereq2} and get the following
identity:
\begin{equation}\label {Bakereq3}
2\T^2[\an] \T^2[\an \hn] = 2^{-g} \sum_\en  e^{\pi i |\an \en|} (1 + e^{\pi i|\an \en, \hn|}) \T^2[\en] \T^2[\en \hn].
\end{equation}
If $|\an \en, \hn| \equiv  1 \mod 2$, the corresponding terms in the summation vanish. Otherwise $1 + e^{\pi i|\an \en,
\hn|} = 2.$ In this case, if either $\en$ is odd or $\en \hn$ is odd, the corresponding terms in the summation vanish
again. Therefore, we need $|\an \en, \hn| \equiv 0 \mod 2$ and $|\en| \equiv |\en \hn| \equiv 0 \mod 2,$ in order to
get nonzero terms in the summation. If $\en^*$ satisfies $|\en^*| \equiv |\en^* \hn^*| \equiv 0 \mod 2$ for some
$\hn^*,$ then $\en^*\hn^*$ is also a candidate for the left hand side of the summation. Only one of such two values
$\en^*$ and $\en^* \hn^*$ is taken. As  a result, we have the following identity among theta constants
\begin{equation}\label {eq1}
\T^2[\an] \T^2[\an \hn] = \frac{1}{2^{g-1}} \sum_\en  e^{\pi i |\an \en|} { \hn \choose \an \en} \T^2[\en]\T^2[\en
\hn],
\end{equation}
where $\an, \hn$ are any characteristics and $\en$ is a characteristics such that $|\an \en, \hn| \equiv 0 \mod 2,$
$|\en| \equiv |\en \hn| \equiv 0 \mod 2$ and $\en \neq \en \hn.$

By starting from the Eq.~\eqref{Bakereq1} with $z_1 = z_2$ and following a similar argument to the one above, we can
derive the identity,
\begin{equation}\label{eq2}
\T^4[\an] + e^{\pi i |\an, \hn|} \T^4[\an \hn] = \frac{1}{2^{g-1}} \sum_\en  e^{\pi i |\an \en|} \{ \T^4[\en] + e^{ \pi
i |\an, \hn|} \T^4[\en \hn]\}
\end{equation}
where $\an, \hn$ are any characteristics and $\en$ is a characteristic such that $|\hn| + |\en, \hn| \equiv 0 \mod 2,$
$|\en| \equiv |\en \hn| \equiv 0 \mod 2$ and $\en \neq \en \hn.$

\begin{rem}
$|\an \en ,\hn| \equiv 0 \mod 2$ and $|\en \hn| \equiv |\en| \equiv 0 \mod 2$ implies $|\an, \hn| + |\hn| \equiv 0 \mod
2.$
\end{rem}
We use Eq.~\eqref{eq1} and Eq.~\eqref{eq2} to get identities among thetanulls in Chapter 2 and in Chapter 3.
\subsection{Hyperelliptic Curves and Their Theta Functions}
A hyperelliptic curve $\X,$ defined over $\C,$ is a cover of order two of the projective line $\P^1.$ Let $z$ be the
generator (the hyperelliptic involution) of the Galois group $Gal(\X / \P^1).$ It is known that $\langle z \rangle $ is
a normal subgroup of the $\Aut(\X)$ and $z$ is in the center of $\Aut(\X)$. A hyperelliptic curve is ramified in $(2g
+2)$ places $w_1, \cdots, w_{2g+2}.$ This sets up a bijection between isomorphism classes of hyperelliptic genus g
curves and unordered distinct (2g+2)-tuples $w_1,\cdots,w_{2g+2} \in \P^1$ modulo automorphisms of $\P^1.$ An unordered
$(2g+2)$-tuple $\{w_i\}_{i=1}^{2g+2}$ can be described by a binary form (i.e. a homogenous equation $f(X,Z)$ of degree
$2g+2$). To describe $\H_g,$ we need rational functions of the coefficients of a binary form $f(X,Z),$ invariant under
linear substitutions in X and Z. Such functions are called absolute invariants for $g = 2$; see \cite{S7} for their
definitions. The absolute invariants are $GL_2(\C)$ invariants under the natural action of $GL_2(\C)$ on the space of
binary forms of degree $2g + 2.$ Two genus $g$ hyperelliptic curves are isomorphic if and only if they have the same
absolute invariants. The locus of genus $g$ hyperelliptic curves with an extra involution is an irreducible
$g$-dimensional subvariety of $\H_g$ which is denoted by $\L_g.$ Finding an explicit description of $\L_g$ means
finding explicit equations in terms of absolute invariants. Such equations are computed only for $g = 2;$ see \cite{S7}
for details. Writing the equations of $\L_2$ in terms of theta constants is the main focus of Chapter 2. Computing
similar equations for $g \geq 3$ requires first finding the corresponding absolute invariants. This is still an open
problem in classical invariant theory even for $g = 3.$

 Let $\X \longrightarrow \P^1$ be the degree 2 hyperelliptic projection.
We can assume that $\infty$ is a branch point.

Let
\[ B := \{\a_1,\a_2, \cdots ,\a_{2g+1} \}\]
be the set of other branch points.  Let $S = \{1,2, \cdots, 2g+1\}$ be the index set of $B$ and $\e : S \longrightarrow
\frac{1}{2}\Z^{2g}/\Z^{2g}$  be a map defined as follows:
%
\[
\begin{split}
\e(2i-1) & = \begin{bmatrix}
              0 & \cdots & 0 & \frac{1}{2} & 0 & \cdots & 0\\
              \frac{1}{2} & \cdots & \frac{1}{2} & 0 & 0 & \cdots & 0\\
            \end{bmatrix}, \\
 \e(2i) & =\begin{bmatrix}
              0 & \cdots & 0 & \frac{1}{2} & 0 & \cdots & 0\\
              \frac{1}{2} & \cdots & \frac{1}{2} & \frac{1}{2} & 0 & \cdots & 0\\
            \end{bmatrix}
\end{split}
\]
%
where the nonzero element of the first row appears in $i^{th}$ column.
We define  $\e(\infty) $ to be $
\begin{bmatrix}
              0 & \cdots & 0 & 0\\
              0 & \cdots & 0 & 0\\
            \end{bmatrix}$.
For any $T \subset B $, we define the half-integer characteristic as

\[ \e_T = \sum_{a_k \in T } \e(k) .\]
Let $T^c$ denote the complement of $T$ in $B.$ Note that $\e_B \in \Z^{2g}.$ If we view $\e_T$ as an element of
$\frac{1}{2}\Z^{2g}/\Z^{2g}$ then $\e_T= \e_{T^c}.$ Let $\triangle$ denote the symmetric difference of sets, that is $T
\triangle R = (T \cup R) - (T \cap R).$ It can be shown that the set of subsets of $B$ is a group under $\triangle.$ We
have the following group isomorphism:
\[ \{T \subset B\,  |\, \#T \equiv g+1 \mod 2\} / T \sim T^c \cong \frac{1}{2}\Z^{2g}/\Z^{2g}.\]
For $\gamma = \ch{\gamma ^\prime}{\gamma^{\prime \prime}} \in \frac{1}{2}\Z^{2g}/\Z^{2g}$, we have
\begin{equation}\label{parityIdentity}
 \T [\gamma] (-z , \t) = e_* (\gamma) \T [\gamma] (z , \t).\end{equation}
It is known that for hyperelliptic curves, $2^{g-1}(2^g+1) - {2g+1 \choose g}$ of the even theta constants are zero.
The following theorem provides a condition for the characteristics in which theta characteristics become zero. The
proof of the theorem can be found  in \cite{Mu2}.
\begin{thm}\label{vanishingProperty}
Let $\X$ be a hyperelliptic curve, with a set $B$ of branch points. Let $S$ be the index set as above and $U $ be the
set of all odd values of $S$. Then for all $T \subset S$ with even cardinality, we have $ \T[\e_T] = 0$  if and only if
$\#(T \triangle U) \neq g+1$, where $\T[\e_T]$ is the theta constant corresponding to the characteristics $\e_T$.
\end{thm}
When the characteristic $\gamma$ is odd, $e_* (\gamma)=1.$ Then from Eq.~\eqref{parityIdentity} all odd theta constants
are zero. There is a formula which satisfies half-integer theta characteristics for hyperelliptic curves called
Frobenius' theta formula.

\begin{lem} [Frobenius]
For all $z_i \in \C^g$, $1\leq i \leq 4$ such that $z_1 + z_2 + z_3 + z_4 = 0$ and for all $b_i \in \Q^{2g}$, $1\leq i
\leq 4$ such that $b_1 + b_2 + b_3 + b_4 = 0$, we have
\[ \sum_{j \in S \cup \{\infty\}} \epsilon_U(j) \prod_{i =1}^4 \T[b_i+\e(j)](z_i) = 0, \]
where for any $A \subset B$,
\[
\epsilon_A(k) =
          \begin{cases}
           1 & \textit {if $k \in A$}, \\
           -1 & \textit {otherwise}.
          \end{cases}
\]
\end{lem}
\proof See \cite[pg.  107]{Mu1}. \qed

A relationship between theta constants and the branch points of the hyperelliptic curve is given by Thomae's formula.
\begin{lem}[Thomae]\label{Thomae}
%
For all sets of branch points $B=\{\a_1,\a_2, \cdots ,\a_{2g+1} \},$ there is a constant $A$ such that for all $T\subset B,$ $\# T$ is even, \\
\[\T[\eta_T](0;\t)^4 =(-1)^{\#T \cap U} A \prod_{i<j \atop i,j \in T \triangle U} (\a_{i} - \a_{j}) \prod_{i<j \atop i,j \notin T \triangle U} (\a_{i} -
\a_{j})  \] \\
\noindent where $\eta_T$ is a non singular even half-integer characteristic corresponding to the subset $T$ of branch
points.

\end{lem}
See \cite[pg.  128]{Mu1} for the description of $A$ and \cite[pg.  120]{Mu1} for the proof. Using Thomae's formula and
Frobenius' theta identities we express the branch points of the hyperelliptic curves in terms of even theta constants.

\subsection{Cyclic Curves and Their Theta Functions}
A cyclic cover $\X \longrightarrow \P^1$ is defined to be a Galois cover with cyclic Galois group $C.$ We call it a
normal cyclic cover of $\P^1$ if $C$ is normal in $G=Aut(\X)$ where $Aut(\X)$ is the automorphism group of the curve
$\X.$ Then $\bar{G} = G/C$ embeds as a finite subgroup of $PGL(2,\C)$ and it is called the reduced automorphism group
of $G.$

\noindent An affine  equation of a cyclic curve can be given by the following:
\begin{equation}\label{cyclic}
y^m = f(x) = \prod_{i=1}^s (x-\a_i)^{d_i} , \, m = |C|, \,\, 0 < d_i < m.
\end{equation}
Note that when $d_i>0$ for some $i$ the curve is singular.
Hyperelliptic curves are cyclic curves with $m=2$. After Thomae, many mathematicians, for example Fuchs, Bolza, Fay,
Mumford, et al., gave derivations of Thomae's formula in the hyperelliptic case. In 1988 Bershdaski and Radul found a
generalization of Thomae's formula for $Z_N $ curves of the form
\begin{equation} \label{Nakayashiki}
y^N = f(x) = \prod_{i=1}^{Nm}(x-a_i).
\end{equation}

In 1988 Shiga showed the representation of the Picard modular function by theta constants. He considered the algebraic
curve in the $(x,y)$ plane which is given by
\begin{equation}\label{picardcurve}
C(\ep) : y^3 = x (x-a_0) (x-a_1) (x-a_2)
\end{equation}
where $\ep = [a_0, a_1,a_2]$ is a parameter on the domain

$$\Lambda = \{ \ep : a_0 a_1 a_2 (a_0-a_1) (a_0-a_2) (a_1-a_2) \neq 0 \}.$$
He gave a concrete description of the Picard work \cite{picard}. His result can be considered  an extension of the
classical Jacobi representation $\lambda = \frac{\T_2^4}{\T_3^4}$, where $\T_i(z,\t)$ indicates Jacobi's theta function
and $\T_i$ is the convention for $\T_i(0,\t)$, for the elliptic modular function $\lambda(\t)$ to the special case of
genus 3.

In 1991, Gonzalez Diez studied the moduli spaces parameterizing algebraic curves which are Galois covering of $\P^1$
with prime order and with given ramification numbers. These curves have equation of the form
\begin{equation}\label{Gabino}
y^p = f(x) = \prod_{i=1}^{r}(x-a_i)^{m_i}  ; \textit {p prime and  } p  \nmid \sum m_i.
\end{equation}
He expresses $a_i$ in terms of functions of the period matrix of the curve.

Farkas (1996) gave a procedure for calculating the complex numbers $a_i$ which appear in the algebraic equation
\begin{equation}\label{Farkas}
y^p = \prod_{i=1}^{k}(x-a_i) \, \, \, \, \textit{with} \,\,\, p|k
\end{equation}
in terms of the theta functions associated with the Riemann surface of the algebraic curve defined by the
Eq.~\eqref{Farkas}. He used the generalized cross ratio of four points according to Gunning. Furthermore he considered
the more general problem of a branched two-sheeted cover of a given compact Riemann surface and obtained the relations
between the theta functions on the cover and the theta function to the original surface.

Nakayashiki, in 1997, gave an elementary proof of Thomae's formula for $Z_N$ curves which was discovered by Bershadsky
and Radul.
Enolski and Grava, in 2006, derived the analogous generalized Thomae's formula for the $Z_N$ singular curve of the form
\begin{equation} \label{Enolski}
y^N = f(x) = \prod_{i=1}^{m}(x-\lambda_{2i})^{N-1} \prod_{i=1}^{m} ( x-\lambda_{2i+1}).
\end{equation}

\noindent We summarize all the results in the following theorem.

\begin{thm}\label{CyclicCurvesThm}
Consider the algebraic curve $ \X : y^n = f(x)$ defined over the complex field $\C.$
%

\noindent\textbf{Case 1:} If $\triangle_f \neq 0,$ say $f(x) = \prod_{i=1}^k (x- \l_i)$ then,

\underline {i)} If $n | k,$ say $k = mn$ for some $m \in \N$ then,

for an ordered partition $\Lambda = (\Lambda_0, \cdots, \Lambda_{n-1})$ of $\{1,2,\cdots, nm\},$ we have
\[ \T[e_\Lambda](0)^{2n} = C_\Lambda (det A)^n \prod_{i<j}(\lambda_i -
\lambda_j)^{2 n \sum_{\ell \in L} q_\ell(k_i) q_{\ell}(k_j) + \frac{(n-1)(2n-1)}{6}}\]
where $k_i = j$ for $i \in \Lambda_j$ and $e_\Lambda \equiv \Lambda_1 + 2 \Lambda_2 +\cdots +  (n-1)\Lambda_{n-1} - D -
\varsigma$ is the associated divisor class of the partition $\Lambda,$ $L  = \big\{-\frac{N-1}{2}, -\frac{N-1}{2} +1,
\cdots, \frac{N-1}{2}\big\},$ $q_\ell(i) = \frac{1-N}{2N} +\,\, \textit{fraction part of}\,\,\big(\frac { \ell + i +
\frac{N-1}{2}}{N}\big)   \, \, \, \, \textit{for} \, \, \ell \in L,$ $\varsigma$ is Riemann's constant and $C_\Lambda$
depends on the partition $\Lambda$ having the property that for two different partitions $\Lambda$ and $\Lambda^\prime$
we have $C_\Lambda ^{2N} =C_{\Lambda ^\prime} ^{2N}.$

Moreover if $n $ is a prime $p,$ the branch points $\lambda_i$ of the curve $y^n = x(x-1)(x-\lambda_1) \cdots
(x-\lambda_{k-3})$  can be given by
\[ E_i^n \lambda_i = (\lambda(P_k, Q_0, Q_1, Q_\infty))^n\]
where
$\lambda(P_k, Q_0, Q_1, Q_\infty) = \frac{\theta(e+\phi_{Q_0}(P_k))\theta(e+\phi_{Q_{\infty}}(Q_1))}
{\theta(e+\phi_{Q_\infty}(P_k))\theta(e+\phi_{Q_{0}}(Q_1))},$ while  $Q_0,$ $Q_1,$ and $Q_\infty$ denote the points in
the curve corresponding to the points $0,$ $1,$ and $\infty$ in $\P^1$ respectively, $P_i$'s are points in the curve
corresponding to the points $\lambda_i,$ $E_i$ is a constant depending on the point $P_i$
%
and $\phi_P $ is an injective map from $\X$ to $\C^g/G.$

 \underline {ii)} If $ n \nmid k,$ then,

if $n=3$ and $k=4$, then the parameters $\l_1, \l_2, \l_3$ can be given as follows:

$$\l_1 =  \T^3 \chs{0}{\frac{1}{6}}{0}{0}{\frac{1}{6}}{0},  \quad \quad \l_2 = \T^3
\chs{0}{\frac{1}{6}}{0}{\frac{1}{3}}{\frac{1}{6}}{\frac{1}{3}}, \quad \quad \l_3 =  \T^3
\chs{0}{\frac{1}{6}}{0}{\frac{2}{3}}{\frac{1}{6}}{\frac{2}{3}}.$$

\noindent \textbf{Case 2:} If $\triangle_f = 0,$ let $f(x) = \prod_{k=0}^m (x- \lambda_{2k+1})\prod_{k=1}^m (x-
\lambda_{2k})^{n-1}.$ Then,
\[
\begin{split}
\T[e_m] (0; \O)^{4N} = &\frac{\prod_{i=1}^{N-1} \det A_i^{2N}}{(2 \pi i )^{2mN(N-1)} } \prod_{1 \leq i < k \leq m
}(\lambda_{2i}-\lambda_{2k})^{N(N-1)} \\
 & \times \prod_{0 \leq i < k \leq m
}(\lambda_{2i+1}-\lambda_{2k+1})^{N(N-1)} \\
 & \times (\frac{\prod_{i \in I_1 , j \in J_1 }(\lambda_i - \lambda_j) \prod_{i \in I_2 , j \in J_2 }(\lambda_i - \lambda_j)}
 {\prod_{i \in I_1 , k \in I_2 }(\lambda_i - \lambda_k) \prod_{j \in J_1 , k \in J_2 }(\lambda_i -
 \lambda_j)})^{2(N-1)},
\end{split}
\]
where
%
$e_m = \nu ((N-1) \sum_{i \in I_1}P_i +(N-1) \sum_{j \in J_1} P_j - D- \triangle )$ is a nonsingular
$\frac{1}{N}$ characteristic,
 $J_1 \subset J_0 = \{2,4,\cdots,2m+2\}$ and $I_1 \subset I_0 =
\{1,3,\cdots,2m+1\}$ with
$|J_1| + |I_1| = m+1$ and $I_2 = I_0 - I_1 , J_2 = J_0 - J_1 - 2m+2,$ and
$\triangle = (N-1) \sum_{k=1}^m P_{2k} - P_{\infty}$ is the Riemann divisor of the curve $\X.$
%

\end{thm}

\begin{proof}
For proof of the part $i)$ of case 1, see \cite{NK}. When $n$ is prime, the proof can be found in \cite{Farkas}. The
main point of \cite{SHI} is to prove part $ii)$ of case 1. The proof of case 2 can be found in \cite{Enolski}.
\end{proof}

\subsubsection{Relations Among Theta Functions for Algebraic Curves with Automorphisms}
In this section we develop an algorithm to determine relations among theta functions of a cyclic curve $\X$ with
automorphism group $\Aut (\X)$. The proof of the following lemma can be found in \cite{RF}.
\begin{lem}\label{Shiga}
Let $f$ be a meromorphic function on $\X,$ and let $$(f) = \sum_{i=1}^{m}b_i - \sum_{i=1}^{m}c_i $$ be the divisor
defined by $f.$ Take paths from $P_0$ (initial point) to $b_i$ and $P_0$ to $c_i$ so that $\sum_{i=1}^{m}
\int_{P_0}^{b_i} \om = \sum_{i=1}^{m} \int_{P_0}^{c_i} \om .$

For an effective divisor $P_1 + \cdots + P_g,$ we have
\begin{equation}
f(P_1) \cdots f(P_g) = \frac{1}{E} \prod_{k=1}  \frac{\T(\sum_{i} \int_{P_0}^{P_i} \om - \int_{P_0}^{b_k} \om -
\triangle , \tau )} {\T(\sum_{i} \int_{P_0}^{P_i} \om - \int_{P_0}^{c_k} \om - \triangle , \tau )}
\end{equation}
where $E$ is a constant independent of $P_1, \dots , P_g,$ the integrals from $P_0$ to $P_i$ take the same paths both
in the numerator and in the denominator, $\triangle$ denotes the Riemann's constant, and $\int_{P_0}^{P_i} \om = \left(
\int_{P_0}^{P_i} \om_1, \dots, \int_{P_0}^{P_i} \om_g \right)^t.$
\end{lem}
This lemma gives us a  tool that can be used to find branch points in terms of theta constants. By considering the
meromorphic function $f = x$ on $\X$ and suitable effective divisors, we can write branch points as ratios of
thetanulls.  We present some explicit calculations using the Lemma ~\ref{Shiga} in Chapter 3 and 4. The hard part of
this

\noindent method is the difficulty of writing complex integrals in terms of characteristics.

\begin{alg}

Input: A cyclic curve $\X$ with automorphism group $G$, $\s \in G$  such that $| \s| =n$, $g (\X^\s)
=0$ and $\< \s \> \normal G$. \\

Output: Relations among the theta functions of $\X$\\

Step 1: Let $\G = G/\<\s\>$ and pick $\t \in \G$ such that $\t$ has the largest order $m$.

Step 2:  Write the equation of the curve in the form \[ y^n = f(x^m) \,\,\, \textit{or} \,\,\,   y^n = x f(x^m).\]

Step 3: Determine the roots $\l_1, \dots , \l_r$ of $f(x^\t)$ in terms of the theta functions.

Step 4: Determine relations on theta functions using Gr\"obner basis techniques.

\end{alg}

For step 3, we can use Lemma ~\ref{Shiga}. If the curve in step 3 falls into one of the categories given in Theorem
~\ref{CyclicCurvesThm}, we can use the corresponding equation to invert the period map without worrying about the
complex integrals.


\section{Genus 2 curves}
Let $k$ be an algebraically closed field of  characteristic zero and $\X$ be a genus 2 curve defined over $k$. Consider
a binary sextic, i.e. a homogeneous polynomial $f(X,Z)$ in $k[X,Z]$ of degree 6:
$$f(X,Z)=a_6 X^6+ a_5 X^5Z+\dots  +a_0 Z^6.$$ The polynomial functions of the coefficients of a binary sextic $f(X,Z)$ invariant under
linear substitutions in $X,Z$ of determinant one. These invariants were worked out by Clebsch and  Bolza in the case of
zero characteristic and generalized by Igusa for any characteristic different from 2.

{\it Igusa  $J$-invariants} $\, \, \{ J_{2i} \}$ of $f(X,Z)$ are homogeneous polynomials of degree $2i$ in $k[a_0,
\dots , a_6]$, for  $i=1,2,3,5$; see   \cite{S7} for their definitions. Here  $J_{10}$ is the discriminant of $f(X,Z)$.
It vanishes if and only if the binary sextic has a multiple linear factor. These $J_{2i}$  are invariant under the
natural action of $SL_2(k)$ on sextics. Dividing such an invariant by another invariant with the same degree, gives an
invariant (eg. absolute invariant) under $GL_2(k)$ action. The absolute invariants of $\X$ are defined in terms of
Igusa invariants as follows:  \[ i_1 := 144 \frac{J_4}{J_2^2}, \quad \quad i_2:= -1728 \frac{J_2 J_4 - 3 J_6}{J_2^3},
\quad \quad i_3:= 486 \frac{J_{10}}{J_2 ^5}. \]
Two genus  2 fields (resp., curves) in the standard form $Y^2=f(X,1)$ are isomorphic if and only if the corresponding
sextics are $GL_2(k)$ conjugate.
\subsection{Half Integer Theta Characteristics} For genus two curve, we have six odd theta
characteristics and ten even theta characteristics. The following are the sixteen theta characteristics where the first
ten are even and the last six are odd. For simplicity, we denote them by $\T_i(z)$ instead of $\T_i \ch{a} {b} (z ,
\t)$ where $i=1,\dots ,10$ for the even functions and $i=11, \dots, 16$ for the odd functions.
\[
\begin{split}
\T_1(z) &= \T_1 \chr {0}{0}{0}{0} (z , \t), \qquad \qquad \T_2(z) = \T_2 \chr {0}{0}{\frac{1}{2}} {\frac{1}{2}} (z , \t)\\
\T_3(z) &= \T_3 \chr {0}{0}{\frac{1}{2}}{0}(z , \t), \qquad \qquad
\T_4(z) = \T_4 \chr {0}{0}{0}{\frac{1}{2}} (z , \t)\\
\T_5(z) &= \T_5 \chr{\frac{1}{2}}{0} {0}{0}(z , \t), \qquad \qquad
\T_6(z) = \T_6 \chr {\frac{1}{2}}{0}{0}{\frac{1}{2}} (z , \t)\\
\T_7(z) &= \T_7 \chr{0}{\frac{1}{2}} {0}{0} (z , \t), \qquad \qquad
\T_8(z) = \T_8 \chr{\frac{1}{2}}{\frac{1}{2}} {0}{0} (z , \t)\\
\end{split}
\]
\[
\begin{split}
\T_9(z) &= \T_9 \chr{0}{\frac{1}{2}} {\frac{1}{2}}{0}(z , \t), \qquad \qquad
\T_{10}(z) = \T_{10} \chr{\frac{1}{2}}{\frac{1}{2}} {\frac{1}{2}}{\frac{1}{2}} (z , \t)\\
\T_{11}(z) &= \T_{11} \chr{0}{\frac{1}{2}} {0}{\frac{1}{2}} (z , \t), \qquad \qquad
\T_{12}(z) = \T_{12} \chr{0}{\frac{1}{2}} {\frac{1}{2}}{\frac{1}{2}} (z , \t)\\
\T_{13}(z) &= \T_{13} \chr{\frac{1}{2}}{0} {\frac{1}{2}}{0} (z , \t), \qquad \qquad
\T_{14}(z) = \T_{14} \chr{\frac{1}{2}}{\frac{1}{2}} {\frac{1}{2}}{0} (z , \t)\\
\T_{15}(z) &= \T_{15} \chr{\frac{1}{2}}{0} {\frac{1}{2}}{\frac{1}{2}} (z , \t), \qquad \qquad
\T_{16}(z) = \T_{16} \chr{\frac{1}{2}}{\frac{1}{2}} {0}{\frac{1}{2}} (z , \t)\\
\end{split}
\]
%
\begin{rem}
All the possible half-integer characteristics except the zero characteristic can be obtained as the sum of not more
than 2 characteristics chosen from the following 5 characteristics:

\[ \left\{\chr {0}{0}{\frac{1}{2}} {\frac{1}{2}}, \chr {\frac{1}{2}}{0}{0}{\frac{1}{2}}, \chr{0}{\frac{1}{2}} {0}{0}, \chr{\frac{1}{2}}{0}
{\frac{1}{2}}{\frac{1}{2}},\chr{0}{\frac{1}{2}} {0}{\frac{1}{2}} \right\}.\] The sum of all 5 characteristics in the
set determines the zero characteristic.
\end{rem}

Take $\sigma = g - r =0 $. Then a G\"opel group $G$ contains four elements. The number of such G\"opel groups is 15.
Let $G = \{0, \m_1, \m_2, \m_1 \m_2\}$ be a G\"opel group of even characteristics (we have six such groups). Let
$\bn_1, \bn_2, \bn_1 \bn_2 $ be the characteristics such that the $G , \bn_1 G , \bn_2 G , \bn_1 \bn_2 G $ are all the
cosets of the group $G.$ Then each of the systems other than $G$ contains two odd characteristics and two even
characteristics. Consider equations given by Eq.~\eqref{eq1} and Eq.~\eqref{eq2}. If $\hn$ denotes any one of the 3
characteristics $\m_1,\m_2,\m_1 \m_2$, then we have 6 possible characteristics for $\en$, which satisfy $|\en,\hn|
\equiv | \hn| \equiv 0$. They are $0, \n, \bn, \hn, \n \hn, \bn \hn$ where $\n$ is a characteristic in the G\"opel
group other than $\hn$, and $\bn$ is an even characteristic chosen from one of the systems $\bn_1 G, \bn_2 G, \bn_1
\bn_2 G$. The following three cases illustrate the possible values for characteristic $\hn$ and for characteristic
$\en$. Without loss of generality, we can take only three values for $\en$ which give rise to different terms on the
right hand side of Eq.~\eqref{eq1} and Eq.~\eqref{eq2}.

\noindent \textbf{Case 1:}  $\hn= \m_1.$

Take $\en \in \{ 0, \m_2, \bn_1 \}$ and take $\an = \bn_1$. Then from Eq.~\eqref{eq1} and Eq.~\eqref{eq2} we have
%
\[
\begin{split}
&{ \m_1 \choose \bn_1} \T^2[ 0] \T^2[ \m_1] + e^{\pi i |\bn_1 \m_2|} { \m_1 \choose \bn_1 \m_2} \T^2[\m_2] \T^2[\m_1
\m_2] -
\T^2[\bn_1] \T^2[\bn_1 \m_1] = 0, \, \,  \\
&\T^4[ 0] + \T^4[ \m_1] + e^{\pi i |\bn_1 \m_2|} [ \T^4[ \m_2] + \T^4[ \m_2 \m_1]] - [ \T^4[ \bn_1] + \T^4[
\bn_1 \m_1]] = 0. \\
\end{split}
\]

\noindent \textbf{Case 2:} $\hn= \m_2.$

Take $\en \in \{ 0, \m_1, \bn_2 \}$ and take $\an = \bn_2$. Then from Eq.~\eqref{eq1} and Eq.~\eqref{eq2} we have
%
\[
\begin{split}
&{ \m_2 \choose \bn_2} \T^2[ 0] \T^2[ \m_2] + e^{pi i |\bn_2 \m_1|} { \m_2 \choose \bn_2 \m_1} \T^2[ \m_1] \T^2[\m_1
\m_2] -
\T^2[ \bn_2] \T^2[\bn_2 \m_2] = 0, \,\,  \\
&\T^4[ 0] + \T^4[ \m_2] + e^{pi i |\bn_2 \m_2|} [ \T^4[ \m_1] + \T^4[ \m_1 \m_2]] - [ \T^4[ \bn_2] + \T^4[
\bn_2 \m_2]] = 0. \\
\end{split}
\]

\vspace*{.4cm}

\noindent \textbf{Case 3:} $\hn= \m_1 \m_2.$

Take $\en \in \{ 0, \m_1, \bn_1 \bn_2 \}$ and take $\an = \bn_1 \bn_2$. Then from Eq.~\eqref{eq1} and Eq.~\eqref{eq2}
we have
%
\[
\begin{split}
&{ \m_1 \m_2 \choose \bn_1 \bn_2} \T^2[ 0] \T^2[ \m_1 \m_2] + e^{pi i |\bn_1 \bn_2 \m_1|} { \m_1 \m_2 \choose \bn_1
\bn_2 \m_1} \T^2[ \m_1] \T^2[ \m_2] -\\
&\T^2[ \bn_1 \bn_2] \T^2[ \bn_1 \bn_2 \m_1 \m_2] = 0, \,\,  \\
&\T^4[ 0] + \T^4[ \m_1 \m_2] + e^{pi i |\bn_1 \bn_2 \m_1|} [ \T^4[ \m_1] + \T^4[ \m_2]] - [ \T^4[ \bn_1 \bn_2] +
\T^4[ \bn_1 \bn_2 \m_1 \m_2]] = 0. \\
\end{split}
\]
%
%
The identities above express the even theta constants in terms of four theta constants; therefore, we may call them
fundamental theta constants,
$$\T[0] ,\T[ \m_1], \T[ \m_2], \T[\m_1 \m_2].$$
\subsection{Identities of Theta Constants}
We have only six G\"opel groups such that all characteristics are even. The following are such G\"opel groups and
corresponding identities of theta constants.
\begin{description}

\item [i)] $G = \left\{0 = \chr{0}{0}{0}{0}, \m_1 = \chr {0}{0}{0}{\frac{1}{2}}, \m_2 = \chr
{0}{0}{\frac{1}{2}}{0}, \m_1 \m_2 = \chr {0}{0}{\frac{1}{2}} {\frac{1}{2}} \right\}$ is a G\"opel group. If $\bn_1 =
\chr{\frac{1}{2}}{0} {0}{0}, \bn_2 = \chr{0}{\frac{1}{2}} {\frac{1}{2}}{0}$, then the corresponding G\"opel systems are
given by the following:
%
\[
\begin{split}
G &= \left\{ \chr{0}{0}{0}{0},  \chr {0}{0}{0}{\frac{1}{2}},  \chr {0}{0}{\frac{1}{2}}{0},
 \chr {0}{0}{\frac{1}{2}} {\frac{1}{2}} \right\}, \\
\bn_1 G &= \left\{\chr{\frac{1}{2}}{0} {0}{0}, \chr {\frac{1}{2}}{0}{0}{\frac{1}{2}},
 \chr{\frac{1}{2}}{0} {\frac{1}{2}}{0}, \chr{\frac{1}{2}}{0} {\frac{1}{2}}{\frac{1}{2}} \right\}, \\
\bn_2 G &= \left\{\chr{0}{\frac{1}{2}} {\frac{1}{2}}{0}, \chr{0}{\frac{1}{2}} {\frac{1}{2}}{\frac{1}{2}},
\chr{0}{\frac{1}{2}} {0}{0},
\chr{0}{\frac{1}{2}} {0}{\frac{1}{2}} \right\}, \\
\end{split}
\]
\[
\begin{split}
\bn_3 G &= \left\{\chr{\frac{1}{2}}{\frac{1}{2}} {\frac{1}{2}}{0}, \chr{\frac{1}{2}}{\frac{1}{2}}
{\frac{1}{2}}{\frac{1}{2}}, \chr{\frac{1}{2}}{\frac{1}{2}} {0}{0}, \chr{\frac{1}{2}}{\frac{1}{2}} {0}{\frac{1}{2}}
\right\}.
\end{split}
\]
%
Notice that from all four cosets, only $G$ has all even characteristics as noticed in Corollary~\ref{numb_systems}.
Using Eq.~\eqref{eq1} and Eq.~\eqref{eq2}, we have the following six identities for the above G\"opel group:
\[
\left \{ \begin{array}{lll}
 \T_5^2 \T_6^2 & = & \T_1^2 \T_4^2 - \T_2^2 \T_3^2 ,\\
 \T_5^4 + \T_6^4 & =& \T_1^4 - \T_2^4 - \T_3^4 + \T_4^4, \\
 \T_7^2 \T_9^2 & = & \T_1^2 \T_3^2 - \T_2^2 \T_4^2, \\
 \T_7^4 + \T_9^4 &= & \T_1^4 - \T_2^4 + \T_3^4 - \T_4^4, \\
 \T_8^2 \T_{10}^2 & = & \T_1^2 \T_2^2 - \T_3^2 \T_4^2, \\
 \T_8^4 + \T_{10}^4 & = & \T_1^4 + \T_2^4 - \T_3^4 - \T_4^4. \\
\end{array}
\right.
\]
These identities express even theta constants in terms of four theta constants. We call them fundamental theta
constants $\T_1, \, \T_2, \, \T_3, \, \T_4$.
Following the same procedure, we can find similar identities for each possible G\"opel group.
\item [ii)] $G = \left\{0 = \chr {0}{0}{0}{0}, \m_1 = \chr {0}{0}{\frac{1}{2}}{0}, \m_2 = \chr{\frac{1}{2}}{\frac{1}{2}} {\frac{1}{2}}{\frac{1}{2}},
 \m_1 \m_2 = \chr{\frac{1}{2}}{\frac{1}{2}} {0}{0}\right\}$
 is a G\"opel group. If $\bn_1 = \chr {0}{0}{\frac{1}{2}}{0} , \bn_2 =
\chr {\frac{1}{2}}{0}{0}{\frac{1}{2}} $, then the corresponding G\"opel systems are given by the following:
%
\[
\begin{split}
G &= \left\{ \chr {0}{0}{0}{0},  \chr {0}{0}{\frac{1}{2}}{0},  \chr{\frac{1}{2}}{\frac{1}{2}}
{\frac{1}{2}}{\frac{1}{2}}, \chr{\frac{1}{2}}{\frac{1}{2}} {0}{0}\right\}, \\
\bn_1 G &= \left\{\chr {0}{0}{\frac{1}{2}}{0}, \chr {0}{0}{0}{\frac{1}{2}}, \chr{\frac{1}{2}}{\frac{1}{2}}
{0}{\frac{1}{2}},
 \chr{\frac{1}{2}}{\frac{1}{2}} {\frac{1}{2}}{0} \right\}, \\
\end{split}
\]
\[
\begin{split}
\bn_2 G &= \left\{\chr {\frac{1}{2}}{0}{0}{\frac{1}{2}}, \chr{\frac{1}{2}}{0} {\frac{1}{2}}{0}, \chr{0}{\frac{1}{2}}
{\frac{1}{2}}{0}
, \chr{0}{\frac{1}{2}} {0}{\frac{1}{2}} \right\}, \\
\bn_3 G &= \left\{\chr{\frac{1}{2}}{0} {\frac{1}{2}}{\frac{1}{2}}, \chr{\frac{1}{2}}{0} {0}{0}, \chr{0}{\frac{1}{2}}
{0}{0} \chr{0}{\frac{1}{2}} {\frac{1}{2}}{\frac{1}{2}} \right\}.
\end{split}
\]
%
We have the following six identities for the above G\"opel group:

\[
\left \{ \begin{array}{lll}
\T_3^2 \T_4^2 & =& \T_1^2 \T_2^2 - \T_8^2 \T_{10}^2, \\
\T_3^4 + \T_4^4 & = &\T_1^4 + \T_2^4 - \T_8^4 - \T_{10}^4, \\
 \T_6^2 \T_9^2 & = & -\T_1^2 \T_{10}^2 + \T_2^2 \T_8^2, \\
\T_6^4 + \T_9^4 & = & \T_1^4 - \T_2^4 - \T_8^4 + \T_{10}^4, \\
 \T_5^2 \T_7^2 & = & \T_1^2 \T_8^2 - \T_2^2 \T_{10}^2, \\
\T_5^4 + \T_7^4 & = & \T_1^4 - \T_2^4 + \T_8^4 - \T_{10}^4. \\
%
\end{array}
\right.
\]
%
\item [iii)] $G = \left\{0 = \chr {0}{0}{0}{0}, \m_1 = \chr {0}{0}{\frac{1}{2}}{0}, \m_2 = \chr{0}{\frac{1}{2}}{0}{0},
\m_1 \m_2 = \chr{0}{\frac{1}{2}} {\frac{1}{2}}{0} \right\}$ is a G\"opel group. If $\bn_1 = \chr {0}{0}{\frac{1}{2}}
{\frac{1}{2}} , \bn_2 = \chr{\frac{1}{2}}{\frac{1}{2}} {0}{0} $, then the corresponding G\"opel systems are given by
the following:
%
\[
\begin{split}
G &= \left\{ \chr {0}{0}{0}{0},  \chr {0}{0}{\frac{1}{2}}{0},  \chr{0}{\frac{1}{2}}{0}{0},
 \chr{0}{\frac{1}{2}} {\frac{1}{2}}{0} \right\}, \\
\bn_1 G &= \left\{\chr {0}{0}{\frac{1}{2}} {\frac{1}{2}}, \chr {0}{0}{0}{\frac{1}{2}}, \chr{0}{\frac{1}{2}}
{\frac{1}{2}}{\frac{1}{2}},
 \chr{0}{\frac{1}{2}} {0}{\frac{1}{2}} \right\}, \\
\bn_2 G &= \left\{\chr{\frac{1}{2}}{\frac{1}{2}} {0}{0}, \chr{\frac{1}{2}}{\frac{1}{2}} {\frac{1}{2}}{0},
\chr{\frac{1}{2}}{0} {0}{0},
\chr{\frac{1}{2}}{0} {\frac{1}{2}}{0} \right\}, \\
\bn_3 G &= \left\{\chr{\frac{1}{2}}{\frac{1}{2}} {\frac{1}{2}}{\frac{1}{2}}, \chr{\frac{1}{2}}{\frac{1}{2}}
{0}{\frac{1}{2}}, \chr{\frac{1}{2}}{0} {\frac{1}{2}}{\frac{1}{2}}, \chr {\frac{1}{2}}{0}{0}{\frac{1}{2}} \right\}.
\end{split}
\]
%
We have the following six identities for the above G\"opel group:
%
\[
\left \{ \begin{array}{lll}
\T_2^2 \T_4^2  & = &\T_1^2 \T_3^2 - \T_7^2 \T_{9}^2, \\
 \T_2^4 + \T_4^4 & = & \T_1^4 + \T_3^4 - \T_7^4 - \T_{9}^4, \\
 \T_8^2 \T_5^2 & = & \T_1^2 \T_{7}^2 - \T_3^2 \T_9^2, \\
 \T_8^4 + \T_5^4 & = & \T_1^4 - \T_3^4 + \T_7^4 - \T_{9}^4, \\
 \T_6^2 \T_{10}^2 & = & -\T_1^2 \T_9^2 + \T_3^2 \T_{7}^2, \\
 \T_6^4 + \T_{10}^4 & = & \T_1^4 - \T_3^4 - \T_7^4 + \T_{9}^4. \\
%
\end{array}
\right.
\]
%
\medskip
\item [iv)] $G = \left\{0 = \chr {0}{0}{0}{0}, \m_1 = \chr {0}{0}{0}{\frac{1}{2}}, \m_2 = \chr{\frac{1}{2}}{0} {0}{0},
\m_1 \m_2 = \chr {\frac{1}{2}}{0}{0}{\frac{1}{2}} \right\}$ is a G\"opel group. If $\bn_1 = \chr {0}{0}{\frac{1}{2}}
{\frac{1}{2}}, \bn_2 = \chr{\frac{1}{2}}{\frac{1}{2}} {0}{0}$, then the corresponding G\"opel systems are given by the
following:

%
\[
\begin{split}
G  &= \left\{ \chr {0}{0}{0}{0},  \chr {0}{0}{0}{\frac{1}{2}},  \chr{\frac{1}{2}}{0} {0}{0},
 \chr {\frac{1}{2}}{0}{0}{\frac{1}{2}} \right\},\\
\bn_1 G &= \left\{\chr {0}{0}{\frac{1}{2}} {\frac{1}{2}}, \chr {0}{0}{\frac{1}{2}}{0}, \chr{\frac{1}{2}}{0}
{\frac{1}{2}}{\frac{1}{2}},
 \chr{\frac{1}{2}}{0} {\frac{1}{2}}{0} \right\}, \\
 \bn_2 G &= \left\{\chr{\frac{1}{2}}{\frac{1}{2}} {0}{0}, \chr{\frac{1}{2}}{\frac{1}{2}}
{0}{\frac{1}{2}},\chr{0}{\frac{1}{2}} {0}{0},
\chr{0}{\frac{1}{2}} {0}{\frac{1}{2}} \right\}, \\
 \bn_3 G &= \left\{\chr{\frac{1}{2}}{\frac{1}{2}} {\frac{1}{2}}{\frac{1}{2}} ,\chr{\frac{1}{2}}{\frac{1}{2}}
{\frac{1}{2}}{0}, \chr{0}{\frac{1}{2}} {\frac{1}{2}}{\frac{1}{2}}, \chr{0}{\frac{1}{2}} {\frac{1}{2}}{0} \right\}.
\end{split}
\]
%
We have the following six identities for the above G\"opel group:

\[
\left \{ \begin{array}{lll}
 \T_2^2 \T_3^2 & = & \T_1^2 \T_4^2 - \T_5^2 \T_{6}^2, \\
 \T_2^4 + \T_3^4 & = & \T_1^4 + \T_4^4 - \T_5^4 - \T_{6}^4, \\
 \T_8^2 \T_7^2 & = & \T_1^2 \T_{5}^2 - \T_4^2 \T_6^2, \\
 \T_8^4 + \T_7^4 & = & \T_1^4 - \T_4^4 + \T_5^4 - \T_{6}^4, \\
 \T_9^2 \T_{10}^2 & = & -\T_1^2 \T_6^2 + \T_4^2 \T_{5}^2, \\
 \T_9^4 + \T_{10}^4 & = & \T_1^4 - \T_4^4 - \T_5^4 + \T_{6}^4. \\
%
\end{array}
\right.
\]
%
\medskip
\item [v)] $G = \left\{0 = \chr {0}{0}{0}{0}, \m_1 = \chr {\frac{1}{2}}{0} {0}{0}, \m_2 = \chr {0}{\frac{1}{2}} {0}{0},
 \m_1 \m_2 = \chr {\frac{1}{2}}{\frac{1}{2}}{0}{0} \right\}$ is a G\"opel group. If $\bn_1 = \chr {\frac{1}{2}}{0}{0}{\frac{1}{2}} , \bn_2 =
\chr {0}{0}{\frac{1}{2}}{0} $, then the corresponding G\"opel systems are given by the following:
%
\[
\begin{split}
G &= \left\{ \chr {0}{0}{0}{0},  \chr {\frac{1}{2}}{0} {0}{0},  \chr {0}{\frac{1}{2}} {0}{0},
 \chr {\frac{1}{2}}{\frac{1}{2}}{0}{0} \right\}, \\
 \bn_1 G &= \left\{\chr {\frac{1}{2}}{0}{0}{\frac{1}{2}}, \chr {0}{0}{0}{\frac{1}{2}}, \chr{\frac{1}{2}}{\frac{1}{2}}
{0}{\frac{1}{2}},
\chr{0}{\frac{1}{2}} {0}{\frac{1}{2}} \right\}, \\
\end{split}
\]
\[
\begin{split}
 \bn_2 G &= \left\{\chr {0}{0}{\frac{1}{2}}{0},\chr{\frac{1}{2}}{0} {\frac{1}{2}}{0}, \chr{0}{\frac{1}{2}}
{\frac{1}{2}}{0},
\chr{\frac{1}{2}}{\frac{1}{2}} {\frac{1}{2}}{0} \right\}, \\
 \bn_3 G &= \left\{\chr{\frac{1}{2}}{0} {\frac{1}{2}}{\frac{1}{2}}, \chr {0}{0}{\frac{1}{2}} {\frac{1}{2}},
\chr{\frac{1}{2}}{\frac{1}{2}} {\frac{1}{2}}{\frac{1}{2}}, \chr{0}{\frac{1}{2}} {\frac{1}{2}}{\frac{1}{2}}\right\}.
\end{split}
\]
%
We have the following six identities for the above G\"opel group:

\[
\left \{ \begin{array}{lll}
 \T_4^2 \T_6^2 & = & \T_1^2 \T_5^2 - \T_7^2 \T_{8}^2, \\
 \T_4^4 + \T_6^4 & = & \T_1^4 + \T_5^4 - \T_7^4 - \T_{8}^4, \\
 \T_3^2 \T_9^2 & = & \T_1^2 \T_{7}^2 - \T_5^2 \T_8^2, \\
 \T_3^4 + \T_9^4 & = & \T_1^4 - \T_5^4 + \T_7^4 - \T_{8}^4, \\
 \T_2^2 \T_{10}^2 & = & \T_1^2 \T_8^2 - \T_5^2 \T_{7}^2, \\
 \T_2^4 + \T_{10}^4 & = & \T_1^4 - \T_5^4 - \T_7^4 + \T_{8}^4. \\
%
\end{array}
\right.
\]
%
\medskip
\item [vi)] $G = \left\{0 = \chr {0}{0}{0}{0}, \m_1 = \chr {\frac{1}{2}}{0}{0}{\frac{1}{2}}, \m_2 = \chr{0}{\frac{1}{2}} {\frac{1}{2}}{0},
\m_1 \m_2 = \chr{\frac{1}{2}}{\frac{1}{2}} {\frac{1}{2}}{\frac{1}{2}} \right\}$ is a G\"opel group. If $\bn_1 =
\chr{\frac{1}{2}}{0} {0}{0} , \bn_2 = \chr{0}{\frac{1}{2}} {0}{0}$, then the corresponding G\"opel systems are given by
the following:
%
\[
\begin{split}
G &= \left\{ \chr {0}{0}{0}{0},   \chr {\frac{1}{2}}{0}{0}{\frac{1}{2}},   \chr{0}{\frac{1}{2}} {\frac{1}{2}}{0},
 \chr{\frac{1}{2}}{\frac{1}{2}} {\frac{1}{2}}{\frac{1}{2}} \right\}, \\
 \bn_1 G &= \left\{\chr{\frac{1}{2}}{0} {0}{0}, \chr {0}{0}{0}{\frac{1}{2}} , \chr{\frac{1}{2}}{\frac{1}{2}}
{\frac{1}{2}}{0},
  \chr{0}{\frac{1}{2}} {\frac{1}{2}}{\frac{1}{2}}  \right\}, \\
 \bn_2 G &= \left\{\chr{0}{\frac{1}{2}} {0}{0}, \chr{\frac{1}{2}}{\frac{1}{2}} {0}{\frac{1}{2}}, \chr
{0}{0}{\frac{1}{2}}{0},
 \chr{\frac{1}{2}}{0} {\frac{1}{2}}{\frac{1}{2}}\right\}, \\
 \bn_3 G &= \left\{\chr{\frac{1}{2}}{\frac{1}{2}} {0}{0}, \chr{0}{\frac{1}{2}} {0}{\frac{1}{2}}, \chr{\frac{1}{2}}{0}
{\frac{1}{2}}{0}, \chr {0}{0}{\frac{1}{2}} {\frac{1}{2}} \right\}.
\end{split}
\]
%
We have the following six identities for the above G\"opel group:

\[
\left \{ \begin{array}{lll}
 \T_2^2 \T_3^2 & = & \T_1^2 \T_4^2 - \T_5^2 \T_{6}^2, \\
 \T_2^4 + \T_3^4 & = & \T_1^4 + \T_4^4 - \T_5^4 - \T_{6}^4, \\
 \T_8^2 \T_7^2 & = & \T_1^2 \T_{5}^2 - \T_4^2 \T_6^2, \\
 \T_8^4 + \T_7^4 & = & \T_1^4 - \T_4^4 + \T_5^4 - \T_{6}^4, \\
 \T_9^2 \T_{10}^2 & = & -\T_1^2 \T_6^2 + \T_4^2 \T_{5}^2, \\
 \T_9^4 + \T_{10}^4 & = & \T_1^4 - \T_4^4 - \T_5^4 + \T_{6}^4. \\
%
\end{array}
\right.
\]
%
%
\end{description}
 From now on, we consider $\T_1, \, \T_2, \, \T_3,$ and $\T_4$ as the fundamental theta constants.
\medskip
\subsection{Inverting the Moduli Map} Let $\lambda_i,$ $i=1, \dots, n,$ be branch points of the genus
$g$ smooth curve $\X.$ Then the moduli map is a map from the configuration space $\Lambda$ of ordered $n$ distinct
points on $\P^1$ to the Siegel upper half space $\HS_g.$ In this section, we determine the branch points of genus 2
curves as functions of theta characteristics. The following lemma describes these relations using Thomae's formula. The
identities are known as Picard's formulas. We will formulate a somewhat different proof for Picard's lemma.
\begin{lem}[Picard] Let a genus 2 curve be given by
\begin{equation} \label{Rosen2}
Y^2=X(X-1)(X-\lambda)(X-\mu)(X-\nu).
\end{equation} Then, $\lambda, \mu, \nu$   can be written as follows:
%
\begin{equation}\label{Picard}
\l = \frac{\T_1^2\T_3^2}{\T_2^2\T_4^2}, \quad \mu = \frac{\T_3^2\T_8^2}{\T_4^2\T_{10}^2}, \quad \nu =
\frac{\T_1^2\T_8^2}{\T_2^2\T_{10}^2}.
\end{equation}
\end{lem}
\begin{proof}
There are several ways to relate $\lambda, \mu, \nu$ to theta constants, depending on the ordering of the branch points
of the curve. Let $B = \{\nu, \mu,\lambda, 1,0\}$ be the branch points of the curve in this order and $U = \{\nu,
\lambda, 0\}$ be the set of odd branch points. Using Lemma~\ref{Thomae}, we have the following set of equations of
theta constants and branch points:
\begin{equation}\label{Thomaeg=2}
\begin{array}{ll}
\T_1^4 = A \, \nu \l (\mu -1) (\nu - \l), &
\T_2^4  = A \, \mu (\mu -1) ( \nu - \l), \\
\T_3^4  = -A \, \mu  \l (\mu - \l) (\nu - \l), &
\T_4^4  = -A\, \nu (\nu - \l) (\mu - \l), \\
\T_5^4  = A \, \l \mu (\nu - 1) ( \nu - \mu),&
\T_6^4  = -A \, (\nu - \mu) (\nu -\l) ( \mu -\l),\\
\T_7^4  =  -A \, \mu (\nu -1) ( \l -1) (\nu - \l), &
\T_8^4  = -A \, \mu \nu (\nu - \mu) (\l -1), \\
\T_9^4  = A \, \nu ( \mu -1) (\l - 1) (\mu - \l), &
\T_{10}^4  = -A \, \l ( \l - 1) (\nu - \mu) \\
\end{array}
\end{equation}
where $A$ is a constant.
By choosing appropriate equations from the set Eq.~\eqref{Thomaeg=2} we have the following:

\[ \l^2 =\left(\frac{\T_1^2\T_3^2}{\T_2^2\T_4^2}\right)^2, \quad \mu^2 = \left(\frac{\T_3^2\T_8^2}{\T_4^2\T_{10}^2}\right)^2, \quad \nu^2
=\left(\frac{\T_1^2\T_8^2}{\T_2^2\T_{10}^2}\right)^2.
\]
Each value for $(\l, \mu, \nu )$ gives isomorphic genus 2 curves. Hence, we can choose
\[ \l = \frac{\T_1^2\T_3^2}{\T_2^2\T_4^2}, \quad \mu = \frac{\T_3^2\T_8^2}{\T_4^2\T_{10}^2}, \quad \nu =
\frac{\T_1^2\T_8^2}{\T_2^2\T_{10}^2}.\]
This completes the proof.
\end{proof}
\subsection{Automorphism Groups of Curves}
Let $\X$ be a genus 2 curve defined over an algebraically closed field $k$ of characteristic zero. We denote its
function field by $K:=k (\X)$ and $Aut(\X)=Aut(K/k)$ is the automorphism group of $\X$. In any characteristic different
from 2, the automorphism group $Aut(\X)$ is isomorphic to one of the groups given by the following lemma.
%
\begin{lem}\label{thm1} The  automorphism group $G$ of a
genus 2 curve $\X$ in characteristic $\ne2$ is isomorphic to \ $C_2$, $C_{10}$, $V_4$, $D_8$, $D_{12}$, $C_3 \sem D_8$,
$ GL_2(3)$, or $2^+S_5$. The case $G \iso 2^+S_5$ occurs only in characteristic 5. If $G \iso \Z_3 \sem D_8$ (resp., $
GL_2(3)$), then $\X$ has equation $Y^2=X^6-1$ (resp., $Y^2=X(X^4-1)$). If $G \iso C_{10}$, then $\X$ has equation
$Y^2=X^6-X$.
\end{lem}
For the proof of the above lemma and the description of each group see \cite{S7}. For the rest of this chapter, we
assume that $char(k)=0.$
One of the main goals of Section 2.4 is to describe each locus of genus 2 curves with fixed automorphism group in terms
of the fundamental theta constants.
We have the following lemma.
\begin{lem}\label{possibleCurve}
Every genus two curve can be written in the form:
\[
y^2 = x \, (x-1) \, \left(x - \frac {\T_1^2 \T_3^2} {\T_2^2  \T_4^2}\right)\, \left(x^2 \, -   \frac{\T_2^2 \, \T_3^2 +
\T_1^2 \, \T_4^2} { \T_2^2 \, \T_4^2} \cdot    \a  \, x + \frac {\T_1^2 \T_3^2} {\T_2^2 \T_4^2} \, \a^2 \right),
\]
where $\a = \frac {\T_8^2} {\T_{10}^2}$ can be given in terms of $\, \, \T_1, \T_2, \T_3,$ and $\T_4$,

\[
  \a^2 + \frac {\T_1^4 + \T_2^4 - \T_3^4 - \T_4^4}{\T_3^2 \T_4^2 - \T_1^2 \T_2^2  } \, \a + 1 =0.
\]
Furthermore,  if $\alpha = {\pm} 1$ then $V_4 \embd Aut(\X)$.
\end{lem}
\proof  Let us write the genus 2 curve in the following form: $$Y^2 = X (X-1) (X-\l) (X-\mu) (X-\nu)$$ where $\l ,\mu
,\nu$ are given by Eq. \eqref{Picard}. Let $\a := \frac {\T_8^2} {\T_{10}^2}$.  Then,
\[
\begin{array}{ll}
\mu =  \frac{\T_3^2}{\T_4^2}\, \a,  &  \nu =  \frac{\T_1^2}{\T_2^2} \, \a.
\end{array}
\]
Using the following two identities,
\begin{equation}\label{Frobenius}
\begin{split}
\T_8^4 + \T_{10}^4 &= \T_1^4+\T_2^4-\T_3^4-\T_4^4, \\
\T_8^2 \T_{10}^2 &= \T_1^2 \T_2^2 - \T_3^2 \T_4^2
\end{split}
\end{equation}
%
we have
\begin{equation}\label{rootof}
 \a^2 +    \frac {\T_1^4 + \T_2^4 - \T_3^4 - \T_4^4}{\T_3^2 \T_4^2 - \T_1^2 \T_2^2  } \,           \a + 1 = 0.
\end{equation}
If $\a=\pm 1$ then $\mu \nu = \l$.    It is well known that this implies that the genus 2 curve has an elliptic
involution. Hence,  $V_4 \embd Aut(\X)$.

\endproof
\begin{rem} {i)} From the above we have that $\T_8^4=\T_{10}^4$ implies that $V_4 \embd Aut(\X)$. Lemma \ref{lemma1}
determines a necessary and equivalent statement when $V_4 \embd Aut(\X)$.

{ii)}  The last part of Lemma 2.4 shows that if $\T_8^4=\T_{10}^4$, then all coefficients of the genus 2 curve are
given as rational functions of the four fundamental theta functions. Such fundamental theta functions determine the
field of moduli of the given curve. Hence, the curve is defined over its field of moduli.
\end{rem}

\begin{cor}
Let $\X$ be a genus 2 curve which has an elliptic involution. Then $\X$ is defined over its field of moduli.
\end{cor}
This was the main result of  \cite{Ca}.
\subsection{Describing the Locus of Genus Two Curves with Fixed Automorphism Group by Theta Constants}
The locus $\L_2$ of genus 2 curves $\X$ which have an elliptic involution is a closed subvariety of $\mathcal M_2$. Let
$W= \{\a_1, \a_2, \b_1, \b_2, \gamma_1, \gamma_2 \}$ be the set of roots of the binary sextic, and $A$ and $B$ be
subsets of $W$ such that $W=A \cup B$ and $| A \cap B | =2$. We define the cross ratio of the two pairs $z_1,z_2 ;
z_3,z_4 $ by
$$(z_1,z_2 ; z_3,z_4) = \frac{z_1 ; z_3,z_4}{z_2 ; z_3,z_4} = \frac{z_1-z_3}{z_1-z_4} : \frac{z_2-z_3}{z_2-z_4}.$$
Take $ A = \{ \a_1, \a_2, \b_1 , \b_2\}$ and $B = \{ \gamma_1, \gamma_2, \b_1, \b_2\}$.
Jacobi \cite{Krazer} gives a   description of $\L_2$ in terms of the cross ratios of the elements of $W$:
$$ \frac {\a_1-\b_1} {\a_1-\b_2} : \frac {\a_2-\b_1} {\a_2-\b_2}= \frac {\gamma_1-\b_1} {\gamma_1-\b_2} : \frac
{\gamma_2-\b_1} {\gamma_2-\b_2}.
$$
We recall that the following identities hold for cross ratios:
\[
(\a_1,\a_2\,;\b_1,\b_2)=(\a_2,\a_1;\b_2,\b_1)=(\b_1,\b_2;\a_1,\a_2)=(\b_2,\b_1;\a_2,\a_1)
\]
and
\[
(\a_1,\a_2;\infty,\b_2)=(\infty,\b_2;\a_1,\a_2)=(\b_2;\a_2,\a_1).
\]
Next, we use this result to determine relations among theta functions for a genus 2 curve in the locus $\L_2$. Let $\X$
be any genus 2 curve given by the equation
$$Y^2=X(X-1)(X-a_1)(X-a_2)(X-a_3).$$
We take $\infty \in A \cap B$. Then there are five cases for $\a \in A \cap B $, where $\a $ is an element of the set $
 \{0,1, a_1, a_2, a_3\}$. For each of these cases there are three possible relationships for cross ratios as described below:\\
\noindent {i)} $A \cap B = \{ 0, \infty\}$: The possible cross ratios are
\begin{description}
\item $(a_1,1;\infty,0) = (a_3,a_2;\infty,0), \quad \quad (a_2,1;\infty,0) = (a_1,a_3;\infty,0),$ \\
\item $(a_1,1;\infty,0) =(a_2,a_3;\infty,0).$
\end{description}
\noindent {ii)} $A \cap B = \{ 1, \infty\}$: The possible cross ratios are
\begin{description}
\item $(a_1,0;\infty,1)=(a_2,a_3;\infty,1),\quad \quad (a_1,0;\infty,1)=(a_3,a_2;\infty,1),$\\
\item $(a_2,0;\infty,1)=(a_1,a_3;\infty,1).$
\end{description}
\noindent {iii)} $A \cap B = \{ a_1, \infty\}$: The possible cross ratios are
\begin{description}
\item $(1,0;\infty,a_1)=(a_3,a_2;\infty,a_1),\quad \quad (a_2,0;\infty,a_1)=(1,a_3;\infty,a_1),$\\
\item $(1,0;\infty,a_1)=(a_2,a_3;\infty,a_1).$
\end{description}
\noindent {iv)} $A \cap B= \{ a_2, \infty\}$: The possible cross ratios are
\begin{description}
\item $(1,0;\infty,a_2)=(a_1,a_3;\infty,a_2),\quad \quad (1,0;\infty,a_2)=(a_3,a_1;\infty,a_2),$\\
\item $(a_1,0;\infty,a_2)=(1,a_3;\infty,a_2).$
\end{description}
\noindent {v)} $A \cap B = \{ a_3, \infty\}$: The possible cross ratios are
\begin{description}
\item $(a_1,0;\infty,a_3)=(1,a_2;\infty,a_3),\quad \quad (1,0;\infty,a_3)=(a_2,a_1;\infty,
a_3),$\\
\item $(1,0;\infty,a_3)=(a_1,a_2;\infty,a_3).$
\end{description}

We summarize these relationships in Table 2.1.

\begin{lem}\label{lemma1}
Let $\X$ be a genus 2 curve. Then  $Aut(\X)\iso V_4$ if and only if the theta functions of $\X$ satisfy

\begin{equation}\label{V_4locus1}
\begin{split}
(\T_1^4-\T_2^4)(\T_3^4-\T_4^4)(\T_8^4-\T_{10}^4)
(-\T_1^2\T_3^2\T_8^2\T_2^2-\T_1^2\T_2^2\T_4^2\T_{10}^2+\T_1^4\T_3^2\T_{10}^2+ \T_3^2\T_2^4\T_{10}^2)\\
(\T_3^2\T_8^2\T_2^2\T_4^2-\T_2^2\T_4^4\T_{10}^2+\T_1^2\T_3^2\T_4^2\T_{10}^2-\T_3^4\T_2^2\T_{10}^2)
(-\T_8^4\T_3^2\T_2^2+\T_8^2\T_2^2\T_{10}^2\T_4^2\\
+\T_1^2\T_3^2\T_8^2\T_{10}^2
-\T_3^2\T_2^2\T_{10}^4)(-\T_1^2\T_8^4\T_4^2-\T_1^2\T_{10}^4\T_4^2+\T_8^2\T_2^2\T_{10}^2\T_4^2+\T_1^2\T_3^2\T_8^2\T_{10}^2)\\
(-\T_1^2\T_8^2\T_3^2\T_4^2+\T_1^2\T_{10}^2\T_4^4
+\T_1^2\T_3^4\T_{10}^2-\T_3^2\T_2^2\T_{10}^2\T_4^2)(-\T_1^2\T_8^2\T_2^2\T_4^2+\T_1^4\T_{10}^2\T_4^2\\
-\T_1^2\T_3^2\T_2^2\T_{10}^2+\T_2^4\T_4^2\T_{10}^2) (-\T_8^4\T_2^2\T_4^2
+\T_1^2\T_8^2\T_{10}^2\T_4^2-\T_2^2\T_{10}^4\T_4^2+\T_3^2\T_8^2\T_2^2\T_{10}^2)\\
(\T_1^4\T_8^2\T_4^2-\T_1^2\T_2^2\T_4^2\T_{10}^2-\T_1^2\T_3^2\T_8^2\T_2^2+\T_8^2\T_2^4\T_4^2)
 (\T_1^4\T_3^2\T_8^2-\T_1^2\T_8^2\T_2^2\T_4^2\\
-\T_1^2\T_3^2\T_2^2\T_{10}^2+\T_3^2\T_8^2\T_2^4)(\T_1^2\T_8^4\T_3^2-\T_1^2\T_8^2\T_{10}^2\T_4^2+\T_1^2\T_3^2\T_{10}^4
-\T_3^2\T_8^2\T_2^2\T_{10}^2)\\
(\T_1^2\T_8^2\T_4^4-\T_1^2\T_3^2\T_4^2\T_{10}^2+\T_1^2\T_3^4\T_8^2-\T_3^2\T_8^2\T_2^2\T_4^2)
  =0.
\end{split}
\end{equation}
\end{lem}
However, we are unable to determine a similar result for cases $D_8$ or  $D_{12}$ by this argument. Instead, we will
use the invariants of genus 2 curves and a more computational approach. In the process, we will offer a different proof
for the lemma above.
%
\begin{center}
\begin{table}\label{tab_1}
\caption{Relation of theta functions and cross ratios}
\begin{tabular}{|c|c|c|c|}
\hline & Cross ratio &  $f(a_1,a_2,a_3)=0$ & theta constants \tabularnewline[8pt]
\hline
1& $(1,0;\infty,a_1)=(a_3,a_2;\infty,a_1)$ & $a_1a_2+a_1-a_3a_1-a_2$ &
$-\T_1^2\T_3^2\T_8^2\T_2^2-\T_1^2\T_2^2\T_4^2\T_{10}^2+$\\

& & &$\T_1^4\T_3^2\T_{10}^2+ \T_3^2\T_2^4\T_{10}^2$ \tabularnewline[4pt]
\hline
2 & $(a_2,0;\infty,a_1)=(1,a_3;\infty,a_1)$ & $a_1a_2-a_1+a_3a_1-a_3a_2$ &
$\T_3^2\T_8^2\T_2^2\T_4^2-\T_2^2\T_4^4\T_{10}^2+$\\
& && $\T_1^2\T_3^2\T_4^2\T_{10}^2-\T_3^4\T_2^2\T_{10}^2$ \tabularnewline[4pt]
\hline
3& $(1,0;\infty,a_1)=(a_2,a_3;\infty,a_1)$ & $a_1a_2-a_1-a_3a_1+a_3$ & $-\T_8^4\T_3^2\T_2^2+\T_8^2\T_2^2\T_{10}^2\T_4^2+$ \\
& &&$\T_1^2\T_3^2\T_8^2\T_{10}^2-\T_3^2\T_2^2\T_{10}^4$ \tabularnewline[4pt]
\hline
4& $(1,0;\infty,a_2)=(a_1,a_3;\infty,a_2)$ & $a_1a_2-a_2-a_3a_2+a_3$ & $-\T_1^2\T_8^4\T_4^2-\T_1^2\T_{10}^4\T_4^2+$ \\
& &&$\T_8^2\T_2^2\T_{10}^2\T_4^2+\T_1^2\T_3^2\T_8^2\T_{10}^2$ \tabularnewline[4pt]
\hline
5& $(1,0;\infty,a_2)=(a_3,a_1;\infty,a_2)$ & $a_1a_2-a_1+a_2-a_3a_2$&$-\T_1^2\T_8^2\T_3^2\T_4^2+\T_1^2\T_{10}^2\T_4^4+$\\
&&&$\T_1^2\T_3^4\T_{10}^2-\T_3^2\T_2^2\T_{10}^2\T_4^2$ \tabularnewline[4pt]
\hline
6 & $(a_1,0;\infty,a_2)=(1,a_3;\infty,a_2)$ & $a_1a_2-a_3a_1-a_2+a_3 a_2$ &
$-\T_1^2\T_8^2\T_2^2\T_4^2+\T_1^4\T_{10}^2\T_4^2-$ \\
& &&$\T_1^2\T_3^2\T_2^2\T_{10}^2+\T_2^4\T_4^2\T_{10}^2$ \tabularnewline[4pt]
\hline
7&$(a_1,0;\infty,a_3)=(1,a_2;\infty,a_3)$ & $a_1a_2-a_3a_1-a_3a_2+a_3$ &
$-\T_8^4\T_2^2\T_4^2+\T_1^2\T_8^2\T_{10}^2\T_4^2-$ \\
&&&$\T_2^2\T_{10}^4\T_4^2+\T_3^2\T_8^2\T_2^2\T_{10}^2$ \tabularnewline[4pt]
\hline
8&$(1,0;\infty,a_3)=(a_2,a_1;\infty, a_3)$&$a_3a_1-a_1-a_3a_2+a_3$ & $\T_8^4-\T_{10}^4$\tabularnewline[4pt]
\hline
9&$(1,0;\infty,a_3)=(a_1,a_2;\infty,a_3)$ &$ a_3a_1+a_2-a_3-a_3a_2$ &
$\T_1^4\T_8^2\T_4^2-\T_1^2\T_2^2\T_4^2\T_{10}^2-$ \\
&&&$\T_1^2\T_3^2\T_8^2\T_2^2+\T_8^2\T_2^4\T_4^2$ \tabularnewline[4pt]
\hline
10&$(a_1,0;\infty,1)=(a_2,a_3;\infty,1)$ & $-a_1+a_3a_1+a_2-a_3 $&$\T_1^4\T_3^2\T_8^2-\T_1^2\T_8^2\T_2^2\T_4^2-$\\
&&&$\T_1^2\T_3^2\T_2^2\T_{10}^2+\T_3^2\T_8^2\T_2^4$ \tabularnewline[4pt]
\hline
11& $(a_1,0;\infty,1)=(a_3,a_2;\infty,1)$ &$ a_1a_2-a_1-a_2+a_3$&$\T_1^2\T_8^4\T_3^2-\T_1^2\T_8^2\T_{10}^2\T_4^2+$\\
&&&$\T_1^2\T_3^2\T_{10}^4-\T_3^2\T_8^2\T_2^2\T_{10}^2$ \tabularnewline[4pt]
\hline
12& $(a_2,0;\infty,1)=(a_1,a_3;\infty,1)$&$a_1-a_2+a_3a_2-a_3$&$\T_1^2\T_8^2\T_4^4-\T_1^2\T_3^2\T_4^2\T_{10}^2+$ \\
&&&$\T_1^2\T_3^4\T_8^2-\T_3^2\T_8^2\T_2^2\T_4^2$ \tabularnewline[4pt]
\hline
13&$(a_1,1;\infty,0) = (a_3,a_2;\infty,0)$ & $a_1a_2-a_3$ & $\T_8^4-\T_{10}^4 $ \tabularnewline[4pt]
\hline
14& $(a_2,1;\infty,0) = (a_1,a_3;\infty,0)$ & $a_1-a_3a_2$ &$ \T_3^4-\T_4^4$ \tabularnewline[4pt]
\hline
15&$(a_1,1;\infty,0) = (a_2,a_3;\infty,0)$ & $a_3a_1-a_2$ & $ \T_1^4-\T_2^4  $ \tabularnewline[4pt] \hline
\end{tabular}
\end{table}
\end{center}
%
\begin{lem}
{i)} The locus $\L_2$ of genus 2 curves $\X$ which have a degree 2 elliptic subcover is a closed subvariety of
$\mathcal M_2$. The equation of $\L_2$ is given by
\[
\begin{split}\label{eq_L2_J}
8748J_{10}J_2^4J_6^2- 507384000J_{10}^2J_4^2J_2-19245600J_{10}^2J_4J_2^3-6912J_4^3J_6^{34}\\
-592272J_{10}J_4^4J_2^2 +77436J_{10}J_4^3J_2^4 -3499200J_{10}J_2J_6^3+4743360J_{10}J_4^3J_2J_6\\
-870912J_{10}J_4^2J_2^3J_6
+3090960J_{10}J_4J_2^2J_6^2 -78J_2^5J_4^5-125971200000J_{10}^3\\
-81J_2^3J_6^4+1332J_2^4J_4^4J_6 +384J_4^6J_6+41472J_{10}J_4^5+159J_4^6J_2^3  &\\
 -47952J_2J_4J_6^4
 +104976000J_{10}^2J_2^2J_6-1728J_4^5J_2^2J_6+6048J_4^4J_2J_6^2\\
   \end{split}
\]
\begin{equation}
\begin{split}
 -9331200J_{10}J_4^2J_6^2 -J_2^7J_4^4 +12J_2^6J_4^3J_6+29376J_2^2J_4^2J_6^3-8910J_2^3J_4^3J_6^2\\
 -2099520000J_{10}^2J_4J_6+31104J_6^5 -5832J_{10}J_2^5J_4J_6  -54J_2^5J_4^2J_6^2 \\
-236196J_{10}^2J_2^5-80J_4^7J_2 +108J_2^4J_4J_6^3 +972J_{10}J_2^6J_4^2 = & 0.
\end{split}
\end{equation}

{ii)} The locus  of genus 2 curves $\X$ with $Aut(\X)\iso D_8$ is given by the equation of $\L_2$  and
\begin{equation}
\label{D_8_locus} 1706J_4^2J_2^2+2560J_4^3+27J_4J_2^4-81J_2^3J_6-14880J_2J_4J_ 6+28800J_6^2 =0.
\end{equation}
%
{iii)} The locus  of genus 2 curves $\X$ with $Aut(\X)\iso D_{12}$ is
%
\begin{equation}
\label{D_12_locus}
\begin{split}
-J_4J_2^4+12J_2^3J_6-52J_4^2J_2^2+80J_4^3+960J_2J_4J_6-3600
J_6^2 &=0, \\
864J_{10}J_2^5+3456000J_{10}J_4^2J_2-43200J_{10}J_4J_2^3-
2332800000J_{10}^2\\
-J_4^2J_2^6 -768J_4^4J_2^2+48J_4^3J_2^4+4096J_4^5 &=0.\\
\end{split}
\end{equation}

\end{lem}
%
Our goal is to express each of the above loci in terms of the theta characteristics. We obtain the following result.
\begin{thm}\label{theorem1}
 Let $\X$ be a genus 2 curve. Then the following hold:

{i)}  $Aut(\X)\iso V_4$ if and only if the relations of theta functions given Eq.~\eqref{V_4locus1} holds.

{ii)} $Aut(\X)\iso D_8$ if and only if the Eq. I in \cite{web}  is satisfied.

{iii)} $Aut(\X)\iso D_{12}$ if and only if the Eq. II and Eq. III in \cite{web} are satisfied.
\end{thm}
\proof Part {i)} of the theorem is Lemma~\ref{lemma1}. Here we give a somewhat different proof. Assume that $\X$ is a
genus 2 curve with equation
$$Y^2=X(X-1)(X-a_1)(X-a_2)(X-a_3)$$
whose classical invariants satisfy Eq.~\eqref{eq_L2_J}. Expressing the classical invariants of $\X$ in terms of $a_1,
a_2, a_3$, substituting them into \eqref{eq_L2_J}, and factoring the resulting equation yields
\medskip
\begin{equation}
\begin{split}\label{L2_factored}
(a_1a_2-a_3)^2 (a_1-a_3a_2)^2  (a_3a_1-a_2)^2 (a_1a_2-a_2-a_3a_2+a_3)^2\\
(a_3a_1+a_2-a_3-a_3a_2)^2(-a_1+a_3a_1+a_2-a_3)^2(a_1a_2-a_1-a_2+a_3)^2\\
(a_1a_2-a_1+a_3a_1-a_3a_2)^2(a_1a_2-a_3a_1-a_3a_2+a_3)^2\\
(a_3a_1-a_1-a_3a_2+a_3)^2(a_1a_2+a_1-a_3a_1-a_2)^2\\
(a_1a_2-a_1-a_3a_1+a_3)^2 (a_1a_2-a_1+a_2-a_3a_2)^2\\
(a_1-a_2+a_3a_2-a_3)^2(a_1a_2-a_3a_1-a_2+a_3 a_2)^2  =&\, 0.
\end{split}
\end{equation}
%
It is no surprise that we get the 15 factors of Table 2.1. The relations of theta constants follow from Table 2.1.

{ii)} Let $\X$ be a genus 2 curve which has an elliptic involution. Then $\X$ is isomorphic to a curve with the
equation
$$Y^2=X(X-1)(X-a_1)(X-a_2)(X-a_1 a_2).$$
If $\Aut (\X) \iso D_8$ then the $SL_2 (k)$-invariants of such curve must satisfy Eq.~\eqref{D_8_locus}. Then, we get
the equation in terms of $a_1$ and $a_2$.
By writing the relation $a_3 = a_1 a_2$ in terms of theta constants, we get $\T_4^4 = \T_3^4$. All the results above
lead to part ii) of the theorem.
{iii)} The proof of this part is similar to part ii).
\endproof
We express the conditions of the previous lemma in terms of the fundamental theta constants only.
\begin{lem}
Let $\X$ be a genus 2 curve. Then we have the following:
%

\noindent {i)} $V_4 \hookrightarrow Aut(\X)$ if and only if the fundamental theta constants of $\X$ satisfy
\begin{equation}\label{V_4locus2}
\begin{split}
( \theta_{{3}}^4-\theta_{{4}}^4 )  (\theta_{{1}}^4 -\theta_{{3}}^4 ) ( \theta_{{2}}^4-\theta_{{4}}^4 ) ( \theta_{{1}}^4
-\theta_{{4}}^4 )  ( \theta_{{3}}^4-\theta_{{2}}^4 ) ( \theta_{{1}}^4- \theta_{{2}}^4 ) \\
( -\theta_{{4}}^2+\theta_{{3}}^2 +\theta_{{1}}^2-\theta_{{2}}^2 )(
 \theta_{{4}}^2-\theta_{{3}}^2+\theta_{{1}}^2-\theta_{{2}}^2
 )  ( -\theta_{{4}}^2-\theta_{{3}}^2+\theta_{{2}}^2+\theta_{{
1}}^2 ) \\
( \theta_{{4}}^2+\theta_{{3}}^2+\theta_{{2}}^2+\theta_ {{1}}^2 ) ( {\theta_{{1}}}^{4}{\theta_{{2}}}^{4}+
{\theta_{{3}}}^{4}{\theta_{{2}}}^{4} +{\theta_{{1}}}^{4}{\theta_{{3}}}^{4} -2\,\theta_{{1}}^2\theta_{{2}}^2\theta
_{{3}}^2\theta_{{4}}^2 )\\
 \left( -{\theta_{{3}}}^{4}{\theta_{{2}}}^{4}-{
\theta_{{2}}}^{4}{\theta_{{4}}}^{4}-{\theta_{{3}}}^{4}{\theta_{{4}}}^{4} + 2\,\theta_{{1}}^2\theta_{{2}}^2\theta_
{{3}}^2\theta_{{4}}^2 \right)( {\theta_{{2}}}^{4}{\theta_{{4}}}^{4} +{\theta_{{1}}}^{4}{\theta _{{2}}}^{4} +{
\theta_{{1}}}^{4}{\theta_{{4}}}^{4}\\
-2\,\theta_{{1}}^2\theta_{{2}}^2\theta_{{3}}^2\theta_{{4}}^2 )  \left(
{\theta_{{1}}}^{4}{\theta_{{4}}}^{4}+{\theta_{{3}}}^{4}{\theta_{{4}}}^{4}+{\theta_{{1}}}^{4}{\theta_{ {3}}}^{4}
-2\,\theta_{{1 }}^2\theta_{{2}}^2\theta_{{3}}^2\theta_{{4}}^2\right)  = & \, 0.\\
\end{split}
\end{equation}
%
\noindent {ii)} $D_8 \hookrightarrow Aut(\X)$ if and only if the fundamental theta constants of $\X$ satisfy Eq.(3) in
\cite{Sh}.

\noindent {iii)} $D_6 \hookrightarrow Aut(\X)$ if and only if the fundamental theta constants of $\X$ satisfy Eq.(4) in
\cite{Sh}.

\end{lem}

\proof Notice that Eq.~\eqref{V_4locus1} contains only $\T_1, \T_2, \T_3, \T_4, \T_8$ and $\T_{10}.$ Using
Eq.~\eqref{Frobenius}, we can eliminate $\T_8$ and $\T_{10}$ from Eq.~\eqref{V_4locus1}.
The $J_{10}$ invariant of any genus two curve is  given by the following in terms of theta constants:
\[ J_{10} = \frac{\T_1^{12} \T_3^{12}}{\T_2^{28} \T_4^{28} \T_{10}^{40}} \, (\T_1^2\T_2^2 - \T_3^2 \T_4^2)^{12} (\T_1^2\T_4^2 - \T_2^2
\T_3^2)^{12} (\T_1^2\T_3^2 - \T_2^2 \T_4^2)^{12}.\] Since $J_{10} \neq 0,$ the factors $(\T_1^2\T_2^2 - \T_3^2 \T_4^2),
(\T_1^2\T_4^2 - \T_2^2 \T_3^2)$ and $(\T_1^2\T_3^2 - \T_2^2 \T_4^2)$ cancel in the equation of the $V_4$ locus. The
result follows from Theorem ~\ref{theorem1}.  The proof of part ii) and iii) is similar and we avoid details.
\endproof
\begin{rem}
For part ii) and iii), the equations are lengthy and we don't show them here. But by using the extra conditions $\T_4^2
= \T_3^2$ or $\T_4^2 = - \T_3^2$, we could simplify the equation of the $D_8$ locus as follows:

\noindent {i)}When $\T_4^2 = \T_3^2$, we have
%
\begin{equation}
\begin{split}
 ( \theta_{{1}}^4-\theta_{{2}}^4 ) ( \theta_{{1}}^2\theta_{{2}}^2-{\theta_{{3}}}^{4} )  ( {
\theta_{{2}}}^{2}+{\theta_{{1}}}^{2}+2\,{\theta_{{3}}}^{2} )  (
{\theta_{{2}}}^{2}+{\theta_{{1}}}^{2}-2\,{\theta_{{3}}}^{2}
 )  (
2\,{\theta_{{1}}}^{4} \\
-2\,{ \theta_{{1}}}^{2}{\theta_{{2}}}^{2} +{\theta_{{3}}}^{4} )
 ( -2\,{\theta_{{ 2}}}^{4}-{\theta_{{3}}}^{4}+2\,{\theta_{{1}}}^{2}{\theta_{{2}}}^{2}
 )  ( -10\,{\theta_{{1}}}^{4}{\theta_{{2}}}^{12}{\theta
_{{3}}}^{8}\\
+206\,{\theta_{{1}}}^{4}{\theta_{{2}}}^{4}{\theta_{{3}}}^{ 16} +8\,{\theta_{{1}}}^{8}{ \theta_{{2}}}^{16}
-34\,{\theta_{{1}}}^{4}{\theta_{{2}}}^{8}{\theta_{{3}}}^{12} -126\,{
\theta_{{1}}}^{2}{\theta_{{2}}}^{6}{\theta_{{3}}}^{16}\\
+18\,{\theta_{{1 }}}^{2}{\theta_{{2}}}^{10}{\theta_{{3}}}^{12}+27\,{\theta_{{1}}}^{8}{\theta_{{3}}}^{16}
-132\,{\theta_{{1}}}^{8}{\theta_{{2}}}^{8}{\theta_{{3}}}^{8}
 -34\,{\theta_{{1}
}}^{8}{\theta_{{2}}}^{4}{\theta_{{3}}}^{12}\\
-16\,{\theta_{{1}}}^{8}{ \theta_{{3}}}^{4}{\theta_{{2}}}^{12} -16\,{\theta_{{1}}}^{6}{\theta_{{2
}}}^{14}{\theta_{{3}}}^{4} -126\,{\theta_{{1}}}^{6}{\theta_{{2}}}^{2}{ \theta_{{3}}}^{16}
+24\,{\theta_{{1}}}^{6}{\theta_{{2}}}^{6}{\theta_{{3 }}}^{12}\\
+68\,{\theta_{{1}}}^{6}{\theta_{{2}}}^{10}{\theta_{{3}}}^{8} -
24\,{\theta_{{1}}}^{12}{\theta_{{2}}}^{12}+8\,{\theta_{{2}}}^{8}{\theta _{{1}}}^{16} -10\,{\theta_{{1}}}^{12}{
\theta_{{3}}}^{8}{\theta_{{2}}}^{4} \\
-16\,{\theta_{{1}}}^{12}{\theta_{{3 }}}^{4}{\theta_{{2}}}^{8} +88\,{\theta_{{1}}}^{10}{\theta_{{3}}}^{4}{
\theta_{{2}}}^{10} +18\,{\theta_{{1}}}^{10}{\theta_{{2}}}^{2}{\theta_{{ 3}}}^{12}
+68\,{\theta_{{1}}}^{10}{\theta_{{3}}}^{8}{\theta_{{2}}}^{6} \\
+27\,{\theta_{{2}}}^{8}{\theta_{{3}}}^{16}
-16\,{\theta_{{1}}}^{14}{\theta_{{2}}}^{6}{\theta_{{3}}}^{ 4} )  = &  0.
\end{split}
\end{equation}
\noindent {ii)} When $\T_4^2 = -\T_3^2$, we have
\begin{equation}
\begin{split}
( \theta_{{1}}^4-\theta_{{2}}^4 )  ( {\theta _{{3}}}^{4}+{\theta_{{1}}}^{2}{\theta_{{2}}}^{2} ) (
-{\theta_{{2}}}^{2}+{\theta_{{1}}}^{2}-2\,{\theta_{{3}}}^{2} )  (
-{\theta_{{2}}}^{2}+{\theta_{{1}}}^{2}+2\,{\theta_{{3}}}^{2 } )\\
({\theta_{{3}}}^{ 4} +2\,{\theta_{{1}}}^{2}{\theta_{{2}}}^{2} +2\,{\theta_{{1}}}^{4} )
 ( 2\,{ \theta_{{2}}}^{4}+{\theta_{{3}}}^{4}+2\,{\theta_{{1}}}^{2}{\theta_{{2} }}^{2})  (
206\,{\theta_{{1}}}^{4}{\theta_{{2}}}^{4}{\theta_{{3
}}}^{16}\\
-10\,{\theta_{{1}}}^{4}{\theta_{{2}}}^{12}{ \theta_{{3}}}^{8}+27\,{\theta_{{2}}}^{8}{\theta_{{3}}}^{16}
 -34\,{\theta_{{1}}}^{4}{\theta_{{2}}}^{8}{\theta_{{3}}}^{12}
+ 126\,{\theta_{{1}}}^{2}{\theta_{{2}}}^{6}{\theta_{{3}}}^{16}\\
-18\,{ \theta_{{1}}}^{2}{\theta_{{2}}}^{10}{\theta_{{3}}}^{12} -68\,{\theta_{{1}}}^{10}{\theta_{{3}}}^{8}{\theta_{{
2}}}^{6} +8\,{\theta_{{1 }}}^{8}{\theta_{{2}}}^{16} +27\,{\theta_{{1}}}^{8}{\theta_{{3}}}^{16} \\
-132\,{\theta_{{1}}}^{8}{\theta_{{2}}}^{8}{\theta_{{3}}}^{8}-34\,{
\theta_{{1}}}^{8}{\theta_{{2}}}^{4}{\theta_{{3}}}^{12} -16\,{\theta_{{1 }}}^{8}{\theta_{{3}}}^{4}{\theta_{{2}}}^{12}
+16\,{\theta_{{1}}}^{6}{ \theta_{{2}}}^{14}{\theta_{{3}}}^{4} \\
+126\,{\theta_{{1}}}^{6}{\theta_{{ 2}}}^{2}{\theta_{{3}}}^{16} -24\,{\theta_{{1}}}^{6}{\theta_{{2}}}^{6}{
\theta_{{3}}}^{12} -68\,{\theta_{{1}}}^{6}{\theta_{{2}}}^{10}{\theta_{{
3}}}^{8}-24\,{\theta_{{1}}}^{12}{\theta_{{2}}}^{12} \\
+16\,{\theta_{{1}}}^{14}{\theta_{{2}}}^{6}{\theta _{{3}}}^{4} -10\,{\theta_{{1}}}
^{12}{\theta_{{3}}}^{8}{\theta_{{2}}}^{4} -16\,{\theta_{{1}}}^{12}{ \theta_{{3}}}^{4}{\theta_{{2}}}^{8}
-88\,{\theta_{{1}}}^{10}{\theta_{{3 }}}^{4}{\theta_{{2}}}^{10}\\
-18\,{\theta_{{1}}}^{10}{\theta_{{2}}}^{2}{ \theta_{{3}}}^{12} +8\,{\theta_{{2}}}^{ 8}{\theta_{{1}}}^{16} ) = & 0.
\end{split}
\end{equation}
%
\end{rem}
Define the following as\\
%
\[A = (\frac{\T_2}{\T_1})^4,  \quad  B = (\frac{\T_3}{\T_1})^4,  \quad C=(\frac{\T_4}{\T_1})^4,
\quad D = (\frac{\T_8}{\T_1})^4,  \quad E =(\frac{\T_{10}}{\T_1})^4.
\]
%
Using the two identities given by Eq.~\eqref{Frobenius}, we have
%
\[
\begin{split}
1+A-B-C-D-E & = 0, \\
{A}^{2}-2\,  D EA+2\,BCA+{C}^{2}{B}^{2}-2\,  D ECB+{D}^{2}{E}^{2} & =0.
\end{split}
\]
Then we formulate the following lemma.
\def\embd{\hookrightarrow}
\begin{lem}
Let $\X$ be a genus 2 curve. Then $V_4 \embd Aut(\X)$ if and only if the theta constants of $\X$ satisfy
\begin{equation}
\begin{split}
  ( B-A )  ( A-C )  ( B-C )  ( 1-A )  ( 1-B )( 1-C )( 1-2\,C+2\,A  +{A}^{2}{C}^{2}
  \\
  -4\,  D E
  -AC -2\,{A}^{2}BC
  +2\,A  D EBC+A{B}^{2}+ D EBC  +AD EB  -{A}^{2}\\
  +4\,ABC
   -2\,A{B}^{2}{C}^{2}
  -{A}^{2}B+A  D E
  -{B}^{2}{C}^{2}-2\,B{C}^{2}+{B}^{2}C ) ( - D EBC\\
  -4\,ABC
    +{B}^{2}{C}^{2}  +AC+A{B}^{2}C -A  DEB
  +{A}^{2}
  +{A}^{2}C
  +AB{C}^{2}\\
    - D EC
    -2\,A DEC-{A}^{2}{C}^{2}-{A}^{2}BC-A{C}^{2} -A D E )  =&  0.
\end{split}
\end{equation}
\end{lem}


\section{Genus 3 curves}

\subsection{Introduction to Genus 3 Curves} In this section, we focus on genus 3 cyclic curves. The
locus $\L_3$ of genus $3$ hyperelliptic curves with extra involutions is a $3$-dimensional subvariety of $\H_3.$ If $\X
\in \L_3$ then $V_4 \hookrightarrow Aut(\X).$ The normal form of the hyperelliptic genus $3$ curve is given by $$y^3 =
x^8 + a_3 X^6+a_2 x^4 + a_1 x^2 + 1$$ and the dihedral invariants of $\X_3$ are $u_1 = a_1^4 + a_3 ^4,  u_2 =(a_1^2 +
a_3^2) a_2 , u_3 = 2 a_1 a_3.$ The description of the locus of genus $3$ hyperelliptic curves in terms of dihedral
invariants or classical invariants is given in \cite{S2}. We would like to describe the locus of genus $3$
hyperelliptic curves with extra involutions and all its sub loci in terms of theta functions.

The list of groups that occur as automorphism groups of genus $3$ curves has been computed by many authors. We denote
the following groups by $G_1$ and $G_2$: $$G_1 = \langle x,y | x^2, y^6 ,x y x y^4 \rangle, \quad \quad \quad G_2 =
\langle x,y |x^4, y^4, (xy)^2, (x^{-1} y)^2 \rangle .$$
In Table 2, we list all possible hyperelliptic genus 3 algebraic curves; see \cite{MS} for details. In this case
$\Aut(\X)$  has a central subgroup $C$ of order 2 such that the genus of $\X^C$ is zero. In the second column of the
table, the groups which occur as full automorphism groups are given, and the third column indicates the reduced
automorphism group for each case. The dimension $\delta$ of the locus and the equation of the curve are given in the
next two columns. The last column is the GAP identity of each group in the library of small
groups in GAP. Note that $C_2, C_4$ and $C_{14}$ are 
the only groups which don't have extra involutions. Thus, curves with automorphism group $C_2, C_4$ or $C_{14}$ do not
belong to the locus $\L_3$ of genus 3 hyperelliptic curves with extra involutions.

In Table 3, we list the automorphism groups of genus $3$ nonhyperelliptic curves. In the table, the second column
represents the normal cyclic subgroup $C$ such that $g(\X^C) = 0.$ For the last 3 cases in the table, the automorphism
groups of the curves are not normal homocyclic covers of $\P^1.$ The only cyclic curves are curves with automorphism
groups $C_4^2 \rtimes S_3,$ $C_3,$ $C_6,$ $C_9$ and two other groups given by $(16,13)$ and $(48,33)$ in GAP identity.
In this chapter we write the equations of the cyclic curves of genus 3 by using theta constants.
\medskip
\begin{center}
\begin{table}[t!]\label{tablehyperelliptic}
\caption{Genus 3 hyperelliptic curves and their automorphisms}
\begin{tabular}{||c|c|c|c|c|c||}
\hline \hline & &   & & &\\
 &$\Aut(\X)$ & $\bAut(\X)$ & $\, \delta \, $ &equation  $y^2= f(x) $  &  Id.\\
&         &    & &   &\\
\hline \hline  &&&& &  \\
 1 & $C_2$ &$\{1\}$ & 5&$x(x-1)(x^5+ax^4+bx^3+cx^2+dx+e)$&  $(2,1)$ \\
%
& & & &   &\\
2 & $C_2 \times C_2$ &$C_2$& 3&$ x^8 + a_3 x^6 + a_2 x^4 + a_1 x^2 + 1$ &  $(4,2)$\\
3 & $C_4$ & $C_2$ &2&$x(x^2-1)(x^4+ax^2+b)$&  $(4,1)$\\
4 & $C_{14}$ &$C_7$ &0&$x^7-1$ &  $(14,2)$ \\
%
& & & &  &\\
5 & $C_2^3$ &$D_4$ &2&$(x^4+ax^2+1)(x^4+bx^2+1)$   &$(8,5)$\\
6 & $C_2 \times D_8$ &$D_8$ &1&$x^8+ax^4+1$ &    $(16,11)$ \\
7 & $C_2\times C_4$ &$D_4$ & 1&$(x^4-1)(x^4+ax^2+1)$&    $(8,2)$\\
8 & $D_{12}$ &$D_6$  &1&$x(x^6+ax^3+1)$   &$(12,4)$\\
9 & $G_1$ &$D_{12}$ & 0&$x(x^6-1)$    & $(24,5)$\\
10 & $G_2$ &$D_{16}$  &0& $x^8-1$&   $(32,9)$\\
%
%
& & & &   &\\
11 & $C_2 \times S_4$ & $S_4$  &0  & $x^8+14x^2+1$ &   $(48,48)$ \\
&  & & &   &\\
 \hline\hline
\end{tabular}
\end{table}
\end{center}
%
\begin{center}
\begin{table}[h!]\label{tableNonHyper}
\caption{Genus 3 non hyperelliptic curves and their automorphisms}
\begin{tabular}{||c|c|c|c|c|c||}
\hline \hline
&&& &&\\
$\#$ & $Aut(\X)$  & $C$ & $Aut(\X)/C$&  equation & Id. \\
&&&&&  \\
\hline  \hline

&&&&&\\

1& $V_4$ & $ V_4$ & $\{1\}$& $x^4+y^4+ax^2y^2+bx^2+cy^2+1=0$&(4,2) \\

2& $D_8$ & $ V_4$ &  $C_2$&   take\ $b=c$& (8,3)\\

3 & $S_4$ & $ V_4$ &  $S_3$&   take\ $a=b=c$ &(24,12) \\

4& $C_4^2 \xs S_3$ & $ V_4$ & $S_4$&      \ take \, $a=b=c=0$ \, or\, $y^4=x(x^2-1)$ & (96,64) \\

\hline

5 & $16$ & $C_4$& $V_4$&  $y^4=x(x-1)(x-t)$&(16,13)        \\

6& $48$ & $C_4$&  $A_4$&   $y^4=x^3-1$ &(48,33)              \\

\hline

7& $C_3$ & $C_3$& $\{1\}$&  $y^3=x(x-1)(x-s)(x-t)$&(3,1)  \\

8& $C_6$ & $C_3$& $C_2$&   take\ $s=1-t$& (6,2)     \\

9& $C_9$ & $C_3$& $C_3$ & $y^3=x(x^3-1)$&(9,1)\\

\hline &&&&&\\

10& $L_3(2)$ & & & $x^3y+y^3z+z^3x=0$ &(168,42) \\

\hline

& &&&&\\
11& $S_3$ & & &   $a(x^4+y^4+z^4)+b(x^2y^2+x^2z^2+y^2z^2)+$&(6,1)\\

& & &  &$c(x^2yz+y^2xz+z^2xy)=0$&\\
\hline& &&  &&\\

12& $C_2$ & & & $ x^4+x^2(y^2+az^2) + by^4+cy^3z+dy^2z^2$&(2,1)\\
&&&&  $ +eyz^3+gz^4=0$, \  \ either $e=1$ or $g=1$ &\\

\hline \hline
\end{tabular}
\end{table}
\end{center}

Figure ~\ref{figgenus=3} describes the inclusions among all subloci for genus 3 curves. In order to study such
inclusions, the lattice of the list of automorphism groups of genus 3 curves needs to be determined. Let's consider the
locus of the hyperelliptic curve whose automorphism group is $V_4 = \{1,\alpha,\beta,\alpha \beta\}.$ Suppose $\alpha$
is the hyperelliptic involution. Since the hyperelliptic involution is unique, the genus of the quotient curve
$\X^{\langle \beta \rangle}$ is 1. Also we have $\langle \alpha \rangle \cong C_2 \hookrightarrow V_4$ and $\langle
\beta \rangle \cong C_2 \hookrightarrow V_4.$ Therefore the locus of the hyperelliptic curve with automorphism group
$V_4$ can be embedded into two different loci with automorphism group $C_2.$ One comes from a curve that has
hyperelliptic involution and the other comes from a curve which does not have hyperelliptic involution. Similarly we
can describe the inclusions of each locus.  The lattice of the automorphism groups for genus 3 curves is given Figure
1.
\subsection{Theta Functions for Hyperelliptic Curves}  For genus three hyperelliptic curves, we have 28
odd theta characteristics and 36 even theta characteristics. The following shows the corresponding characteristics for
each theta function. The first 36 are for the even functions and the last 28 are for the odd functions. For simplicity,
we denote them by $\T_i(z)$ instead of $\T_i \ch {a} {b} (z , \t)$ where $i=1,\dots ,36$ for the even functions and
$i=37, \dots, 64$ for the odd functions.
\[
\begin{split}
\T_1(z) &= \T_1 \chs {0}{0}{0}{0} 0 0 (z , \t), \qquad \qquad
\T_2(z) = \T_2 \chs {\frac{1}{2}} 0 {\frac{1}{2}} {\frac{1}{2}} {\frac{1}{2}} {\frac{1}{2}} (z , \t)\\
\T_3(z) &= \T_3 \chs {\frac{1}{2}} {\frac{1}{2}} {\frac{1}{2}} 0 0 0 (z , \t), \qquad \qquad
\T_4(z) = \T_4 \chs 0 0 0 {\frac{1}{2}} 0 0  (z , \t)\\
\T_5(z) &= \T_5 \chs {\frac{1}{2}} 0 0 0 {\frac{1}{2}} 0 (z , \t), \qquad \qquad
\T_6(z) = \T_6 \chs {\frac{1}{2}} {\frac{1}{2}} 0 0 0 {\frac{1}{2}} (z , \t)\\
\T_7(z) &= \T_7 \chs 0 {\frac{1}{2}} {\frac{1}{2}} {\frac{1}{2}} 0 0  (z , \t), \qquad \qquad
\T_8(z) = \T_8 \chs 0 0 {\frac{1}{2}} 0 {\frac{1}{2}} 0 (z , \t)\\
\T_9(z) &= \T_9 \chs 0 0 0 0 0 {\frac{1}{2}}(z , \t),\qquad \qquad
\T_{10}(z) = \T_{10} \chs {\frac{1}{2}} 0 0 0 0 0   (z , \t)\\
\T_{11}(z) &= \T_{11} \chs {\frac{1}{2}} {\frac{1}{2}} 0 {\frac{1}{2}} {\frac{1}{2}} 0 (z , \t),\qquad \qquad
\T_{12}(z) = \T_{12} \chs {\frac{1}{2}} {\frac{1}{2}} {\frac{1}{2}} {\frac{1}{2}} 0 {\frac{1}{2}} (z , \t)\\
\T_{13}(z) &= \T_{13} \chs 0 0 0 {\frac{1}{2}} {\frac{1}{2}} 0 (z , \t),\qquad \qquad
\T_{14}(z) = \T_{14} \chs 0 {\frac{1}{2}} 0 0 0 0 (z , \t)\\
\T_{15}(z) &= \T_{15} \chs 0{\frac{1}{2}} {\frac{1}{2}} 0 {\frac{1}{2}} {\frac{1}{2}} (z , \t),\qquad \qquad
\T_{16}(z) = \T_{16} \chs 0 {\frac{1}{2}} 0 {\frac{1}{2}} 0 {\frac{1}{2}} (z , \t)\\
\T_{17}(z) &= \T_{17} \chs 0 0 0 0 {\frac{1}{2}} {\frac{1}{2}} (z , \t),\qquad \qquad
\T_{18}(z) = \T_{18} \chs 0 0 {\frac{1}{2}} 0 0 0 (z , \t)\\
\end{split}
\]

\begin{figure}[h]
\begin{center}
\begin{picture}(485,410)(-65,-50) 
\thicklines
\qbezier(-10,-40)(-10,150)(40,315) 
\put(40,185){\line(0,1){130}} 
\qbezier(40,185)(115,213)(190,241) 
\dashline{4}[1](190,241)(298,190)  
\dashline{4}[1](300,168)(290,110) 
\dashline{4}[1](190,240)(240,110)
\qbezier(40,315)(77,211)(115,110) 
\qbezier(35,165)(44,137)(53,110) 
\qbezier(35,165)(35,100)(35,35) 
\qbezier(35,165)(95,130)(100,35) 
\qbezier(115,93)(133,64)(151,35) 
\dashline{4}[1](159,93)(210,35)
\qbezier(57,93)(104,64)(151,35) 
\dashline{4}[1](159,93)(100,35)
\dashline{4}[1](240,93) (100,35) 
\dashline{4}[1](288,93) (151,35) 
\dashline{4}[1](190,240)(210,35)
\dashline{4}[1](255,35)(288,93)
\dashline{4}[1](288,93)(310,35)
\dashline{4}[1](240,93)(310,35)
\qbezier(35,18)(48,-11)(61,-40) 
\qbezier(100,18)(80,-11)(61,-40) 
\qbezier(100,18)(120,-11)(140,-40) 
\qbezier(90,-40)(120,-11)(151,18) 
\dashline{4}[1](151,18)(140,-40)
\dashline{4}[1](310,18)(140,-40)
\dashline{4}[1](189,-40)(159,93)
\dashline{4}[1](232,-40)(210,18)
\dashline{4}[1](232,-40)(255,18)
\dashline{4}[1](277,-40)(310,18)
\dashline{4}[1](325,-40)(310,18)
%
%
\put(30,175){\small{$ V_4$}}  
%
\put(30,325){\small{ $C_2$}} 

\put(185,250){\small{ $C_2$}} 

\put(295,175){\small{$V_4$}} 

\put(285,100){\small{$D_8$}}

\put(235,100){\small{$S_3$}} 

\put(154,100){\small{$C_3$}}

\put(110,100){\small{$C_4$}} 

\put(48,100){\small{$(C_2)^2$}} 

\put(18,25){\small{$C_2 \times C_2$}} 

\put(95,25){\small{$D_{12}$}} 

\put(133,25){\small{$C_2 \times D_8$}}  

\put(205,25){\small{$C_6$}} 

\put(250,25){\small{$16$}} 

\put(305,25){\small{$S_4$}} 

\put(-15,-50){\small{$C_{14}$}}

\put(56,-50){\small{$G_1$}}

\put(85,-50){\small{$G_2$}}

\put(125,-50){\small{$C_2 \times S_4$}}

\put(184,-50){\small{$C_{9}$}}

\put(227,-50){\small{$48$}}

\put(272,-50){\small{$96$}}

\put(315,-50){\small{$L_3(2)$}}

\put(-40,-50){\small{0}}

\put(-40,25){\small{1}}

\put(-40,100){\small{2}}

\put(-40,175){\small{3}}

\put(-40,250){\small{4}}

\put(-40,325){\small{5}}

\put(-50,355){{Dimension}} \put(-50,340){ of Loci}

\qbezier(205,328)(175,328)(245,328) \put(250,325){{hyperelliptic}}

\dashline{4}[1](199,313)(245,313) \put(250,310){{non hyperelliptic}}
\end{picture}
\caption{Inclusions among the loci for genus $3$ curves with automorphisms.} \label{figgenus=3}
\end{center}
\end{figure}
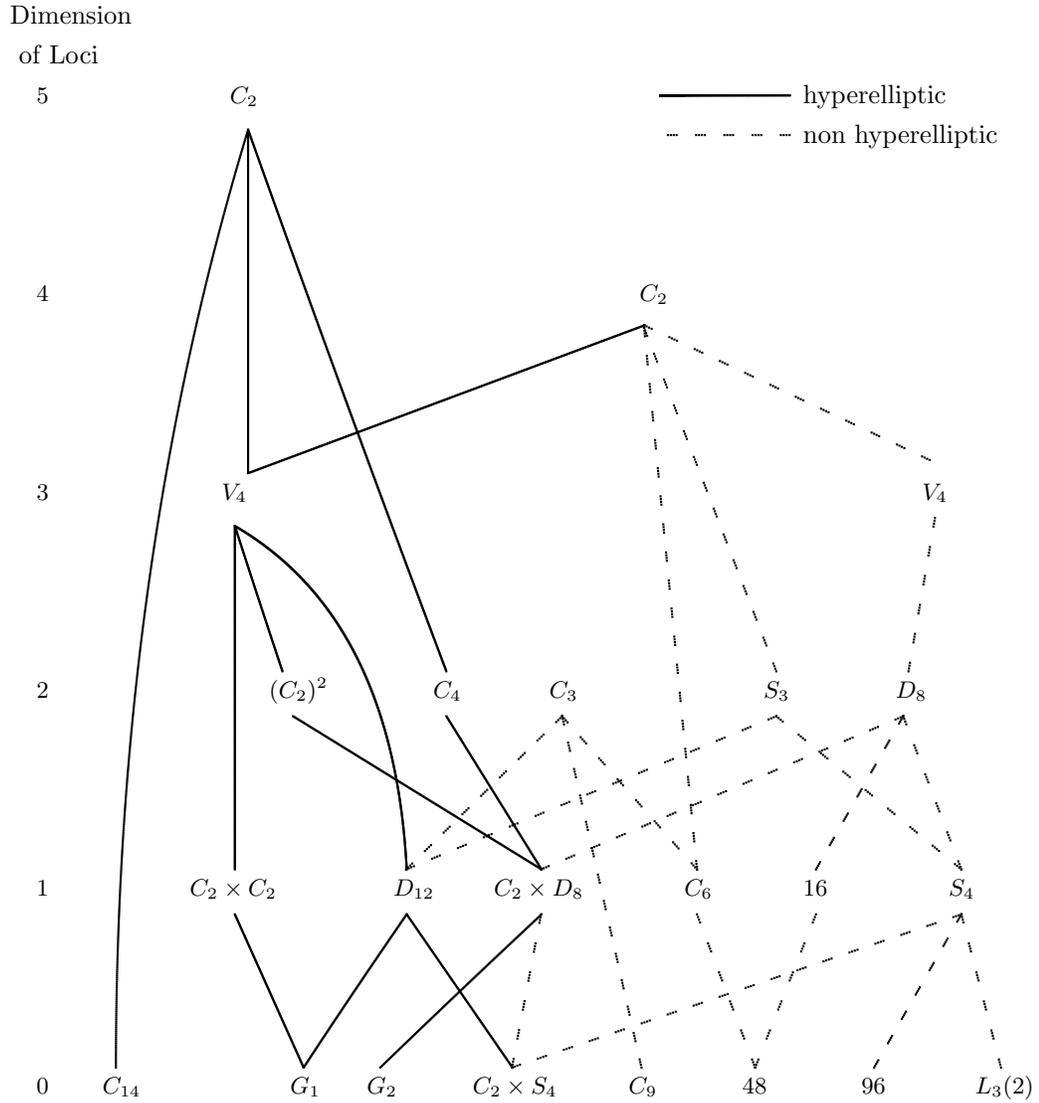

\[
\begin{split}
\T_{19}(z) &= \T_{19} \chs {\frac{1}{2}} {\frac{1}{2}} 0 {\frac{1}{2}} {\frac{1}{2}} {\frac{1}{2}} (z , \t),\qquad
\qquad
\T_{20}(z) = \T_{20} \chs 0 {\frac{1}{2}} 0 0 0 {\frac{1}{2}} (z , \t)\\
\T_{21}(z) &= \T_{21} \chs 0 0 0 0 {\frac{1}{2}} 0 (z , \t),\qquad \qquad
\T_{22}(z) = \T_{22} \chs 0 {\frac{1}{2}} {\frac{1}{2}} 0 0 0 (z , \t)\\
\T_{23}(z) &= \T_{23} \chs {\frac{1}{2}} {\frac{1}{2}} {\frac{1}{2}} {\frac{1}{2}} {\frac{1}{2}} 0 (z , \t),\qquad
\qquad
\T_{24}(z) = \T_{24} \chs {\frac{1}{2}} 0 {\frac{1}{2}} {\frac{1}{2}} 0 {\frac{1}{2}} (z , \t)\\
\end{split}
\]
\[
\begin{split}
\T_{25}(z) &= \T_{25} \chs {\frac{1}{2}} 0 0 0 0 {\frac{1}{2}}(z , \t),\qquad \qquad
\T_{26}(z) = \T_{26} \chs 0 0 0 {\frac{1}{2}} {\frac{1}{2}} {\frac{1}{2}} (z , \t)\\
\T_{27}(z) &= \T_{27} \chs 0 {\frac{1}{2}} 0 {\frac{1}{2}} 0 0 (z , \t),\qquad \qquad
\T_{28}(z) = \T_{28} \chs 0 0 {\frac{1}{2}} {\frac{1}{2}} {\frac{1}{2}} 0 (z , \t)\\
\T_{29}(z) &= \T_{29} \chs {\frac{1}{2}} 0 {\frac{1}{2}} 0 0 0 (z , \t),\qquad \qquad
\T_{30}(z) = \T_{30} \chs {\frac{1}{2}} {\frac{1}{2}} {\frac{1}{2}} 0 {\frac{1}{2}} {\frac{1}{2}} (z , \t)\\
\T_{31}(z) &= \T_{31} \chs{\frac{1}{2}} 0 {\frac{1}{2}} 0 {\frac{1}{2}} 0 (z , \t),\qquad \qquad
\T_{32}(z) = \T_{32} \chs 0 0 {\frac{1}{2}} {\frac{1}{2}} 0 0 (z , \t)\\
\T_{33}(z) &= \T_{33} \chs 0 {\frac{1}{2}} {\frac{1}{2}} {\frac{1}{2}} {\frac{1}{2}} {\frac{1}{2}} (z , \t),\qquad
\qquad
\T_{34}(z) =\T_{34}\chs 0 0 0 {\frac{1}{2}} 0 {\frac{1}{2}} (z , \t) \\
\T_{35}(z) &= \T_{35} \chs {\frac{1}{2}} 0 0 0 {\frac{1}{2}} {\frac{1}{2}} (z , \t),\qquad \qquad
\T_{36}(z) = \T_{36}\chs {\frac{1}{2}} {\frac{1}{2}} 0 0 0 0  (z , \t) \\
\T_{37}(z) &= \T_{37} \chs {\frac{1}{2}} 0 0 {\frac{1}{2}} 0 0  (z , \t),\qquad \qquad
\T_{38}(z) = \T_{38}\chs {\frac{1}{2}} {\frac{1}{2}} 0 0 {\frac{1}{2}} 0 (z , \t) \\
\T_{39}(z) &= \T_{39} \chs {\frac{1}{2}} {\frac{1}{2}} {\frac{1}{2}} 0 0 {\frac{1}{2}} (z , \t),\qquad \qquad
\T_{40}(z) = \T_{40} \chs 0 {\frac{1}{2}} 0 {\frac{1}{2}} {\frac{1}{2}} 0 (z , \t) \\
\T_{41}(z) &= \T_{41} \chs 0 {\frac{1}{2}} {\frac{1}{2}} {\frac{1}{2}} 0 {\frac{1}{2}} (z , \t),\qquad \qquad
\T_{42}(z) = \T_{42}\chs 0 0 {\frac{1}{2}} 0 {\frac{1}{2}} {\frac{1}{2}}(z , \t) \\
\T_{43}(z) &= \T_{43} \chs {\frac{1}{2}} {\frac{1}{2}} {\frac{1}{2}} {\frac{1}{2}} 0 0 (z , \t),\qquad \qquad
\T_{44}(z) =\T_{44} \chs 0 {\frac{1}{2}} {\frac{1}{2}} 0 {\frac{1}{2}} 0 (z ,\t) \\
\T_{45}(z) &= \T_{45} \chs 0 0 {\frac{1}{2}} 0 0 {\frac{1}{2}} (z , \t),\qquad \qquad
\T_{46}(z) = \T_{46} \chs 0 {\frac{1}{2}} 0 0 {\frac{1}{2}} {\frac{1}{2}} (z , \t) \\
\T_{47}(z) &= \T_{47} \chs {\frac{1}{2}} {\frac{1}{2}} 0 {\frac{1}{2}} 0 {\frac{1}{2}} (z , \t),\qquad \qquad
\T_{48}(z) = \T_{48} \chs {\frac{1}{2}} 0 0 {\frac{1}{2}} {\frac{1}{2}} 0 (z , \t)\\
\T_{49}(z) &= \T_{49} \chs{\frac{1}{2}} 0 {\frac{1}{2}} {\frac{1}{2}} {\frac{1}{2}} 0 (z , \t),\qquad \qquad
\T_{50}(z) = \T_{50}\chs {\frac{1}{2}} 0 0 {\frac{1}{2}} 0 {\frac{1}{2}} (z , \t) \\
\T_{51}(z) &= \T_{51} \chs {\frac{1}{2}} {\frac{1}{2}} 0 0 {\frac{1}{2}} {\frac{1}{2}}(z , \t),\qquad \qquad
\T_{52}(z) = \T_{52}\chs 0 0 {\frac{1}{2}} {\frac{1}{2}} {\frac{1}{2}} {\frac{1}{2}} (z , \t) \\
\T_{53}(z) &= \T_{53} \chs 0 {\frac{1}{2}} {\frac{1}{2}} 0 0 {\frac{1}{2}}(z , \t),\qquad \qquad
\T_{54}(z) = \T_{54}\chs 0 {\frac{1}{2}} 0 0 {\frac{1}{2}} 0 (z , \t) \\
\T_{55}(z) &= \T_{55} \chs {\frac{1}{2}} 0 {\frac{1}{2}} 0 0 {\frac{1}{2}} (z , \t),\qquad \qquad
\T_{56}(z) = \T_{56} \chs {\frac{1}{2}} {\frac{1}{2}} {\frac{1}{2}} {\frac{1}{2}} {\frac{1}{2}} {\frac{1}{2}} (z , \t)\\
\T_{57}(z) &= \T_{57} \chs {\frac{1}{2}} {\frac{1}{2}} 0 {\frac{1}{2}} 0 0 (z , \t),\qquad \qquad
\T_{58}(z) = \T_{58} \chs {\frac{1}{2}} {\frac{1}{2}} {\frac{1}{2}} 0 {\frac{1}{2}} 0 (z , \t)\\
\T_{59}(z) &= \T_{59} \chs {\frac{1}{2}} 0 {\frac{1}{2}} {\frac{1}{2}} 0 0 (z , \t),\qquad \qquad
\T_{60}(z) = \T_{60} \chs {\frac{1}{2}} 0 0 {\frac{1}{2}} {\frac{1}{2}} {\frac{1}{2}} (z , \t)\\
\T_{61}(z) &= \T_{61} \chs {\frac{1}{2}} 0 {\frac{1}{2}} 0 {\frac{1}{2}} {\frac{1}{2}} (z , \t),\qquad \qquad
\T_{62}(z) = \T_{62} \chs 0 0 {\frac{1}{2}} {\frac{1}{2}} 0 {\frac{1}{2}}(z , \t)\\
\T_{63}(z) &= \T_{63} \chs 0 {\frac{1}{2}} {\frac{1}{2}} {\frac{1}{2}} {\frac{1}{2}} 0 (z , \t),\qquad \qquad
\T_{64}(z) = \T_{64} \chs 0 {\frac{1}{2}} 0 {\frac{1}{2}} {\frac{1}{2}} {\frac{1}{2}} (z , \t)\\
\end{split}
\]

\begin{rem}
Each half-integer characteristic other than the zero characteristic can be formed as a sum of not more than 3 of the
following seven characteristics:

\[
\begin{split}
& \left\{  \chs {\frac{1}{2}} 0 0 {\frac{1}{2}} 0 0, \chs {\frac{1}{2}} {\frac{1}{2}} 0 0 {\frac{1}{2}} 0, \chs
{\frac{1}{2}} {\frac{1}{2}} {\frac{1}{2}} 0 0 {\frac{1}{2}}, \chs {\frac{1}{2}} 0 {\frac{1}{2}} 0 {\frac{1}{2}}
{\frac{1}{2}}, \chs
0 0 {\frac{1}{2}} {\frac{1}{2}} 0 {\frac{1}{2}},\right.\\
&  \left.  \chs 0 {\frac{1}{2}} {\frac{1}{2}} {\frac{1}{2}} {\frac{1}{2}} 0, \chs 0 {\frac{1}{2}} 0 {\frac{1}{2}}
{\frac{1}{2}} {\frac{1}{2}}\right\}.\\
\end{split}
\]
The sum of all characteristics of the above set gives the zero characteristic. The sums of three characteristics give
the rest of the 35 even characteristics and the sums of two characteristics give 21 odd characteristics.
\end{rem}
It can be shown that one of the even theta constants is zero. Let's pick $S = \{1,2,3,4,5,6,7\}$ and $U = \{1,3,5,7\}.$
Let $T = U.$ Then By Theorem ~\ref{vanishingProperty} the theta constant corresponding to the characteristic $\e_T
=\chs {\frac{1}{2}} {\frac{1}{2}} {\frac{1}{2}} {\frac{1}{2}} 0 {\frac{1}{2}} $ is zero. That is $\T_{12} = 0.$ Next,
we  give the relation between theta characteristics and branch points of the genus 3 hyperelliptic curve in the same
way we did in the genus 2 case. Once again, Thomae's formula is used to get these relations. We get 35 equations with
branch points and non-zero even theta constants. By picking suitable equations, we were able to express branch points
in terms of thetanulls similar to Picard's formula for genus $2$ curves. Let $ B = \{a_1 , a_2 , a_3 , a_4 , a_5 , 1 ,
0\}$ be the finite branch points of the curves and $U = \{a_1, a_3, a_5, 0\}$ be the set of odd branch points.
\begin{thm}
Any genus 3 hyperelliptic curve is isomorphic to a curve given by the equation
\[Y^2=X(X-1)(X-a_1)(X-a_2)(X-a_3)(X-a_4)(X-a_5), \] where
\[ a_1 =\frac{\T_{31}^2\T_{21}^2}{\T_{34}^2\T_{24}^2}, \, \,  a_2=
\frac{\T_{31}^2\T_{13}^2}{\T_{9}^2\T_{24}^2}, \, \,  a_3 = \frac{\T_{11}^2\T_{31}^2}{\T_{24}^2\T_{6}^2}, \, \,   a_4 =
\frac{\T_{21}^2\T_{7}^2}{\T_{15}^2\T_{34}^2}, \, \, a_5= \frac{\T_{13}^2\T_{1}^2}{\T_{26}^2\T_{9}^2}.
\]
\end{thm}

\proof Thomae's formula expresses the thetanulls in terms of branch points of hyperelliptic curves. To invert the
period map we are going to use Lemma~\ref{Thomae}.  For simplicity we order the branch points in the order of $a_1,
a_2, a_3, a_4, a_5, 0, 1,$ and $\infty$. Then the following set of equations represents the relations of theta
constants and $a_1,$ $\dots,$ $a_5.$ We use the notation $(i,j)$ for $(a_i-a_j)$.
%
\[
\begin{split}
{\T_{{1}}}^{4} & = A \,\left(1,6 \right) \left(3,6 \right) \left(5,6 \right) \left( 1 ,3 \right)  \left( 1 ,5 \right)
\left( 3, 5 \right) \left( 2 ,4
\right) \left( 2 ,7 \right)  \left( 4 ,7 \right) \\
{\T_{{2}}}^{4} & =- A \, \left( 3,6 \right) \left( 5,6 \right) \left( 3 ,5 \right)  \left( 1 , 2 \right) \left( 1 ,4
\right) \left(2 ,4 \right) \left( 3 , 7 \right)  \left( 5 ,7 \right) \\
{\T_{{3}}}^{4} & = A \, \left( 3,6 \right)\left( 4,6 \right) \left( 3 , 4 \right) \left( 1, 2 \right) \left( 1 , 5
\right) \left( 2 ,5 \right) \left( 1 ,7 \right) \left( 2,7 \right)  \left( 5 ,7 \right) \\
{\T_{{4}}}^{4} & = -A \, \left( 2,6 \right)\left( 3,6 \right)\left(5,6  \right) \left( 2 ,3 \right) \left( 2 ,5 \right)
\left( 3 ,5 \right) \left(1, 4 \right)  \left( 1 ,7 \right)  \left( 4 ,7 \right) \\
{\T_{{5}}}^{4} & = A \, \left(4,6  \right)\left( 5,6 \right) \left( 4,5 \right) \left( 1 ,2 \right) \left(1, 3 \right)
\left( 2 ,3\right) \left( 1 , 7 \right) \left( 2 ,7 \right)  \left( 3 ,7 \right) \\
{\T_{{6}}}^{4} & = A \, \left( 1,6 \right) \left( 2,6 \right) \left( 3 ,4 \right)  \left( 3 ,5 \right) \left( 4 ,5
\right) \left( 1 ,2
\right) \left( 1 , 7 \right)  \left( 2 ,7 \right) \\
\end{split}
\]
\[
\begin{split}
{\T_{{7}}}^{4} & = A \, \left( 2,6 \right)\left( 3,6 \right)\left( 4,6 \right) \left( 1 ,5 \right) \left( 2 ,3 \right)
\left( 2 ,4 \right)  \left( 3 ,4
\right)  \left( 1 , 7 \right) \left( 5 ,7 \right) \\
{\T_{{8}}}^{4} & = A \, \left( 2,6 \right)\left( 3,6 \right) \left( 2 ,3 \right) \left( 1 , 4 \right) \left( 1 ,5
\right)  \left( 4 ,5
\right)  \left( 1 ,7 \right) \left( 4 , 7 \right) \left( 5 ,7 \right) \\
{\T_{{9}}}^{4} & = -A \, \left( 1,6 \right)\left( 3,6 \right) \left( 1 ,3 \right)  \left( 2 ,4 \right) \left( 2 ,5
\right) \left( 4 ,5
\right) \left( 1, 7 \right) \left( 3 ,7 \right) \\
{\T_{{10}}}^{4} & = -A \,\left( 3,6 \right) \left( 5,6 \right) \left( 3 ,5 \right)  \left( 1 ,2 \right) \left( 1, 4
\right) \left( 2 ,4
\right) \left( 1 ,7 \right)  \left( 2 ,7 \right) \left( 4 , 7 \right) \\
 {\T_{{11}}}^{4} & = -A \, \left(3,6  \right)\left( 4,6 \right) \left( 5,6 \right) \left( 3 ,4 \right)
\left( 3 ,5 \right)  \left( 4 ,5 \right) \left(
1 ,2 \right)  \left( 1, 7 \right)  \left( 2 ,7 \right) \\
{\T_{{13}}}^{4} & = A \, \left( 2,6 \right)\left(  4,6\right)\left(  5,6\right) \left( 1 ,3 \right) \left( 2 ,4 \right)
\left( 2 ,5 \right) \left( 4 ,5
\right) \left( 1 ,7 \right) \left( 3 ,7 \right) \\
{\T_{{14}}}^{4} & = A \, \left( 2,6 \right)\left( 5,6 \right) \left( 2 ,5 \right) \left( 1 ,3 \right) \left( 1 , 4
\right)  \left( 3 ,4
\right)  \left( 1 ,7 \right) \left( 3 ,7 \right) \left( 4 ,7 \right) \\
{\T_{{15}}}^{4} & = -A \, \left( 1,6 \right) \left( 5,6 \right) \left( 1 ,5 \right) \left( 2 ,3 \right) \left( 2 ,4
\right) \left( 3 ,4
\right) \left( 1 , 7 \right) \left( 5 ,7 \right) \\
{\T_{{16}}}^{4} & = A \, \left( 1,6 \right) \left( 2 ,3 \right)  \left( 2 ,4 \right) \left( 2 ,5 \right) \left( 3 ,4
\right) \left( 3 ,5 \right) \left( 4 , 5 \right) \left( 1 ,7 \right)
\\
{\T_{{17}}}^{4} & = A \, \left( 1,6 \right) \left( 4,6 \right) \left( 2 ,3 \right) \left( 2 ,5 \right) \left( 3 ,5
\right) \left( 1 , 4
\right) \left( 1 ,7 \right) \left( 4 ,7 \right) \\
{\T_{{18}}}^{4} & = -A \, \left(2,6  \right) \left( 4,6 \right) \left( 1 ,3 \right) \left( 1 ,5 \right) \left( 3 ,5
\right) \left( 2 ,4
\right) \left( 1 ,7 \right) \left( 3 ,7 \right) \left( 5, 7 \right) \\
{\T_{{19}}}^{4} & = A \, \left( 3,6 \right)\left( 4,6 \right) \left( 1 ,2 \right) \left( 1 ,5 \right) \left( 2 ,5
\right) \left( 3 ,4
\right)  \left( 3 , 7 \right) \left( 4 , 7 \right) \\
{\T_{{20}}}^{4} & = -A \, \left(  2,6\right) \left( 1 ,3 \right) \left( 1 , 4 \right) \left( 1 , 5 \right) \left( 3 ,4
\right) \left( 3 ,5 \right) \left( 4 , 5 \right) \left( 2 , 7 \right)
\\
{\T_{{21}}}^{4} & = -A \, \left( 1,6 \right)\left( 4,6 \right)\left( 5,6 \right) \left( 1 ,4 \right) \left( 1 ,5
\right) \left( 4 ,5 \right) \left( 2 ,3
\right) \left( 2 , 7 \right) \left( 3 ,7 \right) \\
{\T_{{22}}}^{4} & = -A \, \left( 1,6 \right)\left( 3,6 \right)\left( 4,6 \right) \left( 1 , 3 \right) \left( 1 , 4
\right) \left( 3 ,4 \right) \left( 2 ,5
\right)  \left( 2 , 7 \right) \left( 5 , 7 \right) \\
{\T_{{23}}}^{4} & = A \, \left( 1,6 \right)\left( 2,6 \right) \left(3 ,4 \right)  \left( 3 ,5 \right) \left( 4 , 5
\right) \left( 1 ,2
\right)  \left( 3 , 7 \right)  \left( 4, 7 \right) \left( 5, 7 \right) \\
{\T_{{24}}}^{4} & = A \, \left( 4,6 \right)\left( 5,6 \right) \left( 1 ,2 \right) \left( 1 ,3 \right) \left( 2, 3
\right) \left( 4 ,5
\right) \left( 4 , 7 \right) \left( 5 ,7 \right) \\
{\T_{{25}}}^{4} & = A \, \left( 3,6 \right) \left( 1 ,2 \right)  \left( 1 ,4 \right) \left( 1 ,5 \right)
\left( 2 ,4 \right) \left( 2 ,5 \right)  \left( 4 ,5 \right) \left( 3 ,7 \right) \\
{\T_{{26}}}^{4} & =  -A \, \left( 2,6 \right)\left( 4,6 \right) \left( 1 ,3 \right) \left( 1 ,5 \right) \left( 3 ,5
\right) \left( 2 ,4
\right) \left( 2 ,7 \right) \left( 4 ,7 \right) \\
{\T_{{27}}}^{4} & = -A \, \left( 1,6 \right)\left( 5,6 \right) \left( 1 ,5 \right)  \left( 2 ,3 \right) \left( 2 ,4
\right) \left( 3 ,4
\right)  \left( 2 ,7 \right)  \left( 3 ,7 \right) \left( 4 , 7 \right) \\
{\T_{{28}}}^{4} & = -A \, \left( 1,6 \right)\left( 3,6 \right) \left( 1 ,3 \right) \left( 2 ,4 \right) \left( 2 ,5
\right) \left( 4 ,5
\right)  \left( 2 ,7 \right)  \left( 4 ,7 \right) \left( 5 ,7 \right) \\
{\T_{{29}}}^{4} & = A \, \left( 1,6 \right)\left( 2,6 \right)\left( 4,6 \right) \left( 3 ,5 \right) \left( 1 ,2 \right)
\left( 1 ,4 \right) \left( 2, 4
\right) \left( 3, 7 \right)  \left( 5 ,7 \right) \\
{\T_{{30}}}^{4} & = A \, \left(5,6  \right) \left( 1 ,2 \right) \left( 1,3 \right) \left( 1 ,4 \right) \left(
2, 3 \right) \left( 2 ,4 \right) \left( 3,4 \right) \left( 5 , 7 \right) \\
{\T_{{31}}}^{4} & = -A \, \left( 1,6 \right)\left( 2,6 \right)\left(  3,6 \right) \left( 1, 2 \right) \left( 1 ,3
\right) \left( 2 ,3 \right) \left( 4 ,5
\right) \left( 4 ,7 \right) \left( 5 ,7 \right) \\
{\T_{{32}}}^{4} & = A \, \left( 1,6 \right)\left( 4,6 \right) \left( 2 ,3 \right) \left( 2, 5 \right) \left( 3 ,5
\right) \left( 1 , 4
\right) \left( 2 , 7 \right) \left( 3, 7 \right) \left( 5 , 7 \right) \\
{\T_{{33}}}^{4} & = A  \, \left( 2,6 \right)\left( 5,6 \right) \left( 1 ,3 \right) \left( 1 ,4 \right) \left( 3, 4
\right) \left( 2 ,5
\right) \left( 2 , 7 \right) \left( 5 ,7 \right) \\
{\T_{{34}}}^{4} & = A \, \left(  2,6\right)\left( 3,6 \right) \left( 1 ,4 \right)  \left( 1 , 5 \right) \left( 4, 5
\right) \left( 2 ,3
\right) \left( 2 ,7 \right) \left( 3 , 7 \right) \\
{\T_{{35}}}^{4} & = -A \, \left( 4,6 \right) \left( 1 ,2 \right) \left( 1 ,3 \right) \left( 1, 5 \right) \left( 2, 3
\right) \left( 2 ,5 \right) \left( 3, 5 \right) \left( 4 ,7 \right)
\\
{\T_{{36}}}^{4} & = -A \, \left(  1,6\right)\left( 2,6 \right)\left( 5,6 \right) \left( 1 ,2 \right) \left( 1 ,5
\right) \left( 2 ,5 \right) \left( 3 ,4
\right) \left( 3, 7 \right) \left( 4 ,7 \right) \\
\end{split}
\]

\noindent Our expectation is to write down the branch points as quotients of thetanulls. By using the set of equations
given above we have several choices for $a_1,\dots,a_5$ in terms of theta constants.
\begin{center}
\begin{tabular}{c c c c}
Branch Points &  \multicolumn {3}{c} {Possible Ratios}\\[5pt]
$a_1^2$ & $\left(\frac{\T_{36}^2\T_{22}^2}{\T_{33}^2\T_{19}^2}\right)^2$, &
$\left(\frac{\T_{31}^2\T_{21}^2}{\T_{34}^2\T_{24}^2}\right)^2$, &
$\left(\frac{\T_{29}^2\T_{1}^2}{\T_{26}^2\T_{2}^2}\right)^2$ \\
%
 $a_2^2$ & $\left(\frac{\T_{4}^2\T_{29}^2}{\T_{2}^2\T_{17}^2}\right)^2$, &
 $\left(\frac{\T_{36}^2\T_{7}^2}{\T_{15}^2\T_{19}^2}\right)^2$,
&
$\left(\frac{\T_{31}^2\T_{13}^2}{\T_{9}^2\T_{24}^2}\right)^2$ \\
$a_3^2$ & $\left(\frac{\T_{4}^2\T_{22}^2}{\T_{33}^2\T_{17}^2}\right)^2$, &
$\left(\frac{\T_{11}^2\T_{31}^2}{\T_{24}^2\T_{6}^2}\right)^2$, &
$\left(\frac{\T_{7}^2\T_{1}^2}{\T_{26}^2\T_{15}^2}\right)^2$ \\
$a_4^2$ & $\left(\frac{\T_{11}^2\T_{29}^2}{\T_{2}^2\T_{6}^2}\right)^2$, &
$\left(\frac{\T_{21}^2\T_{7}^2}{\T_{15}^2\T_{34}^2}\right)^2$, &
$\left(\frac{\T_{22}^2\T_{13}^2}{\T_{9}^2\T_{33}^2}\right)^2$ \\
$a_5^2$ & $\left(\frac{\T_{4}^2\T_{21}^2}{\T_{34}^2\T_{17}^2}\right)^2$, &
$\left(\frac{\T_{11}^2\T_{36}^2}{\T_{19}^2\T_{6}^2}\right)^2$, &
$\left(\frac{\T_{13}^2\T_{1}^2}{\T_{26}^2\T_{9}^2}\right)^2$ \\
\end{tabular}
\end{center}
%
Let's select the following choices for $a_1, \cdots, a_5$:
\[
a_1 =  \frac{\T_{31}^2\T_{21}^2}{\T_{34}^2\T_{24}^2},  \, \,
 a_2=
\frac{\T_{31}^2\T_{13}^2}{\T_{9}^2\T_{24}^2}, \, \,  a_3 =\frac{\T_{11}^2\T_{31}^2}{\T_{24}^2\T_{6}^2},  \, \, a_4 =
\frac{\T_{21}^2\T_{7}^2}{\T_{15}^2\T_{34}^2},  \quad a_5= \frac{\T_{13}^2\T_{1}^2}{\T_{26}^2\T_{9}^2}.
\]
This completes the proof.
\endproof

\begin{rem}
{i)} Unlike the genus 2 case, here only $\T_1,$ $ \T_6,$ $ \T_7,$ $\T_{11},$ $ \T_{15},$ $ \T_{24},$ $ \T_{31}$ are
from the same G\"opel group.

{ii)} For genus 2 case such relations are known as Picard's formulae. The calculations proposed by Gaudry on genus 2
arithmetic on theta function in cryptography is mainly based on Picard's formulae.
\end{rem}

\subsection{Theta Identities for Hyperelliptic Curves}
Similar to the genus 2 case we can find identities that hyperelliptic theta constants are satisfied. We would like to
find a set of identities that contains all possible even theta constants. A G\"opel group, Eq.~\eqref{eq1} and
Eq.~\eqref{eq2} all play a main role in this task. Now consider a G\"opel group for genus 3 curves. Any G\"opel group
$G$ contains $2^3=8$ elements. The number of such G\"opel groups is $135.$ We have $24$ G\"opel groups such that all of
the characteristics of the groups are even. The following is one of the G\"opel groups which has only even
characteristics:
\[
  \begin{split}
  G = & \left\{ c_1 = \chs {0}{0}{0}{0} 0 0, c_2 = \chs {\frac{1}{2}} {\frac{1}{2}} {\frac{1}{2}} 0 0 0,c_3 \chs
{\frac{1}{2}} 0 0 0 0 0, c_4 = \chs 0 {\frac{1}{2}} 0 0 0 0, c_5 = \chs 0 {\frac{1}{2}} {\frac{1}{2}} 0 0 0, \right. \\
& \left.c_6 = \chs {\frac{1}{2}} 0 {\frac{1}{2}} 0 0 0, c_7 = \chs {\frac{1}{2}} {\frac{1}{2}} 0 0 0 0 , c_8 = \chs 0 0
{\frac{1}{2}} 0 0 0 \right\}.
\end{split}
\]

\noindent By picking suitable characteristics $\bn_1,$ $\bn_2,$  and $\bn_3$ we can find the G\"opel systems  for group
$G.$ Let's pick $\bn_1 = \chs 0 0 0 {\frac{1}{2}} {\frac{1}{2}} 0,$ $\bn_2 = \chs 0 0 0 0 {\frac{1}{2}} 0,$ and $\bn_3
= \chs 0 0 0 {\frac{1}{2}} 0 {\frac{1}{2}},$ then the corresponding G\"opel systems are given by the following:

\[
  \begin{split}
  G = & \left\{ \chs {0}{0}{0}{0} 0 0, \chs {\frac{1}{2}} {\frac{1}{2}} {\frac{1}{2}} 0 0 0, \chs
{\frac{1}{2}} 0 0 0 0 0, \chs 0 {\frac{1}{2}} 0 0 0 0, \chs 0 {\frac{1}{2}} {\frac{1}{2}} 0 0 0, \right. \\
& \left.\chs {\frac{1}{2}} 0 {\frac{1}{2}} 0 0 0, \chs {\frac{1}{2}} {\frac{1}{2}} 0 0 0 0 , \chs 0 0 {\frac{1}{2}} 0 0
0 \right\},\\
 \bn_1 G = & \left\{ \chs 0 0 0 {\frac{1}{2}} {\frac{1}{2}} 0, \chs {\frac{1}{2}} {\frac{1}{2}} {\frac{1}{2}} {\frac{1}{2}} {\frac{1}{2}}
 0,\chs {\frac{1}{2}} 0 0 {\frac{1}{2}} {\frac{1}{2}} 0,\chs 0 {\frac{1}{2}} 0 {\frac{1}{2}} {\frac{1}{2}} 0,
 \chs 0 {\frac{1}{2}} {\frac{1}{2}} {\frac{1}{2}} {\frac{1}{2}} 0, \right.\\
  & \left. \chs{\frac{1}{2}} 0 {\frac{1}{2}} {\frac{1}{2}} {\frac{1}{2}}
 0, \chs {\frac{1}{2}} {\frac{1}{2}} 0 {\frac{1}{2}} {\frac{1}{2}} 0, \chs 0 0 {\frac{1}{2}} {\frac{1}{2}} {\frac{1}{2}} 0 \right\},\\
  \bn_2 G = & \left\{ \chs 0 0 0 0 {\frac{1}{2}} 0, \chs {\frac{1}{2}} {\frac{1}{2}} {\frac{1}{2}} 0 {\frac{1}{2}} 0 ,\chs {\frac{1}{2}} 0 0 0 {\frac{1}{2}}
 0, \chs 0 {\frac{1}{2}} 0 0 {\frac{1}{2}} 0, \chs 0 {\frac{1}{2}} {\frac{1}{2}} 0 {\frac{1}{2}} 0,\right.\\
 & \left.\chs{\frac{1}{2}} 0 {\frac{1}{2}} 0 {\frac{1}{2}}
 0,\chs {\frac{1}{2}} {\frac{1}{2}} 0 0 {\frac{1}{2}} 0, \chs 0 0 {\frac{1}{2}} 0 {\frac{1}{2}} 0 \right\},\\
 \bn_3 G = & \left\{\chs 0 0 0 {\frac{1}{2}} 0 {\frac{1}{2}}, \chs {\frac{1}{2}} {\frac{1}{2}} {\frac{1}{2}} {\frac{1}{2}} 0
 {\frac{1}{2}},\chs {\frac{1}{2}} 0 0 {\frac{1}{2}} 0 {\frac{1}{2}}, \chs 0 {\frac{1}{2}} 0 {\frac{1}{2}} 0
 {\frac{1}{2}}, \chs 0 {\frac{1}{2}} {\frac{1}{2}} {\frac{1}{2}} 0 {\frac{1}{2}}, \right.\\
 & \left.\chs {\frac{1}{2}} 0 {\frac{1}{2}} {\frac{1}{2}} 0
 {\frac{1}{2}}, \chs {\frac{1}{2}} {\frac{1}{2}} 0 {\frac{1}{2}} 0 {\frac{1}{2}}, \chs 0 0 {\frac{1}{2}} {\frac{1}{2}} 0 {\frac{1}{2}}\right\},\\
\bn_1 \bn_2 G = & \left\{\chs 0 0 0 {\frac{1}{2}} 0 0, \chs {\frac{1}{2}} {\frac{1}{2}} {\frac{1}{2}} {\frac{1}{2}} 0
0, \chs {\frac{1}{2}} 0 0 {\frac{1}{2}} 0 0, \chs 0 {\frac{1}{2}} 0 {\frac{1}{2}} 0 0, \chs 0 {\frac{1}{2}}
{\frac{1}{2}} {\frac{1}{2}} 0 0 , \right.\\
 & \left.\chs {\frac{1}{2}} 0 {\frac{1}{2}} {\frac{1}{2}} 0 0, \chs {\frac{1}{2}}
{\frac{1}{2}} 0 {\frac{1}{2}} 0 0 , \chs 0 0 {\frac{1}{2}} {\frac{1}{2}} 0 0 \right\},\\
\bn_1 \bn_3 G = & \left\{\chs 0 0 0 0 {\frac{1}{2}} {\frac{1}{2}}, \chs {\frac{1}{2}} {\frac{1}{2}} {\frac{1}{2}} 0
{\frac{1}{2}} {\frac{1}{2}}, \chs {\frac{1}{2}} 0 0 0 {\frac{1}{2}} {\frac{1}{2}}, \chs 0 {\frac{1}{2}} 0 0
{\frac{1}{2}} {\frac{1}{2}}, \chs 0{\frac{1}{2}} {\frac{1}{2}} 0 {\frac{1}{2}} {\frac{1}{2}}, \right.\\
 & \left.\chs {\frac{1}{2}} 0
{\frac{1}{2}} 0 {\frac{1}{2}} {\frac{1}{2}} , \chs {\frac{1}{2}} {\frac{1}{2}} 0 0 {\frac{1}{2}} {\frac{1}{2}}, \chs 0
0 {\frac{1}{2}} 0 {\frac{1}{2}} {\frac{1}{2}}\right\},\\
\bn_2 \bn_3 G = & \left\{ \chs 0 0 0 {\frac{1}{2}} {\frac{1}{2}} {\frac{1}{2}}, \chs {\frac{1}{2}} {\frac{1}{2}}
{\frac{1}{2}} {\frac{1}{2}} {\frac{1}{2}} {\frac{1}{2}}, \chs {\frac{1}{2}} 0 0 {\frac{1}{2}} {\frac{1}{2}}
{\frac{1}{2}}, \chs 0 {\frac{1}{2}} 0 {\frac{1}{2}} {\frac{1}{2}} {\frac{1}{2}}, \chs 0 {\frac{1}{2}} {\frac{1}{2}}
{\frac{1}{2}} {\frac{1}{2}} {\frac{1}{2}},\right.\\
 & \left. \chs {\frac{1}{2}} 0 {\frac{1}{2}} {\frac{1}{2}} {\frac{1}{2}}
{\frac{1}{2}}, \chs {\frac{1}{2}} {\frac{1}{2}} 0 {\frac{1}{2}} {\frac{1}{2}} {\frac{1}{2}}, \chs 0 0 {\frac{1}{2}}
{\frac{1}{2}} {\frac{1}{2}} {\frac{1}{2}}\right\},\\
\bn_1 \bn_2 \bn_3 G = & \left\{ \chs 0 0 0 0 0 {\frac{1}{2}}, \chs {\frac{1}{2}} {\frac{1}{2}} {\frac{1}{2}} 0 0
{\frac{1}{2}}, \chs {\frac{1}{2}} 0 0 0 0 {\frac{1}{2}}, \chs 0 {\frac{1}{2}} 0 0 0 {\frac{1}{2}}, \chs 0 {\frac{1}{2}}
{\frac{1}{2}} 0 0 {\frac{1}{2}}, \right.\\
  & \left.\chs {\frac{1}{2}} 0 {\frac{1}{2}} 0 0 {\frac{1}{2}}, \chs {\frac{1}{2}} {\frac{1}{2}}
0 0 0 {\frac{1}{2}}, \chs 0 0 {\frac{1}{2}} 0 0 {\frac{1}{2}}\right\}.\\
\end{split}
\]

The above G\"opel systems contain all 64 characteristics for genus 3. Except for the G\"opel group, each of the systems
contains 4 odd characteristics and 4 even characteristics. If $\hn$ denotes one of the characteristics from the G\"opel
group other than $\chs {0}{0}{0}{0} 0 0,$ then $|\en \hn| \equiv |\en| \equiv 0 \mod 2$ has $20$ solutions.
\begin{exa}
If $\hn = \chs {\frac{1}{2}} {\frac{1}{2}} {\frac{1}{2}} 0 0 0 $, then all the characteristics of $G$ and all the even
characteristics of the G\"opel systems of $\bn_1 G,$ $\bn_3 G$ and $\bn_1 \bn_3 G$ are the possible characteristics for
$\en.$ There are 20 of them.
\end{exa}
Without loss of generality, take the $10$ possible choices for $\en$ which give rise to different terms in the series
Eq.~\eqref{eq1} and Eq.~\eqref{eq2}. For each $\hn$ in the G\"opel group other than $\chs {0}{0}{0}{0} 0 0,$ we can
choose $\an$ such that $|\an,\hn| + |\hn| \equiv 0 \mod 2.$ Take $\an$ to be respectively $\bn_1,$ $\bn_2,$ $\bn_3,$
$\bn_1 \bn_2,$ $\bn_1 \bn_3,$ $\bn_2 \bn_3,$ and $\bn_1 \bn_2 \bn_3$ to the cases when $\hn$ is equal to the
characteristics $c_2,$ $c_3,$ $c_4,$ $c_5,$ $c_6,$ $c_7,$ and  $c_8$ respectively. By picking  $\an$ and $\hn$ with
these characteristics, we can obtain formulas which express the zero values of all the even theta functions in terms of
8 theta nulls: $\T_1,$ $\T_3,$ $\T_{10},$ $\T_{14},$ $\T_{18},$ $\T_{22},$ $\T_{29},$ $\T_{36}.$
We obtain the following 14 equations. The first set is obtained by using Eq.~\eqref{eq1}; all the computations are done
by using Maple 10,

\[
\begin{split}
& 3\,{\T_{{13}}}^{2}{\T_{{23}}}^{2}-{\T_{{28}}}^{2}{\T_{{11}}}^{2}-{\T_{{34}}}^{2}{\T_{{12}
}}^{2}+{\T_{{35}}}^{2}{\T_{{15}}}^{2}+{\T_{{
24}}}^{2}{\T_{{16}}}^{2}-{\T_{{30}}}^{2}{\T_{{17}}}^{2}={\T_{{3}}}^{2}{\T_{{1}}}^{2}\\
&-{\T_{{22}}}
^{2}{\T_{{10}}}^{2}-{\T_{{29}}}^{2}{\T_{{14}}}^{2}+{\T_{{36}}}^{2}{\T_ {{18}}}^{2},\\ 
&3\,{\T_{{21}}}^{2}{\T_{{5}}}^{2}+{\T_{{20}}}^{2}{\T_{{6}}}^{2}-{\T_{{31}}}^{2}{\T_{{8}}}^
{2}-{\T_{{25}}}^{2}{\T_{{9}}}^{2}+{\T_{{30}}
}^{2}{\T_{{15}}}^{2}-{\T_{{35}}}^{2}{\T_{{17}}}^{2}={\T_{{10}}}^{2}{\T_{{1}}}^{2}\\
& -{\T_{{22}}}
^{2}{\T_{{3}}}^{2}-{\T_{{36}}}^{2}{\T_{{14}}}^{2}+{\T_{{29}}}^{2}{\T_{{18 }}}^{2},\\ 
&3\,{\T_{{34}}}^{2}{\T_{{16}}}^{2}-{\T_{{27}}}^{2}{\T_{{4}}}^{2}+{\T_{{25}}}^{2}{\T_{{6}}}
^{2}+{\T_{{32}}}^{2}{\T_{{7}}}^{2}-{\T_{{20}}}^{2}{\T_{{9}}}^{2}-{\T_{{24}}}^{2}{\T_{{12}}}^{2}={\T_{{14}}}^{2}{\T_{{1}}}^{2}\\
&+{\T_{{29}} }^{2}{\T_{{3}}}^{2}-{\T_{{36}}
}^{2}{\T_{{10}}}^{2}-{\T_{{22}}}^{2}{\T_{{18 }}}^{2},\\ 
&3\,{\T_{{4}}}^{2}{\T_{{32}}}^{2}+{\T_{{31}}}^{2}{\T_{{5}}}^{2}-{\T_{{27}}}^{2}{\T_{{7}}}^
{2}-{\T_{{21}}}^{2}{\T_{{8}}}^{2}+{\T_{{23}}
}^{2}{\T_{{11}}}^{2}-{\T_{{28}}}^{2}{\T_{{13}}}^{2}={\T_{{18}}}^{2}{\T_{{1}}}^{2}\\
&-{\T_{{36}}}
^{2}{\T_{{3}}}^{2}-{\T_{{29}}}^{2}{\T_{{10}}}^{2}+{\T_{{22}}}^{2}{\T_{{14 }}}^{2},\\ 
&3\,{\T_{{17}}}^{2}{\T_{{15}}}^{2}+{\T_{{19}} }^{2}{\T_{{2}}}^{2}-{\T_{{7}}}^{2}{\T_{{4}}}^
{2}-{\T_{{33}}}^{2}{\T_{{26}}}^{2}+{\T_{{32} }}^{2}{\T_{{27}}}^{2}-{\T_{{35}}}^{2}{\T_{{
30}}}^{2}={\T_{{22}}}^{2}{\T_{{1}}}^{2}\\
&+
{\T_{{10}}}^{2}{\T_{{3}}}^{2}-{\T_{{18}}}^{2}{\T_{{14}}}^{2}-{\T_{{36}}}^{2}{\T_{{29}}}^{2},\\ 
&3\,{\T_{{26}}}^{2}{\T_{{2}}}^{2}+{\T_{{8}}}^{2}{\T_{{5}}}^{2}+{\T_{{16}}}^{2}{\T_{{12}}}^{2}-{\T_{{33}}}^{2}{\T_{{19}}}^{2}-{\T_{{31}
}}^{2}{\T_{{21}}}^{2}-{\T_{{34}}}^{2}{\T_{{ 24}}}^{2}={\T_{{29}}}^{2}{\T_{{1}}}^{2}\\
&-{\T_{{14}}}
^{2}{\T_{{3}}}^{2}-{\T_{{18}}}^{2}{\T_{{10}}}^ {2}+{\T_{{36}}}^{2}{\T_{{22}}}^{2},\\ 
&3\,{\T_{{9}}}^{2}{\T_{{6}}}^{2}+{\T_{{33}}}^
{2}{\T_{{2}}}^{2}-{\T_{{13}}}^{2}{\T_{{11}}}^{2}-{\T_{{26}}}^{2}{\T_{{19}}}^{2}-{\T_{{25}
}}^{2}{\T_{{20}}}^{2}+{\T_{{28}}}^{2}{\T_{{ 23}}}^{2}={\T_{{36}}}^{2}{\T_{{1}}}^{2}\\
&+{\T_{{14}}}^{2}{\T_{{10}}}^ {2}-{\T_{{18}}}^{2}{\T_{{3}}}^{2}-{\T_{{29}}}^{2}{\T_{{22}}}^{2}.\\ 
\end{split}
\]
By using Eq.~\eqref{eq2} we have the following set of equations:
\[
\begin{split}
&3\,{\T_{{13}}}^{4}+3\,{\T_{{23}}}^{4}-{\T_{{28}}}^{4}-{\T_{{11}}}^{4}-{\T_{{34}}}^{4}-
{\T_{{12}}}^{4}+{\T_{{35}}}^{4}+{\T_{{15}}}
^{4}+{\T_{{24}}}^{4}+{\T_{{16}}}^{4}-{\T_{{30}}}^{4}\\
&-{\T_{{17}}}^{4}={\T_{{3}}}^{4}+{\T_{{1}}}^{4}-{\T_{{
22}}}^{4}-{\T_{{10}}}^{4}-{\T_{{29}}}^{4}-{\T_{{14}}}^{4}+{\T_{{ 36}}}^{4}+{\T_{{18}}}^{4},\\ 
&3\,{\T_{{21}}}^{4}+3\,{\T_{{5}}}^{4}+{\T_{{20}}}^{4}+{\T_{{6}}}^{4}-{\T_{{31}}}^{4}-{\T
_{{8}}}^{4}-{\T_{{25}}}^{4}-{\T_{{9}}}^{4}
+{\T_{{30}}}^{4}+{\T_{{15}}}^{4}-{\T_{{35}}}^{4}\\
&-{\T_{{17}}}^{4}={\T_{{10}}}^{4}+{\T_{{1}}}^{4}-{\T_{{
22}}}^{4}-{\T_{{3}}}^{4}-{\T_{{36}}}^{4}-{\T_{{14}}}^{4}+{\T_{{29}} }^{4}+{\T_{{18}}}^{4},\\ 
\end{split}
\]
\[
\begin{split}
&3\,{\T_{{34}}}^{4}+3\,{\T_{{16}}}^{4}-{\T_{{27}}}^{4}-{\T_{{4}}}^{4}+{\T_{{25}}}^{4}+{
\T_{{6}}}^{4}+{\T_{{32}}}^{4}+{\T_{{7}}}^{4}-{\T_{{20}}}^{4}-{\T_{{9}}}^{4}
-{\T_{{24}}}^{4}\\
& -{\T_{{12}}}^{4}={\T_{{14}}}^{4}+{\T_{{1}}}^{4}+{\T_{
{29}}}^{4}+{\T_{{3}}}^{4}-{\T_{{36}}}^{4}-{\T_{{10}}}^{4}-{\T_{{22}} }^{4}-{\T_{{18}}}^{4},\\ 
&3\,{\T_{{4}}}^{4}+3\,{\T_{{32}}}^{4}+{\T_{{31}}}^{4}+{\T_{{5}}}^{4}-{\T_{{27}}}^{4}-{\T
_{{7}}}^{4}-{\T_{{21}}}^{4}-{\T_{{8}}}^{4}
+{\T_{{23}}}^{4}+{\T_{{11}}}^{4}-{\T_{{28}}}^{4}\\
&-{\T_{{13}}}^{4}={\T_{{18}}}^{4}+{\T_{{1}}}^{4}-{\T_{{
36}}}^{4}-{\T_{{3}}}^{4}-{\T_{{29}}}^{4}-{\T_{{10}}}^{4}+{\T_{{22}} }^{4}+{\T_{{14}}}^{4},\\ 
&3\,{\T_{{17}}}^{4}+3\,{\T_{{15}}}^{4}+{\T_{ {19}}}^{4}+{\T_{{2}}}^{4}-{\T_{{7}}}^{4}-{\T
_{{4}}}^{4}-{\T_{{33}}}^{4}-{\T_{{26}}}^{4 }+{\T_{{32}}}^{4}+{\T_{{27}}}^{4}-{\T_{{35}
}}^{4}\\
&-{\T_{{30}}}^{4}={\T_{{22}}}^{4}+{\T_{{1}}}^{4}
+{\T_{{10}}}^{4}+{\T_{{3}}}^{4}-{\T_{{18}}}^{4}-{\T_{{14}}}^{4}-{\T_{{36}}}^{4}-{\T_{{29}}}^{4},\\ 
&3\,{\T_{{26}}}^{4}+3\,{\T_{{2}}}^{4}+{\T_{{8}}}^{4}+{\T_{{5}}}^{4}+{\T_{{16}}}^{4}+{\T_{{12}}}^{4}-{\T_{{33}}}^{4}-{\T_{{19}}}^{4
}-{\T_{{31}}}^{4}-{\T_{{21}}}^{4}-{\T_{{34} }}^{4}\\
& -{\T_{{24}}}^{4}={\T_{{29}}}^{4}+{\T_{{1}}}^{4}-{\T_{{ 14}}}^{4}-{\T_{{3}}}^{4}-{\T_{{18}}}^{4}-{\T_
{{10}}}^{4}+{\T_{{36}}}^{4}+{\T_{{22}}}^{4},\\ 
&3\,{\T_{{9}}}^{4}+3\,{\T_{{6}}}^{4}+{\T_{{
33}}}^{4}+{\T_{{2}}}^{4}-{\T_{{13}}}^{4}-{\T_{{11}}}^{4}-{\T_{{26}}}^{4}-{\T_{{19}}}^{
4}-{\T_{{25}}}^{4}-{\T_{{20}}}^{4}+{\T_{{28
}}}^{4}\\
&+{\T_{{23}}}^{4}={\T_{{36}}}^{4}+{\T_{{1}}}^{4}-{\T_{{18}}}^{4}-{\T_{{3}}}^{4}+{\T_{{14}}}^{4}+{\T
_{{10}}}^{4}-{\T_{{29}}}^{4}-{\T_{{22}}}^{4}.
\end{split}
\]
\begin{rem}
Similar to the genus 2 case we can consider all the G\"opel groups and obtain all possible relations among thetanulls
by following the above procedure. It is tedious and quite long so we don't do it here.
\end{rem}
\subsection{Genus 3 Non-Hyperelliptic Cyclic Curves} Using formulas similar to  Thomae's formula for
each family of cyclic curve $y^n=f(x),$ one can express the roots of $f(x)$ in terms of ratios of theta functions as in
the hyperelliptic case. In this section we study such curves for $g=3$. We only consider the families of curves with
positive dimensions since the curves which belong to 0-dimensional families are well known.
 Notice that the definition of thetanulls is different in this
part from the definitions of thetanulls in the hyperelliptic case. We define the following three theta constants:
\[
\T_1 = \T\chs{0}{\frac{1}{6}}{0}{\frac{2}{3}}{\frac{1}{6}}{\frac{2}{3}}, \quad \T_2 = \T
\chs{0}{\frac{1}{6}}{0}{\frac{1}{3}}{\frac{1}{6}}{\frac{1}{3}}, \quad \T_3 = \T
\chs{0}{\frac{1}{6}}{0}{0}{\frac{1}{6}}{0}.
\]
Next we consider the cases 7, 8 and 5 from Table 3.2. 

%
\noindent \textbf{Case 7:} If the group is $C_3$, then the equation of $\X$ is given by $$y^3=x(x-1)(x-s)(x-t).$$ Let
$Q_i$ where $i= 1..5$ be ramifying points in the fiber of $0,1,s,t,\infty$ respectively. Consider the meromorpic
function $f = x$ on $\X$ of order 3. Then we have $(f) = 3 Q_1 - 3 Q_5.$ By applying the Lemma~\ref{Shiga} with $P_0 =
Q_5$ and an effective divisor $2Q_2 + Q_3,$ we have the following:

\begin{equation} \label{Shiga1}
E s = \prod_{k=1}^3  \frac{\T( 2\int_{Q_5}^{Q_2} \om  + \int_{Q_5}^{Q_3} \om - \int_{Q_5}^{b_k} \om - \triangle , \tau
)} {\T( 2\int_{Q_5}^{Q_2} \om  + \int_{Q_5}^{Q_3} \om  - \triangle , \tau )}.
\end{equation}
Once again, we apply Lemma~\ref{Shiga} with an effective divisor $Q_2 + 2Q_3$ and we have the following:

\begin{equation} \label{Shiga2}
E s^2 = \prod_{k=1}^3  \frac{\T( \int_{Q_5}^{Q_2} \om  + 2\int_{Q_5}^{Q_3} \om - \int_{Q_5}^{b_k} \om - \triangle ,
\tau )} {\T( \int_{Q_5}^{Q_2} \om  + 2\int_{Q_5}^{Q_3} \om  - \triangle , \tau )}.
\end{equation}

By dividing Eq.~\eqref{Shiga2} by Eq.~\eqref{Shiga1} we have
\begin{equation} \label{s}
\begin{split}
s =&  \prod_{k=1}^3  \frac{\T( \int_{Q_5}^{Q_2} \om  + 2\int_{Q_5}^{Q_3} \om - \int_{Q_5}^{b_k} \om - \triangle
, \tau )} {\T( \int_{Q_5}^{Q_2} \om  + 2\int_{Q_5}^{Q_3} \om  - \triangle , \tau )}  \\
   & \times \prod_{k=1}^3  \frac {\T( 2\int_{Q_5}^{Q_2} \om  + \int_{Q_5}^{Q_3} \om  - \triangle , \tau )}{\T( 2\int_{Q_5}^{Q_2} \om
   + \int_{Q_5}^{Q_3} \om - \int_{Q_5}^{b_k} \om - \triangle
, \tau )}.
\end{split}
\end{equation}
By a similar argument, we have
\begin{equation} \label{t}
\begin{split}
t =&  \prod_{k=1}^3  \frac{\T( \int_{Q_5}^{Q_2} \om  + 2\int_{Q_5}^{Q_4} \om - \int_{Q_5}^{b_k} \om - \triangle
, \tau )} {\T( \int_{Q_5}^{Q_2} \om  + 2\int_{Q_5}^{Q_4} \om  - \triangle , \tau )}  \\
   & \times \prod_{k=1}^3  \frac {\T( 2\int_{Q_5}^{Q_2} \om  + \int_{Q_5}^{Q_4} \om  - \triangle , \tau )}{\T( 2\int_{Q_5}^{Q_2} \om
   + \int_{Q_5}^{Q_4} \om - \int_{Q_5}^{b_k} \om - \triangle , \tau )}.
\end{split}
\end{equation}
Computing the right hand side of Eq.~\eqref{s} and Eq.~\eqref{t} was one of the main points of \cite{SHI}. As a result
we have $s = \frac{\T_2^3}{\T_1^3}$ and $ r = \frac{\T_3^3}{\T_1^3}.$
%
%
%

\noindent \textbf{Case 8:} If the group is $C_6$, then the equation is $y^3=x(x-1)(x-s)(x-t)$ with $s = 1-t.$ By using
the results from Case 7, we have $\T_2^3 = \T_1^3 - \T_3^3.$

\noindent \textbf{Case 5:} 
If $\Aut (\X)\iso (16, 13)$, then the equation of $\X$ is given by $$y^4=x(x-1)(x-t).$$ This curve has 4 ramifying
points $Q_i$ where $i= 1..4$ in the fiber of $0,1,t,\infty$ respectively. Consider the meromorpic function $f = x$ on
$\X$ of order 4. Then we have $(f) = 4 Q_1 - 4 Q_4.$ By applying Lemma~\ref{Shiga} with $P_0 = Q_4$ and an effective
divisor $2Q_2 + Q_3$, we have the following:

\begin{equation} \label{Shiga3}
E t = \prod_{k=1}^4  \frac{\T( 2\int_{Q_4}^{Q_2} \om  + \int_{Q_4}^{Q_3} \om - \int_{Q_4}^{b_k} \om - \triangle , \tau
)} {\T( 2\int_{Q_4}^{Q_2} \om  + \int_{Q_4}^{Q_3} \om  - \triangle , \tau )}.
\end{equation}
Once again, we apply Lemma~\ref{Shiga} with an effective divisor $Q_2 + 2Q_3$ and we have the following:

\begin{equation} \label{Shiga4}
E t^2 = \prod_{k=1}^4  \frac{\T( \int_{Q_4}^{Q_2} \om  + 2\int_{Q_4}^{Q_3} \om - \int_{Q_4}^{b_k} \om - \triangle ,
\tau )} {\T( \int_{Q_4}^{Q_2} \om  + 2\int_{Q_4}^{Q_3} \om  - \triangle , \tau )}.
\end{equation}
We have the following by dividing Eq.~\eqref{Shiga4} by Eq.~\eqref{Shiga3}:
\begin{equation}\label{GP16}
\begin{split}
t =& \prod_{k=1}^4  \frac{\T( \int_{Q_4}^{Q_2} \om  + 2\int_{Q_4}^{Q_3} \om - \int_{Q_4}^{b_k} \om -
\triangle , \tau )} {\T( \int_{Q_4}^{Q_2} \om  + 2\int_{Q_4}^{Q_3} \om  - \triangle , \tau )}   \\
   & \times \prod_{k=1}^4 \frac{\T( 2\int_{Q_4}^{Q_2} \om  + \int_{Q_4}^{Q_3} \om  - \triangle , \tau )}  {\T( 2\int_{Q_4}^{Q_2} \om
   + \int_{Q_4}^{Q_3} \om - \int_{Q_4}^{b_k} \om - \triangle
, \tau )}.
\end{split}
\end{equation}
In order to compute the explicit formula for $t,$ one has to find the integrals on the right-hand side. Such
computations are long and tedious and we intend to include them in further work.
%
%
\begin{rem}
In case 5 of Table 3, the parameter $t$ is given by $$\T[e]^4 = A (t-1)^4 t^2,$$
where $[e]$ is the theta characteristic corresponding to the partition $(\{1\}, \{2\}, \{3\}, \{4\})$ and $A$ is a
constant; see \cite{NK} for details. However, this is not satisfactory since we would like $t$ as a rational function
in terms of theta constants. The method in \cite{NK} does not lead to another relation among $t$ and the thetanulls,
since the only partition we could take is the above.
\end{rem}

Summarizing all of the above, we have
\begin{thm}
Let $\X$ be a non-hyperelliptic genus 3 curve. The following statements are true:

\begin{description}
\item [i)] If $\Aut (\X) \iso C_3$, then $\X$ is isomorphic to a curve with equation
\[ y^3 = x (x-1) \left(x-\frac{\T_2^3}{\T_1^3}\right) \left(x- \frac{\T_3^3}{\T_1^3}\right).\]

\item [ii)] If $\Aut (\X) \iso C_6$, then $\X$ is isomorphic to a curve with equation
\[ y^3 = x (x-1) \left(x-\frac{\T_2^3}{\T_1^3}\right) \left(x- \frac{\T_3^3}{\T_1^3}\right) \,\,\textit{with}\,\,  \T_2^3 = \T_1^3 -
\T_3^3.\]
\item [iii)] If $\Aut (\X)$ is isomorphic to the group with GAP identity (16, 13), then $\X$ is isomorphic to a curve with equation
\[y^4 = x (x-1)(x-t)\,\, \]
where $t$ is given by Eq.~\eqref{GP16}.
\end{description}
\end{thm}


\section{Genus 4 curves}
In this section we focus on  genus 4 curves.  For the genus 4 curves, the complete set of all possible full
automorphism groups and the corresponding equations are not completely calculated yet. In this chapter we consider a
few of the cyclic curves of genus 4. Let us first consider the genus 4 hyperelliptic algebraic curves.  For these
curves, we have $2^{g-1} (2^g +1) = 136$ even half-integer characteristics and $2^{g-1} (2^g -1) = 120$ odd
half-integer characteristics. Among the even thetanulls, 10 of them are 0. We won't show the exact information here.
Following the same procedure as for $g=3,$  the branch points of genus 4 hyperelliptic curves can be expressed as
ratios of even theta constants and identities among theta constants can be obtained. The following Table ~\ref{tab_4}
gives some genus 4 non-hyperelliptic cyclic curves; see Table 2 of \cite{SH-KM} for the complete list.
%
\begin{center}
\begin{table}[h]\label{tab_4}
\caption{Some genus 4 non hyperelliptic  cyclic curves and their automorphisms} \centering
\begin{tabular}{||c|c|c|c||}
\hline \hline
$\#$ & dim & Aut($\X$) & Equation\\
\hline \hline
%
%
1 & 3 & $C_2$ & $y^3 = x(x-1)(x-a_1)(x-a_2)(x-a_3)$\\
2 & 2 &$C_3 \times C_2$ & $y^3 = (x^2-1)(x^2 - \alpha_1)(x^2- \alpha_2)$\\
3 & 1 & $C_5$ & $y^5 = x(x-1)(x-\alpha)$\\
4 & 1& $C_3 \times C_2 $ & $y^3 = (x^2-1)(x^4 - \alpha x^2 +1)$ \\
\hline \hline
\end{tabular}
\end{table}
\end{center}
%
The Figure ~\ref{figg=4} shows the inclusions of loci of the genus 4 curves.
\subsection{Inverting the Moduli Map} In this section we will express branch points of each cyclic curve
in Table 4.1
as ratios of theta nulls.

\noindent\textbf{Case 1:} $C : y^3 = x (x-1) (x-a_1) (x-a_2) (x-a_3).$ In this curve $\infty$ is a branch point. We can
use  result of \cite{NK} to find out $a_1, a_2, a_3$ in terms of thetanulls. First we need to find the partitions of
the set $\{1,2,3,4,5,6\}.$ The Table 5 shows all possible partitions of $\{1,2,3,4,5,6\}$ into 3 sets and the labeling
of the corresponding thetanulls.

For each partition we can apply the generalized Thomae's formula to obtain an identity. According to this labeling of
theta constants and the generalized Thomae's formula we have the following relations:
\[
\begin{split}
&{\theta_{{1}}}^{6} = c_1 \left( a_{{1}}-a_{{2}} \right) ^{3} \left( a_{{1} }-a_{{3}} \right)  \left( a_{{2}}-a_{{3}}
\right) a_{{1}}a_{{2}}a_{{3} } \left( a_{{1}}-1 \right)  \left( a_{{2}}-1 \right)  \left( a_{{3}}-1
 \right) ^{3},\\
 &{\theta_{{2}}}^{6} = c_2 \left( a_{{1}}-a_{{2}} \right) ^{3}
 \left( a_{{1}}-a_{{3}} \right)  \left( a_{{2}}-a_{{3}} \right) a_{{1}
}a_{{2}}{a_{{3}}}^{3} \left( a_{{1}}-1 \right)  \left( a_{{2}}-1
 \right)  \left( a_{{3}}-1 \right), \\
 &{\theta_{{3}}}^{6} =c_3  \left( a_{{1}} -a_{{2}} \right) ^{3} \left( a_{{1}}-a_{{3}} \right)  \left( a_{{2}}-a _{{3}}
\right) a_{{1}}a_{{2}}a_{{3}} \left(
a_{{1}}-1 \right)  \left( a_{{2}}-1 \right)  \left( a_{{3}}-1 \right), \\
&{\theta_{{4}}}^{6} = c_4 \left( a_{{1}}-a_{{2}} \right)  \left( a_{{1}}-a_{{3}} \right) ^{3} \left( a_{{2}}-a_{{3}}
\right)
a_{{1}}a_{{2}}a_{{3}} \left( a_{{1}}-1 \right)  \left( a_{{2}}-1 \right) ^{3} \left( a_{{3}}-1 \right), \\
&{ \theta_{{5}}}^{6} = c_5 \left( a_{{1}}-a_{{2}} \right)  \left( a_{{1}}-a_{{ 3}} \right) ^{3} \left( a_{{2}}-a_{{3}}
\right) a_{{1}}{a_{{2}}}^{3}a_ {{3}} \left( a_{{1}}-1 \right)  \left( a_{{2}}-1 \right)  \left( a_{{3 }}-1 \right),\\
&{\theta_{{6}}}^{6} = c_6 \left( a_{{1}}-a_{{2}} \right) \left( a_{{1}}-a_{{3}} \right) ^{3} \left( a_{{2}}-a_{{3}}
\right)
a_ {{1}}a_{{2}}a_{{3}} \left( a_{{1}}-1 \right)  \left( a_{{2}}-1 \right)  \left( a_{{3}}-1 \right),\\
  \end{split}
 \]

%
%
\begin{table}[ht]
\caption{Partitions of $\{1,2,3,4,5,6\}$ into 3 sets}
\begin{tabular}{||c|c||}
\hline \hline
Theta constant & Corresponding partition\\
\hline \hline
$\T_1$ & $[1,2],[3,4],[5,6]$\\
$\T_2$ & $[1,2],[3,5],[4,6]$\\
$\T_3$ & $[1,2],[3,6],[4,5]$\\
$\T_4$ & $[1,3],[2,4],[5,6]$\\
$\T_5$ & $[1,3],[2,5],[4,6]$\\
$\T_6$ & $[1,3],[2,6],[4,5]$\\
$\T_7$ & $[1,4],[2,3],[5,6]$\\
$\T_8$ &   $[1,4],[2,5],[3,6]$\\
$\T_9$ &  $[1,4],[2,6],[3,5]$\\
$\T_{10}$ &  $[1,5],[2,3],[4,6]$\\
$\T_{11}$ &  $[1,5],[2 ,4],[3,6]$\\
$\T_{12}$ &  $[1,5],[2,6],[3,4]$\\
$\T_{13}$ &  $[1,6],[2,3],[4,5]$\\
$\T_{14}$ & $[1,6],[2,4],[3,5]$\\
$\T_{15}$ &  $[1,6],[2,5],[3,4]$\\
\hline
\end{tabular}
\end{table}
%
%
\[
\begin{split}
&{\theta_{{7}}}^{6} = c_7 \left( a_{{1}} -a_{{2}} \right)  \left( a_{{1}}-a_{{3}} \right)  \left( a_{{2}}-a_{{3 }}
\right)
^{3}a_{{1}}a_{{2}}a_{{3}} \left( a_{{1}}-1 \right) ^{3} \left( a_{{2}}-1 \right)  \left( a_{{3}}-1 \right),\\
&{\theta_{{8}}}^{6 } = c_8 \left( a_{{1}}-a_{{2}} \right)  \left( a_{{1}}-a_{{3}} \right) \left( a_{{2}}-a_{{3}}
\right)
a_{{1}}{a_{{2}}}^{3}a_{{3}} \left( a_{ {1}}-1 \right) ^{3} \left( a_{{2}}-1 \right)  \left( a_{{3}}-1 \right),\\
&{\theta_{{9}}}^{6} = c_9 \left( a_{{1}}-a_{{2}} \right)  \left( a _{{1}}-a_{{3}} \right)  \left( a_{{2}}-a_{{3}}
\right)
a_{{1}}a_{{2}}{ a_{{3}}}^{3} \left( a_{{1}}-1 \right) ^{3} \left( a_{{2}}-1 \right) \left( a_{{3}}-1 \right),\\
&{\theta_{{10}}}^{6} = c_{10} \left( a_{{1}}-a_{{2}} \right)  \left( a_{{1}}-a_{{3}} \right)  \left( a_{{2}}-a_{{3}}
 \right) ^{3}{a_{{1}}}^{3}a_{{2}}a_{{3}} \left( a_{{1}}-1 \right) \left( a_{{2}}-1 \right)  \left( a_{{3}}-1 \right), \\
 &{\theta_{{11}}}^{6} = c_{11} \left( a_{{1}}-a_{{2}} \right)  \left( a_{{1}}-a_{{3}} \right) \left( a_{{2}}-a_{{3}} \right)
  {a_{{1}}}^{3}a_{{2}}a_{{3}} \left( a_{{1}}-1 \right)  \left( a_{{2}}-1 \right) ^{3} \left( a_{{3}}-1 \right),\\
&{\theta_{{12}}}^{6} =  c_{12} \left( a_{{1}}-a_{{2}} \right)  \left( a_{{1}}-a_{{3}} \right)  \left( a_{{2}}-a_{{3}}
\right) {a_{{1}}}^{3}a _{{2}}a_{{3}} \left( a_{{1}}-1 \right)  \left( a_{{2}}-1 \right) \left( a_{{3}}-1
\right)^{3},\\
&{\theta_{{13}}}^{6} = c_{13} \left( a_{{1}}-a_{ {2}} \right)  \left( a_{{1}}-a_{{3}} \right)  \left( a_{{2}}-a_{{3}}
\right)
^{3}a_{{1}}a_{{2}}a_{{3}} \left( a_{{1}}-1 \right)  \left( a_ {{2}}-1 \right)  \left( a_{{3}}-1 \right),\\
&{\theta_{{14}}}^{6} = c_{14} \left( a_{{1}}-a_{{2}} \right)  \left( a_{{1}}-a_{{3}} \right) \left( a_{{2}}-a_{{3}}
\right)
a_{{1}}a_{{2}}{a_{{3}}}^{3} \left( a_{ {1}}-1 \right)  \left( a_{{2}}-1 \right) ^{3} \left( a_{{3}}-1 \right),\\
&{\theta_{{15}}}^{6} = c_{15} \left( a_{{1}}-a_{{2}} \right)  \left( a_{{1}}-a_{{3}} \right)  \left( a_{{2}}-a_{{3}}
\right) a_{{1}}{a_{{2} }}^{3}a_{{3}} \left( a_{{1}}-1 \right)  \left( a_{{2}}-1 \right) \left( a_{{3}}-1 \right) ^{3}
\end{split}
\]

where $c_i$'s are constants and depend on the partition $\Lambda_i.$ From the above set of equations we can write
$a_1,a_2,a_3$ in terms of theta constants:
\begin{equation}\label{case1_g4}
a_1 ^2 = \delta_1(\frac{\T_{10}}{\T_{13}})^6,  \quad \quad a_2 ^2 = \delta_2(\frac{\T_{5}}{\T_{6}})^6,   \quad \quad
a_3^2 = \delta_3 (\frac{\T_{2}}{\T_{3}})^6
\end{equation}
where $\delta_1 = \frac{c_{10}}{c_{13}}, \quad \quad \delta_2 =\frac{c_{5}}{c_{6}}, \quad \quad \delta_3 =
\frac{c_{2}}{c_{3}}.$

Using the result of case 1 we can write the equations of cases 2 and case 4 in terms of thetanulls.

\begin{figure}
   \begin{center}
        \includegraphics[width=1.2\textwidth]{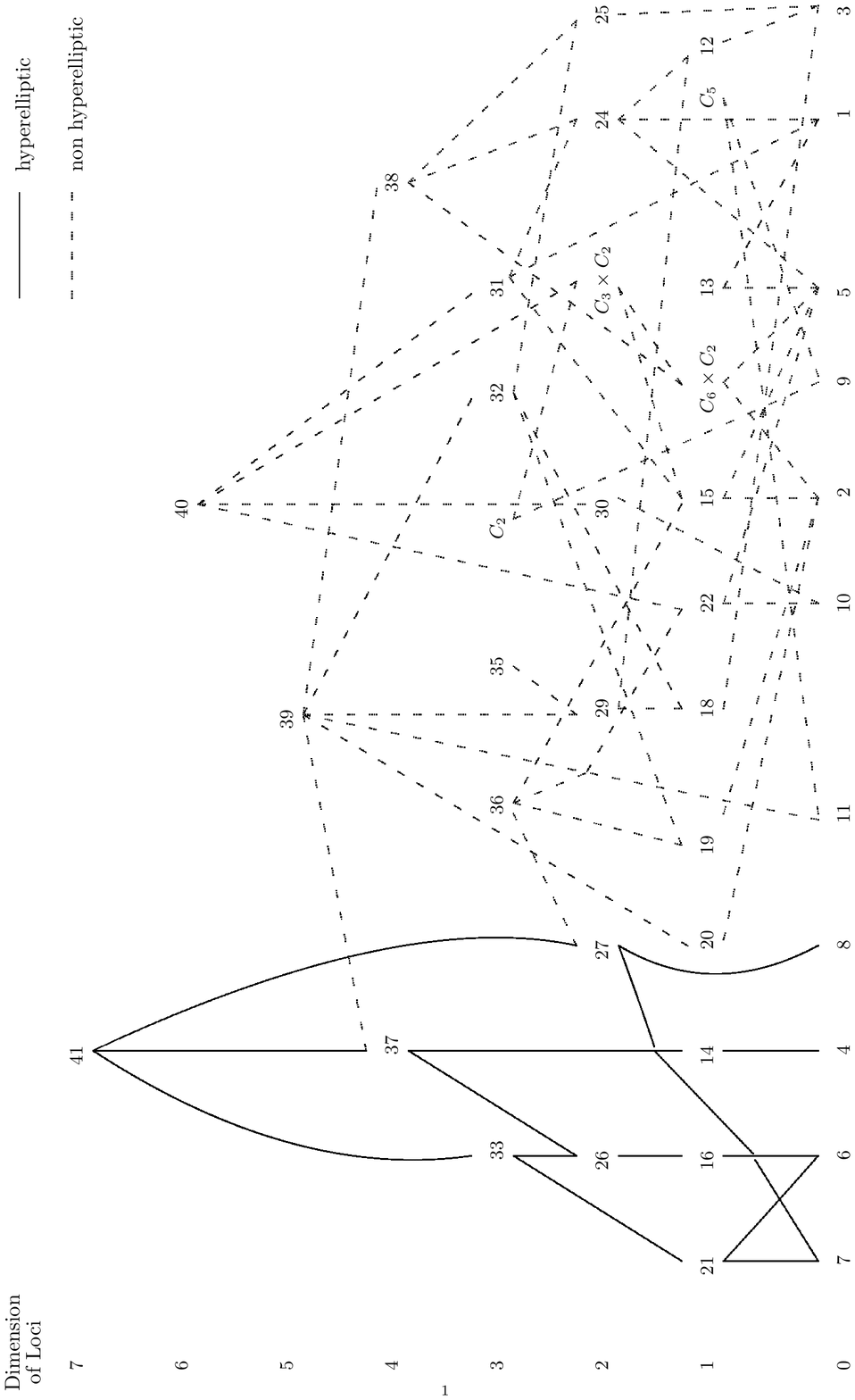}
       \caption{Inclusions among the loci for genus 4 curves.}
      \label{figg=4}
   \end{center}
\end{figure}
\clearpage
%

\noindent\textbf{Case 2:} In this case the curve can be written as
\begin{equation}\label{case3}
y^3 = (x-1)(x+1)(x - \sqrt{\alpha_1})(x + \sqrt{\alpha_1})(x- \sqrt{\alpha_2})(x+ \sqrt{\alpha_2}).
\end{equation}
Consider the transformation given by $$ x \longrightarrow \frac{x-1}{2x-1}.$$
Under this transformation we obtain a curve that is isomorphic to the given curve and

\noindent the new curve is given by the equation $$y^3 = x(x-\frac{2}{3})(x - \gamma_1)(x - \gamma_2)(x - \gamma_3)(x -
\gamma_4)$$
where $\gamma_1 = \frac{\sqrt{\alpha_1} -1 }{2\sqrt{\alpha_1} -1 },$ $\gamma_2 = \frac{-\sqrt{\alpha_1} -1
}{-2\sqrt{\alpha_1} -1 },$  $\gamma_3 = \frac{\sqrt{\alpha_2} -1 }{2\sqrt{\alpha_2} -1 }$, and $\gamma_4 =
\frac{-\sqrt{\alpha_2} -1 }{-2\sqrt{\alpha_2} -1 }.$  Using this transformation we map the branch point 1 of the curve
given by the Eq.~\eqref{case3} to 0.
Again by using the transformation $$x \longrightarrow  \frac {-2x +1 } {3x -2}, $$ we can find another curve isomorphic
to the above two curves. This transformation maps $\frac{2}{3}$ to $\infty.$ With this transformation the curve is
given by the equation
$$y^3 = x(x - \delta_1)(x - \delta_2)(x - \delta_3)(x-\delta_4)$$
where  $\delta_i = \frac{-2\gamma_i + 1}{3 \gamma_i-2}.$ By using the transformation given by  $$x \longrightarrow
\frac {x+1} {\frac{\delta_1}{\delta_1+1}x+ \frac{2\delta_1 +1}{\delta_1+1}},$$
we can find the curve $$y^3 = x(x-1)(x - \beta_1)(x - \beta_2)(x - \beta_3)$$
 where $\beta_i =\frac{(\delta_1+1)(\delta_{i+1}+1)}{\delta_1 \delta_{i+1} +2 \delta_1 +1},$ which is isomorphic to the previous 3
algebraic curves. Now we are in case 1. From the result of case 1, we can write the $\beta_i,$ $i=1,2,3$ as ratios of
thetanulls. But we like to have $\alpha_1$ and $\alpha_2$ as functions of theta constants. Notice that we have the
following 3 relations on $\alpha_1,$ $\alpha_2,$ $\beta_1,$ $\beta_2,$ and $\beta_3$:

\[
\begin{split}
\beta_1 & = \frac{\alpha_1}{\alpha_1 -2 -2(\sqrt{\alpha_1} -1)},\\
\beta_2 & = \frac{\sqrt{\alpha_1 \alpha_2}} {\sqrt{\alpha_1 \alpha_2} + \sqrt{\alpha_1} - \sqrt{\alpha_2}}, \\
\end{split}
\]
\[
\begin{split}
\beta_3 & = \frac{\sqrt{\alpha_1 \alpha_2}} {\sqrt{\alpha_1 \alpha_2} - \sqrt{\alpha_1} - \sqrt{\alpha_2}}. \\
\end{split}
\]
Using these relations, $\alpha_1$ and $\alpha_2$ can be written as rational functions of $\beta_1,$ $\beta_2,$ and
$\beta_3$ given by the following:
\begin{equation}
\begin{split}
{\alpha_1} &= \frac { 2\, \beta_{{1}}\beta_{ {2}}\left( -\beta_{{3}}+\beta_{{2}} \right)
}{2\,\beta_{{1}}\beta_{{3}}+2\,\beta_{{1}}\beta_{{2}}+{\beta_{{2}}
}^{2}\beta_{{3}}-6\,\beta_{{1}}\beta_{{2}}\beta_{{3}}-2\,\beta_{{1}}{
\beta_{{2}}}^{2}+3\,\beta_{{1}}{\beta_{{2}}}^{2}\beta_{{3}}},\\
{\alpha_2}& = {\frac {2\,\beta_{{1}}(\beta_{{3}}-\beta_{{2}})}{-4\,\beta
_{{1}}-\beta_{{2}}\beta_{{3}}+4\,\beta_{{1}}\beta_{{3}}+4\,\beta_{{1}}
\beta_{{2}}-3\,\beta_{{1}}\beta_{{2}}\beta_{{3}}}},
\end{split}
\end{equation}

\noindent with the condition of $\beta_1, \beta_2$ and $\beta_3$
\[
\begin{split}
( \beta_{{1}}{\beta_{{3}}}^{2}+2\,\beta_{{1}}\beta_{{2}
}\beta_{{3}}+\beta_{{1}}{\beta_{{2}}}^{2}+{\beta_{{2}}}^{2}{\beta_{{3}
}}^{2}-4\,\beta_{{1}}\beta_{{2}}{\beta_{{3}}}^{2}-4\,\beta_{{1}}{\beta
_{{2}}}^{2}\beta_{{3}}+3\,\beta_{{1}}{\beta_{{2}}}^{2}{\beta_{{3}}}^{2 } ) \\
( -\beta_{{3}}-\beta_{{2}}+2\,\beta_{{2}}\beta_{{3}}
 ) &=0.
\end{split}
\]
The branch points of the curve given by Eq.~\eqref{case3} can be expressed as ratios of theta constants by using all of
the above information.

\noindent\textbf{Case 4:} In this case the curve is given by
\begin{equation}\label{case1}
y^3 = (x^2-1)(x^4 - \alpha x^2 +1).
\end{equation}
This is a special case of case 2.  By writing out the equation of case 2, we have $y^3 = (x^2-1)(x^4 - (\alpha_1
+\alpha_2) x^2 + \alpha_1 \alpha_2).$ Take $\alpha = \alpha_1 +\alpha_2 $ and $\alpha_1 \alpha_2 = 1.$

\noindent\textbf{Case 3:} In this case, the equation is given by $y^5 = x(x-1)(x-\alpha).$ This curve has 4 ramifying
points $Q_i$ where $i= 1..4$ in the fiber of $0,1,t,\infty$ respectively. The meromorpic function $f = x$ on $\X$ of
order 4 has $(f) = 4 Q_1 - 4 Q_4.$ By applying Lemma~\ref{Shiga} with

\noindent $P_0 = Q_4$ and an effective divisor $4Q_2 + Q_3,$ we have the following:

\begin{equation} \label{case2-1}
E \alpha = \prod_{k=1}^5  \frac{\T( 4\int_{Q_4}^{Q_2} \om  + \int_{Q_4}^{Q_3} \om - \int_{Q_4}^{b_k} \om - \triangle ,
\tau )} {\T( 4\int_{Q_4}^{Q_2} \om  + \int_{Q_4}^{Q_3} \om  - \triangle , \tau )}.
\end{equation}
Again by applying Lemma~\ref{Shiga} with an effective divisor $3Q_2 + 2Q_3,$ we have the following:

\begin{equation} \label{case2-2}
E \alpha^2 = \prod_{k=1}^5  \frac{\T(3 \int_{Q_4}^{Q_2} \om  + 2\int_{Q_4}^{Q_3} \om - \int_{Q_4}^{b_k} \om - \triangle
, \tau )} {\T( 3\int_{Q_4}^{Q_2} \om  + 2\int_{Q_4}^{Q_3} \om  - \triangle , \tau )}.
\end{equation}
We have the following by dividing Eq.~\eqref{case2-2} by Eq.~\eqref{case2-1}:
\begin{equation}\label{case3_g4}
\begin{split}
\alpha =& \prod_{k=1}^5  \frac{\T(3 \int_{Q_4}^{Q_2} \om  + 2\int_{Q_4}^{Q_3} \om - \int_{Q_4}^{b_k} \om -
\triangle , \tau )} {\T( 3\int_{Q_4}^{Q_2} \om  + 2\int_{Q_4}^{Q_3} \om  - \triangle , \tau )}   \\
   & \times \prod_{k=1}^5 \frac{\T( 4\int_{Q_4}^{Q_2} \om  + \int_{Q_4}^{Q_3} \om  - \triangle , \tau )}  {\T( 4\int_{Q_4}^{Q_2} \om  + \int_{Q_4}^{Q_3} \om - \int_{Q_4}^{b_k} \om - \triangle
, \tau )}.
\end{split}
\end{equation}
By calculating integrals on the right-hand side in terms of thetanulls, we can write the branch point $\alpha$ as a
ratio of thetanulls. Summarizing all of the above, we have
\begin{thm}
\begin{description}
\item [i)] If $\Aut(\X) \cong C_3,$ then $\X$ is isomorphic to a curve with equation $$y^3 =
x(x-1)(x-a_1)(x-a_2)(x-a_3),$$ where $a_1$, $a_2,$ and $a_3$  are given in case (1)  in terms of thetanulls.
\item [ii)] If $\Aut(\X) \cong C_3 \times C_2,$ then $\X$ is isomorphic to a curve with equation $$y^3 = (x^2-1)(x^2 - \alpha_1)(x^2-
\alpha_2),$$ where $\a_1,$ and $\a_2$ are given in case (2) in terms of thetanulls.
\item [iii)] If $\Aut(\X) \cong C_5,$ then $\X$ is isomorphic to a curve with equation $$y^5 = x(x-1)(x-\alpha),$$ where
$\a$ is given in case (4) in terms of thetanulls.
\item [iv)] If $\Aut(\X) \cong C_6 \times C_2,$ then $\X$ is isomorphic to a curve with equation $$y^3 = (x^2-1)(x^4 - \alpha x^2
+1),$$ where $\a$ is given in case (3) in terms of thetanulls.
\end{description}
\end{thm}
\section{Concluding Remarks}
In Sections 2, 3, and 4, the main idea was  to write down the branch points as quotients of thetanulls explicitly for
cyclic curves of genus 2, 3, and 4 with extra automorphisms.  For hyperelliptic algebraic curves, we can use Thomae's
formula to express branch points as ratios of thetanulls. We used Maple 10 for all computations. For non-hyperelliptic
cyclic curves, we used various methods in order to invert the period map. The method described in Lemma ~\ref{Shiga} in
Chapter 1 gives the general method to find branch points in terms of thetanulls. The main drawback of this method is
the difficulty of writing complex integrals as functions of theta characteristics. Some of the results in Chapter 2 and
Chapter 3 already appeared in \cite{Previato}.


\end{document}